\documentclass[letter,10pt]{amsart}
\usepackage{amssymb,amsmath,amsfonts}
\usepackage{mathrsfs}
\usepackage[left=1 in, right=1 in,top=1 in, bottom=1 in]{geometry}
\usepackage{mathenv}

\newtheorem{theorem}{Theorem}[section]
\newtheorem{lemma}[theorem]{Lemma}

\newtheorem{definition}[theorem]{Definition}

\makeatletter
\renewcommand \theequation {%
\ifnum \c@section>\z@ \@arabic\c@section.%
\fi\@arabic\c@equation} \@addtoreset{equation}{section}
\makeatother

\providecommand{\ud}[1]{\mathrm{d}{#1}}
\providecommand{\abs}[1]{\left\vert#1\right\vert}
\providecommand{\nm}[1]{\left\Vert#1\right\Vert}

\providecommand{\tm}[2]{\left\Vert#1\right\Vert_{L^2(#2)}}
\providecommand{\im}[2]{\left\Vert#1\right\Vert_{L^{\infty}(#2)}}

\providecommand{\lnnm}[1]{{\left\Vert#1\right\Vert}_{L^{\infty}L^{\infty}}}
\providecommand{\lnm}[1]{\left\Vert#1\right\Vert_{L^{\infty}}}
\providecommand{\tnm}[1]{\left\Vert#1\right\Vert_{L^{2}}}
\providecommand{\tnnm}[1]{{\left\Vert#1\right\Vert}_{L^{2}L^2}}

\def\dt{\partial_t}
\def\p{\partial}

\def\half{\frac{1}{2}}
\def\rt{\rightarrow}
\def\r{\mathbb{R}}
\def\no{\nonumber}
\def\ue{\mathrm{e}}

\def\ds{\displaystyle}

\def\u{U}
\def\ub{\mathscr{U}}
\def\bu{\bar U}
\def\bub{\bar{\mathscr{U}}}

\def\e{\epsilon}
\def\s{\mathbb{S}}

\def\vx{\vec x}
\def\vw{\vec w}

\def\nx{\nabla_{x}}

\def\l{\lambda}
\def\ll{\mathcal{L}}

\def\ff{\mathscr{F}}

\def\q{Q}
\def\qb{\mathscr{Q}}
\def\rk{R_{\kappa}}
\def\rr{\mathscr{R}}

\def\v{\mathscr{V}}
\def\d{\delta}

\def\pp{\mathcal{P}}
\def\vn{\vec\nu}
\def\xc{X_{cl}}
\def\wc{W_{cl}}

\def\ss{f}
\def\g{g}
\def\h{h}
\def\id{\textbf{\textrm{1}}}

\begin{document}
\title{Asymptotic Analysis of Unsteady Neutron Transport Equation}
\author{Lei Wu}
\address{
Department of Mathematical Sciences\\
Carnegie Mellon University \\
Wean Hall 6113, Pittsburgh, PA 15213, USA } \email[L.
Wu]{lwu2@andrew,cmu.edu}

\begin{abstract}
Consider the unsteady neutron transport equation with diffusive boundary condition in 2D convex domains. We establish the diffusive limit with both initial layer and boundary layer corrections. The major difficulty is the lack of regularity in the boundary layer with geometric correction. Our contribution relies on a detailed analysis of asymptotic expansions inspired by the compatibility condition and an intricate $L^{2m}-L^{\infty}$ framework which yields stronger remainder estimates.
\end{abstract}

\maketitle

%


\pagestyle{myheadings} \thispagestyle{plain} \markboth{LEI WU}{ASYMPTOTIC ANALYSIS OF NEUTRON
TRANSPORT EQUATION}

\section{Introduction}

\subsection{Problem Formulation}

We consider the unsteady neutron transport equation in a
two-dimensional smooth convex domain with diffusive boundary. This model describes the motion of neutrons in nuclear reactors, where the particles may be reflected diffusively on the boundary wall. Mathematically, in the time domain $[0,\infty)\ni t$, the space
domain $\Omega\ni\vx=(x_1,x_2)$ where $\p\Omega\in C^3$, and the velocity domain
$\s^1\ni\vw=(w_1,w_2)$, the neutron density $u^{\e}(t,\vx,\vw)$
satisfies
\begin{eqnarray}\label{transport}
\left\{
\begin{array}{l}
\e^2\dt u^{\e}+\e \vw\cdot\nabla_x u^{\e}+u^{\e}-\bar
u^{\e}=0\ \ \ \text{for}\ \
(t,\vx,\vw)\in[0,\infty)\times\Omega\times\s^1,\\\rule{0ex}{1.0em}
u^{\e}(0,\vx,\vw)=h(\vx,\vw)\ \ \text{for}\ \ (\vx,\vw)\in\Omega\times\s^1\\\rule{0ex}{1.0em}
u^{\e}(t,\vx_0,\vw)=\pp[u^{\e}](t,\vx_0)+\e g(t,\vx_0,\vw)\ \ \text{for}\ \ t\in[0,\infty),\ \ \vx_0\in\p\Omega,\ \ \text{and}\ \ \vw\cdot\vn<0,
\end{array}
\right.
\end{eqnarray}
where
\begin{eqnarray}\label{average}
\bar u^{\e}(t,\vx)=\frac{1}{2\pi}\int_{\s^1}u^{\e}(t,\vx,\vw)\ud{\vw},
\end{eqnarray}
the diffusive boundary
\begin{eqnarray}\label{diffusive}
\pp[u^{\e}](t,\vx_0)=\frac{1}{2}\int_{\vw\cdot\vn>0}u^{\e}(t,\vx_0,\vw)(\vw\cdot\vn)\ud{\vw},
\end{eqnarray}
and $\vn$ is the outward unit normal vector, with the Knudsen number $0<\e<<1$. The initial and boundary data satisfy the
compatibility condition
\begin{eqnarray}\label{compatibility condition}
h(\vx_0,\vw)=\pp[h](\vx_0)+\e g(0,\vx_0,\vw)\ \ \text{for}\ \ \vx_0\in\p\Omega\ \ \text{and}\ \ \vw\cdot\vn<0.
\end{eqnarray}
We intend to study the behavior of $u^{\e}$ as $\e\rt0$. Heuristically, the Knudsen number $\e$ represents the scale of mean free path, which measures the average distance a particle can travel between two scattering collisions. When $\e$ shrinks to zero, the collisions occur more and more frequently and the overall behavior of the system is closer and closer to the macroscopic phenomenon.

Based on the flow direction, we can divide the physical boundary
$\Gamma=\{(\vx_0,\vw):\ \vx_0\in\p\Omega,\ \vw\in\s^1\}$ into the in-flow boundary
$\Gamma^-$, the out-flow boundary $\Gamma^+$, and the grazing set
$\Gamma^0$ as
\begin{eqnarray}
\Gamma^{-}&=&\{(\vx_0,\vw):\ \vx_0\in\p\Omega,\ \vw\cdot\vn<0\},\\
\Gamma^{+}&=&\{(\vx_0,\vw):\ \vx_0\in\p\Omega,\ \vw\cdot\vn>0\},\\
\Gamma^{0}&=&\{(\vx_0,\vw):\ \vx_0\in\p\Omega,\ \vw\cdot\vn=0\}.
\end{eqnarray}
It is easy to see $\Gamma=\Gamma^+\cup\Gamma^-\cup\Gamma^0$.

\subsection{Background}

Diffusive limit, or more general hydrodynamic limit, is central to connecting kinetic theory and fluid mechanics. Since early 20th century, this type of problems have been extensively studied in many different settings: steady or unsteady, linear or nonlinear, strong solution or weak solution, etc.

Among all these variations, one of the simplest but most important models - neutron transport equation in bounded domains, where the boundary layer effect shows up, is widely regarded as a prototype of more complicated nonlinear Boltzmann equation, and has attracted a lot of attention since the dawn of atomic age. We refer to the references
\cite{Larsen1974=}, \cite{Larsen1974}, \cite{Larsen1975}, \cite{Larsen1977}, \cite{Larsen.D'Arruda1976}, \cite{Larsen.Habetler1973}, \cite{Larsen.Keller1974}, \cite{Larsen.Zweifel1974}, \cite{Larsen.Zweifel1976}, \cite{Li.Lu.Sun2015=}, \cite{Li.Lu.Sun2015} for more details.

For steady neutron transport equation, the exact solution can be approximated by the sum of an interior solution and a boundary layer. This type of problems
has long been believed to be satisfactorily solved since Bensoussan, Lions and Papanicolaou published their remarkable paper \cite{Bensoussan.Lions.Papanicolaou1979} in 1979. Their formulation was later extended to treat nonlinear Boltzmann equation (see \cite{Sone2002} and \cite{Sone2007}).

Unfortunately, their results are shown to be false due to lack of regularity for the boundary layer equation in \cite{AA003}. A new approach with geometric correction to the boundary layer construction has been developed to ensure regularity in the cases of disk and annulus in \cite{AA003}, \cite{AA006} and \cite{AA004}.

However, this new method fails to treat more general domains. Roughly speaking, we have two contradictory goals to achieve:
\begin{itemize}
\item
To prove diffusive limit, the remainder estimates require higher-order regularity estimate of the boundary layer.
\item
The geometric correction in the boundary layer equation is related to the curvature of the boundary curve, which prevents regularity estimates.
\end{itemize}
In \cite{AA007} and \cite{AA009}, the argument is pushed from both sides. Using delicate estimates along the characteristics in the mild formulation, the authors prove the weighted $W^{1,\infty}$ estimates of the boundary layer. Also, the remainder estimates are improved based on a non-standard energy method and a stationary $L^{2m}-L^{\infty}$ framework. Eventually, the diffusive limit is proved with a non-Hilbert expansion.

As for the unsteady neutron transport equation, things become much more complicated. Traditionally, it is believed that the exact solution can be approximated by the sum of an interior solution, an initial layer, a boundary layer, and further an initial-boundary layer due to the interaction of previous two layers. The construction of the initial-boundary layer relies on the analysis of so-called evolution Milne problem, which is not done even in 1D case.

In \cite{AA005}, based on a detailed analysis of the compatibility condition of the initial data and boundary data, it is shown that the leading-order initial-boundary layer is absent and the diffusive limit is achievable in the cases of disk and annulus. Similar to the steady problems, in more general domains, this approach does not work.

\subsection{Major Difficulties and Methods}

In this paper, we extend the results for unsteady neutron transport equation to treat general 2D convex domains. Basically, the proof relies on an innovative combination of almost all the techniques above and a careful design of asymptotic expansions.
It mainly consists of the following steps:\\
\ \\
Step 1: Interior Solution Expansion.\\
This step is classical and we use Hilbert's expansion to derive the diffusion equation. However, the expansion does not satisfy the initial data and boundary data of $u^{\e}$, so we need initial layer and boundary layer corrections.\\
\ \\
Step 2: Initial Layer Expansion.\\
Here, we utilize the idea in \cite{AA005} to construct the initial layer based on a hierarchy of ordinary differential equations. This step is standard.\\
\ \\
Step 3: Boundary Layer Expansion:\\
This is the core of \cite{AA003} and \cite{AA007}. We abandon the classical expansion based on the flat Milne problem and introduce the $\e$-Milne problem with geometric correction. As \cite{AA007} pointed out, we can show the weighted $W^{1,\infty}$ estimate of the leading-order boundary layer, but it is impossible to obtain higher regularity (like $W^{2,\infty}$ estimates). {\bf This is the key reason to improve the remainder estimates.}\\
\ \\
Step 4: Initial-Boundary Layer Expansion:\\
\cite{AA005} proposed the construction of the initial-boundary layer in the in-flow boundary case. In this paper, we prove that this argument can be recovered in the diffusive boundary case and the leading-order initial-boundary layer is eliminated.\\
\ \\
{\bf Step 1 - Step 4 is relatively standard based on our previous results in \cite{AA003}, \cite{AA006}, \cite{AA005}, \cite{AA007}, \cite{AA008}. Our major contribution focuses on the next step of remainder estimates.}\\
\ \\
Step 5: Improved Remainder Estimates.\\
This step is based on the application of $L^{2m}-L^{\infty}$ framework to time-dependent transport equations
\begin{eqnarray}\label{intemp 1}
\left\{
\begin{array}{l}
\e^2\dt u+\e \vw\cdot\nabla_x u+u-\bar
u=\ss\ \ \ \text{for}\ \
(t,\vx,\vw)\in[0,\infty)\times\Omega\times\s^1,\\\rule{0ex}{1.0em}
u(0,\vx,\vw)=\h(\vx,\vw)\ \ \text{for}\ \ (\vx,\vw)\in\Omega\times\s^1\\\rule{0ex}{1.0em}
u(t,\vx_0,\vw)-\pp[u](t,\vx_0)=\g(t,\vx_0,\vw)\ \ \text{for}\ \ t\in[0,\infty),\ \ \vx_0\in\p\Omega\ \ \text{and}\ \ \vw\cdot\vn<0,
\end{array}
\right.
\end{eqnarray}
The main idea is to introduce a special test function in weak formulation to treat kernel $\bar u$ and non-kernel parts $u-\bar u$ separately, which yields $\nm{u}_{L^2}$ or $\nm{u}_{L^{2m}}$ estimates, and further improve it to $\nm{u}_{L^{\infty}}$ estimate by a modified double Duhamel's principle with a delicate bootstrapping argument. The major difficulty includes:
\begin{itemize}
\item
{\bf Diffusive Boundary:} A direct energy estimate in (\ref{intemp 1}) and the application of Cauchy's inequality imply
\begin{eqnarray}\label{intemp 2}
&&\frac{\e^2}{2}\nm{u(t)}_{L^2(\Omega\times\s^1)}^2+\frac{\e}{2}\nm{(1-\pp)[u]}^2_{L^2([0,t)\times\Gamma^+)}
+\nm{u-\bar
u}_{L^2([0,t)\times\Omega\times\s^1)}^2\\
&\leq&C\bigg(\iint_{[0,t)\times\Omega\times\s^1}\ss u+\frac{\e^2}{2}\nm{\h}_{L^2(\Omega\times\s^1)}^2+\nm{\g}_{L^2([0,t)\times\Gamma^-)}^2
+\e^2\nm{\pp[u]}^2_{L^2([0,t)\times\Gamma^+)}\bigg).\no
\end{eqnarray}
Note that we cannot obtain the estimate of $\nm{\pp[u]}^2_{L^2([0,t)\times\Gamma^+)}$ from weak formulation itself regardless of the test functions. Here, we utilize an intricate grazing estimate in Lemma \ref{well-posedness lemma 2} to bound $\pp[u]$ through the equation (\ref{intemp 1}) in $L^1$ estimates. This is done in Step 3 of the proof of Theorem \ref{LT estimate}.
\item
{\bf Time Derivative:} The $\dt u$ term is harmless in the energy estimate (\ref{intemp 2}), but becomes a big headache in estimating the kernel $\bar u$. Here, the central idea is to choose a special test function related to $\bar u$ to delicately create $\nm{\bar
u}_{L^2([0,t)\times\Omega\times\s^1)}$ and bound it in term of all the other terms. Note that now $\dt u$ term is on the right-hand side of the inequality, not the left-hand side, and has the shape $\nm{\dt\bar u}_{L^2([0,t)\times\Omega\times\s^1)}$.

Here, we utilize an argument based on the temporal difference quotients and locally time-independent test functions to extract the information of $\nm{\dt\bar u}_{L^2(\Omega\times\s^1)}$. This is done in Step 5 of the proof of Theorem \ref{LT estimate}. It is highly non-trivial and we pay the price to lose powers of $\e$. The scenario becomes extremely worse in $L^{2m}$ estimate. We resort to the interpolation estimates and Young's inequality to reduce the power loss. {\bf This is the key reason why the unsteady estimate in Theorem \ref{LN estimate} is weaker than the similar estimates for steady problems in \cite{AA007}.}
\end{itemize}
Due to above difficulties, though the general framework here is similar to that of \cite{AA005} and \cite{AA007}, we have to start from scratch to present the delicate new terms in detail.

\subsection{Main Result}

\begin{theorem}\label{main}
Assume $g(t,\vx_0,\vw)\in C^2([0,\infty)\times\Gamma^-)$ and
$h(\vx,\vw)\in C^2(\Omega\times\s^1)$. Also, there exists $K_0>0$ such that $\ue^{K_0t}g(t,\vx_0,\vw)\in L^{\infty}([0,\infty)\times\Gamma^-)$. Then the unsteady neutron transport
equation (\ref{transport}) has a unique solution
$u^{\e}(t,\vx,\vw)\in L^{\infty}([0,\infty)\times\Omega\times\s^1)$
satisfying that for some $0<K\leq K_0$,
\begin{eqnarray}\label{main theorem 2}
\lim_{\e\rt0}\nm{\ue^{Kt}\Big(u^{\e}-\u_0-\ub_{I,0}-\ub_{B,0}\Big)}_{L^{\infty}([0,\infty)\times\Omega\times\s^1)}=0,
\end{eqnarray}
where $\u_0$ is the interior solution, $\ub_{I,0}$ is the initial layer, and $\ub_{B,0}$ is the boundary layer.
In particular, the interior solution $\u_0(t,\vx)$ satisfies the heat equation with Neumann boundary condition
\begin{eqnarray}
\left\{
\begin{array}{l}
\dt\u_0-\dfrac{1}{2}\Delta_x\u_0=0\ \ \text{for}\ \ (t,\vx)\in[0,\infty)\times\Omega,\\\rule{0ex}{2em}
\u_0(0,\vx)=\displaystyle\frac{1}{2\pi}\int_{\s^1}h(\vx,\vw)\ud{\vw}\
\ \text{for}\ \ \vx\in\Omega,\\\rule{0ex}{2em}
\dfrac{\p\u_0}{\p\vn}(t,\vx_0)=-\dfrac{1}{\pi}\displaystyle
\int_{\vw\cdot\vn<0}g(t,\vx_0,\vw)(\vw\cdot\vn)\ud{\vw}\ \ \text{for}\ \
(t,\vx_0)\in[0,\infty)\times\p\Omega,
\end{array}
\right.
\end{eqnarray}
the initial layer
\begin{eqnarray}
\ub_{I,0}(t,\vx,\vw)=\ue^{-\frac{t}{\e^2}}\left(h(\vx,\vw)-\frac{1}{2\pi}\int_{\s^1}h(\vx,\vw)\ud{\vw}\right),
\end{eqnarray}
and the boundary layer $\ub_{B,0}=0$.
\end{theorem}

\subsection{Notation and Structure}

Throughout this paper, $C>0$ denotes a universal constant which does not depend on the data and
can change from one inequality to another.
When we write $C(z)$, it means a certain positive constant depending
on the quantity $z$.

Our paper is organized as follows: in Section 2, we present the asymptotic analysis of the equation (\ref{transport}) and prove the diffusive limit, i.e. Theorem \ref{main};  in Section 3, we prove the regularity estimates of the $\e$-Milne problem with geometric correction; finally, in Section 4, we prove the estimate of remainder equation, which constitutes the major upshot of this paper.

\section{Asymptotic Analysis}

In this section, we will present the construction of the interior solution, initial layer, and boundary layer. Also, we will show the diffusive limit as $\e\rt0$.

\subsection{Interior Expansion}

We define the interior expansion as follows:
\begin{eqnarray}\label{interior expansion}
\u(t,\vx,\vw)\sim\sum_{k=0}^{2}\e^k\u_k(t,\vx,\vw),
\end{eqnarray}
where $\u_k$ can be defined by comparing the order of $\e$ via
plugging (\ref{interior expansion}) into the equation
(\ref{transport}). Thus, we have
\begin{eqnarray}
\u_0-\bu_0&=&0,\label{expansion temp 1}\\
\u_1-\bu_1&=&-\vw\cdot\nx\u_0,\label{expansion temp 2}\\
\u_2-\bu_2&=&-\dt\u_0-\vw\cdot\nx\u_1.\label{expansion temp 3}
\end{eqnarray}
Plugging (\ref{expansion temp 1}) into (\ref{expansion temp 2}), we
obtain
\begin{eqnarray}
\u_1=\bu_1-\vw\cdot\nx\bu_0.\label{expansion temp 4}
\end{eqnarray}
Plugging (\ref{expansion temp 4}) into (\ref{expansion temp 3}), we
get
\begin{eqnarray}\label{expansion temp 13}
\u_2-\bu_2&=&-\dt\u_0-\vw\cdot\nx(\bu_1-\vw\cdot\nx\bu_0)\\
&=&-\dt\u_0-\vw\cdot\nx\bu_1+w_1^2\p_{x_1x_1}\bu_0+w_2^2\p_{x_2x_2}\bu_0+2w_1w_2\p_{x_1x_2}\bu_0.\no
\end{eqnarray}
Integrating (\ref{expansion temp 13}) over $\vw\in\s^1$ and using the symmetry, we achieve
the final form
\begin{eqnarray}
\dt\bu_0-\frac{1}{2}\Delta_x\bu_0=0,
\end{eqnarray}
which further implies that $\u_0(t,\vx)$ satisfies the equation
\begin{eqnarray}\label{interior 1}
\left\{
\begin{array}{l}
\u_0=\bu_0,\\\rule{0ex}{1em}
\dt\bu_0-\dfrac{1}{2}\Delta_x\bu_0=0.
\end{array}
\right.
\end{eqnarray}
Similarly, we can derive that $\u_k(t,\vx,\vw)$ for $k=1,2$ satisfies
\begin{eqnarray}\label{interior 2}
\left\{
\begin{array}{l}
\u_k=\bu_k-\vw\cdot\nx\u_{k-1},\\\rule{0ex}{1em}
\dt\bu_k-\dfrac{1}{2}\Delta_x\bu_k=0,
\end{array}
\right.
\end{eqnarray}
Note that in order to determine $\u_k$, we need to determine the initial condition and boundary condition.

\subsection{Initial Layer Expansion}

In order to determine the initial condition for $\u_k$, we need to define the initial layer expansion. Hence, we need a substitution:\\
\ \\
Temporal Substitution:\\
We define the stretched variable $\sigma$ by making the
scaling transform for $u^{\e}(t,\vx,\vw)\rt u^{\e}(\sigma,\vx,\vw)$
with $\sigma\in [0,\infty)$ as
\begin{eqnarray}\label{substitution 0}
\sigma&=&\frac{t}{\e^2},
\end{eqnarray}
which implies
\begin{eqnarray}
\frac{\p u^{\e}}{\p t}=\frac{1}{\e^2}\frac{\p u^{\e}}{\p\sigma}.
\end{eqnarray}
In this new variable,
equation (\ref{transport}) can be rewritten as
\begin{eqnarray}\label{initial temp}
\left\{ \begin{array}{l}\displaystyle \p_{\sigma}u^{\e}+\e\vw\cdot\nabla_xu^{\e}+u^{\e}-\bar u^{\e}=0\ \ \ \text{for}\ \
(\sigma,\vx,\vw)\in[0,\infty)\times\Omega\times\s^1,\\\rule{0ex}{1.0em}
u^{\e}(0,\vx,\vw)=h(\vx,\vw)\ \ \ \text{for}\ \
(\vx,\vw)\in\Omega\times\s^1,\\\rule{0ex}{1.0em}
u^{\e}(\sigma,\vx_0,\vw)=\pp[u^{\e}](\sigma,\vx_0)+g(\sigma,\vx_0,\vw)\ \ \text{for}\ \ \sigma\in[0,\infty),\ \ \vx_0\in\p\Omega,\ \ \text{and}\ \
\vw\cdot\vn<0.
\end{array}
\right.
\end{eqnarray}
We define the initial layer expansion as follows:
\begin{eqnarray}\label{initial layer expansion}
\ub^I(\sigma,\vx,\vw)\sim\ub_{0}^I(\sigma,\vx,\vw)+\e\ub^I_{1}(\sigma,\vx,\vw),
\end{eqnarray}
where $\ub^I_{k}$ can be determined by comparing the order of $\e$ via
plugging (\ref{initial layer expansion}) into the equation
(\ref{initial temp}). Thus, we
have
\begin{eqnarray}
\p_{\sigma}\ub_{0}^I+\ub_{0}^I-\bub^I_{0}&=&0,\label{initial expansion 1}\\
\p_{\sigma}\ub^I_{1}+\ub^I_{1}-\bub^I_{1}&=&-\vw\cdot\nabla_x\ub_{0}^I.\label{initial expansion 2}
\end{eqnarray}
Integrate (\ref{initial expansion 1}) over $\vw\in\s^1$, we have
\begin{eqnarray}
\p_{\sigma}\bub_{I,0}=0.
\end{eqnarray}
which further implies
\begin{eqnarray}
\bub^I_{0}(\sigma,\vx)=\bub^I_{0}(0,\vx)\ \ \text{for}\ \ \forall\sigma\in[0,\infty).
\end{eqnarray}
Therefore, from (\ref{initial expansion 1}), we can deduce
\begin{eqnarray}
\ub_{0}^I(\sigma,\vx,\vw)&=&\ue^{-\sigma}\ub_{0}^I(0,\vx,\vw)+\int_0^{\sigma}\bub_{I,0}(s,\vx)\ue^{s-\sigma}\ud{s}\\
&=&\ue^{-\sigma}\ub_{0}^I(0,\vx,\vw)+(1-\ue^{-\sigma})\bub_{0}^I(0,\vx).\no
\end{eqnarray}
This means that we have
\begin{eqnarray}
\left\{
\begin{array}{l}
\p_{\sigma}\bub_{0}^I=0,\\\rule{0ex}{1.0em}
\ub_{0}^I(\sigma,\vx,\vw)=\ue^{-\sigma}\ub_{0}^I(0,\vx,\vw)+(1-\ue^{-\sigma})\bub_{0}^I(0,\vx).
\end{array}
\right.
\end{eqnarray}
Similarly, we can derive that $\ub_{1}^I(\sigma,\vx,\vw)$ satisfies
\begin{eqnarray}
\left\{
\begin{array}{l}
\p_{\sigma}\bub_{1}^I=-\displaystyle\int_{\s^1}\bigg(\vw\cdot\nabla_x\ub_{0}^I\bigg)\ud{\vw},\\\rule{0ex}{1.5em}
\ub_{1}^I(\sigma,\vx,\vw)=\ue^{-\sigma}\ub_{1}^I(0,\vx,\vw)+\displaystyle\int_0^{\sigma}\bigg(\bub_{1}^I-\vw\cdot\nabla_x\ub_{0}^I\bigg)(s,\vx,\vw)\ue^{s-\sigma}\ud{s}.
\end{array}
\right.
\end{eqnarray}

\subsection{Local Coordinate System}

In order to describe the boundary layer effects, we need a local coordinate system in a neighborhood of the boundary.
Assume the Cartesian coordinate system is $\vx=(x_1,x_2)$. Using polar coordinates system $(r,\theta)\in[0,\infty)\times[-\pi,\pi)$ and choosing the pole in $\Omega$, we assume $\p\Omega$ is
\begin{eqnarray}
\left\{
\begin{array}{rcl}
x_1&=&r(\theta)\cos\theta,\\
x_2&=&r(\theta)\sin\theta,
\end{array}
\right.
\end{eqnarray}
where $r(\theta)>0$ is a given function. Our local coordinate system is similar to the polar coordinate
system, but varies to satisfy the specific requirement.

In the domain near the boundary, for each $\theta$, we have the
outward unit normal vector
\begin{eqnarray}
\vn=\left(\frac{r(\theta)\cos\theta+r'(\theta)\sin\theta}{\sqrt{r(\theta)^2+r'(\theta)^2}},\frac{r(\theta)\sin\theta-r'(\theta)\cos\theta}{\sqrt{r(\theta)^2+r'(\theta)^2}}\right).
\end{eqnarray}
We can determine each
point on this normal line by $\theta$ and its distance $\mu$ to
the boundary point $\bigg(r(\theta)\cos\theta,r(\theta)\sin\theta\bigg)$ as follows:
\begin{eqnarray}\label{local}
\left\{
\begin{array}{rcl}
x_1&=&r(\theta)\cos\theta+\mu\dfrac{-r(\theta)\cos\theta-r'(\theta)\sin\theta}{\sqrt{r(\theta)^2+r'(\theta)^2}},\\\rule{0ex}{2.0em}
x_2&=&r(\theta)\sin\theta+\mu\dfrac{-r(\theta)\sin\theta+r'(\theta)\cos\theta}{\sqrt{r(\theta)^2+r'(\theta)^2}},
\end{array}
\right.
\end{eqnarray}
where $r'(\theta)=\dfrac{\ud{r}}{\ud{\theta}}$. It is easy to see that $\mu=0$ denotes the boundary $\p\Omega$ and $\mu>0$ denotes the interior of $\Omega$.

A direct computation using chain rule (see \cite{AA007} for more details) implies
\begin{eqnarray}
\frac{\p\theta}{\p x_1}=\frac{MP}{P^3+Q\mu},&\quad&
\frac{\p\mu}{\p x_1}=-\frac{N}{P},\\
\frac{\p\theta}{\p x_2}=\frac{NP}{P^3+Q\mu},&\quad&
\frac{\p\mu}{\p x_2}=\frac{M}{P},
\end{eqnarray}
where
\begin{eqnarray}
P&=&(r^2+r'^2)^{\frac{1}{2}},\\
Q&=&rr''-r^2-2r'^2,\\
M&=&-r\sin\theta+r'\cos\theta,\\
N&=&r\cos\theta+r'\sin\theta.
\end{eqnarray}
Also, the
Jacobian of the transform $(x_1,x_2)\rightarrow(\mu,\theta)$ is
\begin{eqnarray}
J&=&(r^2+(r')^2)^{1/2}+\mu\frac{rr''-r^2-2r'^2}{(r^2+r'^2)}.
\end{eqnarray}
Note for smooth convex domains, the curvature
\begin{eqnarray}
\kappa(\theta)=\frac{r^2+2r'^2-rr''}{(r^2+r'^2)^{\frac{3}{2}}},
\end{eqnarray}
and radius of curvature
\begin{eqnarray}
\rk(\theta)=\frac{1}{\kappa(\theta)}=\frac{(r^2+r'^2)^{\frac{3}{2}}}{r^2+2r'^2-rr''}.
\end{eqnarray}
In order for the transform is bijective, we require the Jacobian
$J>0$. Then it implies that $0\leq\mu<\rk(\theta)$, which is the maximum
extension of the valid domain for local coordinate system. Since
we will only use this coordinate system for the domain near the
boundary, the above analysis reveals that as long as the largest
curvature of the boundary is strictly positive and finite, which is naturally satisfied in a smooth convex domain, we can take the transform as
valid for area of $0\leq\mu<\min_{\theta}\rk(\theta)=R_{\min}$. For the unit
plate, we have $R_{\kappa}=1$ and the transform is valid for all the
points in the plate except the center.\\
\ \\
We define substitutions as follows:\\
\ \\
Substitution 1: \\
Let $u^{\e}(t,x_1,x_2,\vw)\rt u^{\e}(t,\mu,\theta,\vw)$ with
$(t,\mu,\theta,\vw)\in [0,\infty)\times[0,R_{\min})\times[-\pi,\pi)\times\s^1$ as
\begin{eqnarray}\label{substitution 1}
\left\{
\begin{array}{rcl}
x_1&=&r(\theta)\cos\theta+\mu\dfrac{-r(\theta)\cos\theta-r'(\theta)\sin\theta}{\sqrt{r(\theta)^2+r'(\theta)^2}},\\\rule{0ex}{2.0em}
x_2&=&r(\theta)\sin\theta+\mu\dfrac{-r(\theta)\sin\theta+r'(\theta)\cos\theta}{\sqrt{r(\theta)^2+r'(\theta)^2}},
\end{array}
\right.
\end{eqnarray}
and then the equation (\ref{transport}) is transformed into
\begin{eqnarray}\label{transport 1}
\left\{
\begin{array}{rcl}
&&\displaystyle\e^2\dt u^{\e}+\e\Bigg(w_1\frac{-r\cos\theta-r'\sin\theta}{(r^2+r'^2)^{\frac{1}{2}}}+w_2\frac{-r\sin\theta+r'\cos\theta}{(r^2+r'^2)^{\frac{1}{2}}}\Bigg)\frac{\p u^{\e}}{\p\mu}\\
&&\displaystyle+\e\Bigg(w_1\frac{-r\sin\theta+r'\cos\theta}{(r^2+r'^2)(1-\kappa\mu)}+w_2\frac{r\cos\theta+r'\sin\theta}{(r^2+r'^2)(1-\kappa\mu)}\Bigg)\frac{\p u^{\e}}{\p\theta}+u^{\e}-\bar u^{\e}=0,\\\rule{0ex}{1.5em}
&&u^{\e}(0,\mu,\theta,\vw)=h(\mu,\theta,\vw)\\\rule{0ex}{1.0em}
&&u^{\e}(t,0,\theta,\vw)=\pp[u^{\e}](t,0,\theta)+\e g(t,\theta,\vw)\ \ \text{for}\
\ \vw\cdot\vn<0,
\end{array}
\right.
\end{eqnarray}
where
\begin{eqnarray}
\vw\cdot\vn=w_1\frac{-r\cos\theta-r'\sin\theta}{(r^2+r'^2)^{\frac{1}{2}}}+w_2\frac{-r\sin\theta+r'\cos\theta}{(r^2+r'^2)^{\frac{1}{2}}},
\end{eqnarray}
and
\begin{eqnarray}
\pp[u^{\e}](t,0,\theta)=\frac{1}{2}\int_{\vw\cdot\vn>0}u^{\e}(t,0,\theta,\vw)(\vw\cdot\vn)\ud{\vw}.
\end{eqnarray}
Noting the fact that
\begin{eqnarray}
\left(\frac{M}{P}\right)^2+\left(\frac{N}{P}\right)^2=
\left(\frac{-r\cos\theta-r'\sin\theta}{(r^2+r'^2)^{\frac{1}{2}}}\right)^2+\left(\frac{-r\sin\theta+r'\cos\theta}{(r^2+r'^2)^{\frac{1}{2}}}\right)^2=1,
\end{eqnarray}
we can further simplify (\ref{transport 1}).\\
\ \\
Substitution 2: \\
Let $u^{\e}(t,\mu,\theta,\vw)\rt u^{\e}(t,\mu,\tau,\vw)$ with
$(t,\mu,\tau,\vw)\in [0,\infty)\times[0,R_{\min})\times[-\pi,\pi)\times\s^1$ as
\begin{eqnarray}\label{substitution 2}
\left\{
\begin{array}{rcl}
\sin\tau&=&\dfrac{r\sin\theta-r'\cos\theta}{(r^2+r'^2)^{\frac{1}{2}}},\\\rule{0ex}{2.0em}
\cos\tau&=&\dfrac{r\cos\theta+r'\sin\theta}{(r^2+r'^2)^{\frac{1}{2}}},
\end{array}
\right.
\end{eqnarray}
which implies
\begin{eqnarray}
\frac{\ud{\tau}}{\ud{\theta}}=\kappa(r^2+r'^2)^{\frac{1}{2}}>0,
\end{eqnarray}
and then the equation (\ref{transport}) is transformed into
\begin{eqnarray}\label{transport 2}
\left\{
\begin{array}{l}\displaystyle
\e^2\dt u^{\e}-\e\left(w_1\cos\tau+w_2\sin\tau\right)\frac{\p
u^{\e}}{\p\mu}-\frac{\e}{\rk-\mu}\left(w_1\sin\tau-w_2\cos\tau\right)\frac{\p
u^{\e}}{\p\tau}+u^{\e}-\bar u^{\e}=0,\\\rule{0ex}{1.0em}
u^{\e}(0,\mu,\tau,\vw)=h(\mu,\tau,\vw)\\\rule{0ex}{1.0em}
u^{\e}(t,0,\tau,\vw)=\pp[u^{\e}](t,0,\tau)+\e g(t,\tau,\vw)\ \
\text{for}\ \ w_1\cos\tau+w_2\sin\tau<0,
\end{array}
\right.
\end{eqnarray}
where
\begin{eqnarray}
\pp[u^{\e}](t,0,\tau)=\half\int_{w_1\cos\tau+w_2\sin\tau>0}u^{\e}(t,0,\tau,\vw)(\vw\cdot\vn)\ud{\vw}.
\end{eqnarray}
Since $\tau$ denotes the angle of normal vector, the domain of $\tau$ is the same as $\theta$, i.e. $[-\pi,\pi)$.

\subsection{Boundary Layer Expansion with Geometric Correction}

Using the idea in \cite{AA003} and \cite{AA007}, in order to define boundary layer, we need several more substitutions:\\
\ \\
Substitution 3:\\
We further make the scaling transform for $u^{\e}(t,\mu,\tau,\vw)\rt
u^{\e}(t,\eta,\tau,\vw)$ with $(t,\eta,\tau,\vw)\in
[0,\infty)\times[0,R_{\min}\e^{-1})\times[-\pi,\pi)\times\s^1$ as
\begin{eqnarray}\label{substitution 3}
\eta&=&\dfrac{\mu}{\e},
\end{eqnarray}
which implies
\begin{eqnarray}
\frac{\p u^{\e}}{\p\mu}=\frac{1}{\e}\frac{\p u^{\e}}{\p\eta}.
\end{eqnarray}
Then equation (\ref{transport}) is transformed into
\begin{eqnarray}\label{transport 3}
\left\{\begin{array}{l}\displaystyle
\e^2\dt u^{\e}-\bigg(w_1\cos\tau+w_2\sin\tau\bigg)\frac{\p
u^{\e}}{\p\eta}-\frac{\e}{\rk-\e\eta}\bigg(w_1\sin\tau-w_2\cos\tau\bigg)\frac{\p
u^{\e}}{\p\tau}+u^{\e}-\bar u^{\e}=0,\\\rule{0ex}{1.0em}
u^{\e}(0,\eta,\tau,\vw)=h(\eta,\tau,\vw)\\\rule{0ex}{1.0em}
u^{\e}(t,0,\tau,\vw)=\pp[u^{\e}](t,0,\tau)+\e g(t,\tau,\vw)\ \
\text{for}\ \ w_1\cos\tau+w_2\sin\tau<0,
\end{array}
\right.
\end{eqnarray}
where
\begin{eqnarray}
\pp[u^{\e}](t,0,\tau)=\half\int_{w_1\cos\tau+w_2\sin\tau>0}u^{\e}(t,0,\tau,\vw)(\vw\cdot\vn)\ud{\vw}.
\end{eqnarray}
\ \\
Substitution 4:\\
Define the velocity substitution for $u^{\e}(t,\eta,\tau,\vw)\rt
u^{\e}(t,\eta,\tau,\xi)$ with $(t,\eta,\tau,\xi)\in
[0,\infty)\times[0,R_{\min}\e^{-1})\times[-\pi,\pi)\times[-\pi,\pi)$ as
\begin{eqnarray}\label{substitution 4}
\left\{
\begin{array}{rcl}
w_1&=&-\sin\xi\\
w_2&=&-\cos\xi
\end{array}
\right.
\end{eqnarray}
We have the succinct form
\begin{eqnarray}\label{transport 4}
\left\{\begin{array}{l}\displaystyle \e^2\dt u^{\e}+\sin(\tau+\xi)\frac{\p
u^{\e}}{\p\eta}-\frac{\e}{R_{\kappa}-\e\xi}\cos(\tau+\xi)\frac{\p
u^{\e}}{\p\tau}+u^{\e}-\bar u^{\e}=0,\\\rule{0ex}{1.0em}
u^{\e}(0,\eta,\tau,\xi)=h(\eta,\tau,\xi)\\\rule{0ex}{1.0em}
u^{\e}(t,0,\tau,\xi)=\pp[u^{\e}](t,0,\tau)+\e g(t,\tau,\xi),\ \ \text{for}\ \
\sin(\tau+\xi)>0,
\end{array}
\right.
\end{eqnarray}
where
\begin{eqnarray}
\pp[u^{\e}](t,0,\tau)=-\half\int_{\sin(\tau+\xi)<0}u^{\e}(t,0,\tau,\xi)\sin(\tau+\xi)\ud{\xi}.
\end{eqnarray}
\ \\
Substitution 5:\\
Finally, we make the substitution for $u^{\e}(t,\eta,\tau,\xi)\rt
u^{\e}(t,\eta,\tau,\phi)$ with $(t,\eta,\tau,\phi)\in
[0,\infty)\times[0,R_{\min}\e^{-1})\times[-\pi,\pi)\times[-\pi,\pi)$ as
\begin{eqnarray}\label{substitution 5}
\phi&=&\tau+\xi,
\end{eqnarray}
and achieve the form
\begin{eqnarray}\label{transport temp}
\left\{\begin{array}{l}\displaystyle \e^2\dt u^{\e}+\sin\phi\frac{\p
u^{\e}}{\p\eta}-\frac{\e}{\rk-\e\eta}\cos\phi\bigg(\frac{\p
u^{\e}}{\p\phi}+\frac{\p
u^{\e}}{\p\tau}\bigg)+u^{\e}-\bar u^{\e}=0\\\rule{0ex}{1.0em}
u^{\e}(0,\eta,\tau,\phi)=h(\eta,\tau,\phi)\\\rule{0ex}{1.0em}
u^{\e}(t,0,\tau,\phi)=\pp u^{\e}(t,0,\tau)+\e g(t,\tau,\phi)\ \ \text{for}\ \
\sin\phi>0,
\end{array}
\right.
\end{eqnarray}
where
\begin{eqnarray}
\pp[u^{\e}](t,0,\tau)=-\half\int_{\sin\phi<0}u^{\e}(t,0,\tau,\phi)\sin\phi\ud{\phi}.
\end{eqnarray}
\ \\
We define the boundary layer solution expansion as follows:
\begin{eqnarray}\label{boundary layer expansion}
\ub^B(t,\eta,\tau,\phi)=\ub^B_{0}(t,\eta,\tau,\phi)+\e\ub^B_{1}(t,\eta,\tau,\phi)
\end{eqnarray}
where $\ub^B_k$ can be defined by comparing the order of $\e$ via
plugging (\ref{boundary layer expansion}) into the equation
(\ref{transport temp}). Thus, in a neighborhood of the boundary, we have
\begin{eqnarray}
\sin\phi\frac{\p \ub^B_{0}}{\p\eta}-\frac{\e}{\rk -\e\eta}\cos\phi\frac{\p
\ub^B_{0}}{\p\phi}+\ub^B_{0}-\bub^B_{0}&=&0,\label{expansion temp 6}\\
\sin\phi\frac{\p \ub^B_{1}}{\p\eta}-\frac{\e}{\rk -\e\eta}\cos\phi\frac{\p
\ub^B_{1}}{\p\phi}+\ub^B_{1}-\bub^B_{1}&=&\frac{1}{\rk -\e\eta}\cos\phi\frac{\p
\ub^B_{0}}{\p\tau},\label{expansion temp 7}
\end{eqnarray}
where
\begin{eqnarray}
\bub^B_{k}(t,\eta,\tau)=\frac{1}{2\pi}\int_{-\pi}^{\pi}\ub^B_{k}(t,\eta,\tau,\phi)\ud{\phi}.
\end{eqnarray}

\subsection{Initial-Boundary Layer Expansion}

Above construction of initial layer and boundary layer yields an interesting fact that at the corner point $(t,\vx)=(0,\vx_0)$ for $\vx_0\in\p\Omega$, the initial layer starting from this point has a contribution on the boundary data, and the boundary layer starting from this point has a contribution on the initial data. Therefore, we have to find some additional functions to compensate for these effects.

The classical theory of asymptotic analysis requires the so-called initial-boundary layer, where the temporal scaling and spacial rescaling should be used simultaneously.


Fortunately, the initial and boundary data satisfy the
compatibility condition
\begin{eqnarray}
h(\vx_0,\vw)=\pp[h](\vx_0)+\e g(0,\vx_0,\vw)\ \ \text{for}\ \ \vx_0\in\p\Omega\ \ \text{and}\ \ \vw\cdot\vn<0.
\end{eqnarray}
Since $0<\e<<1$ can be arbitrary, comparing the order of $\e$, we must have
\begin{eqnarray}
h(\vx_0,\vw)&=&\pp[h](\vx_0),\\
g(0,\vx_0,\vw)&=&0,
\end{eqnarray}
for $\vx_0\in\p\Omega$ and $\vw\cdot\vn<0$, which further implies that $u^{\e}(0,\vx_0,\vw)=h(\vx_0,\vw)=C(\vx_0)$ is a constant that only depends on $\vx_0$.

On the other hand,  in the half-space $\vw\cdot\vn<0$ at $(0,\vx_0,\vw)$, the equation
\begin{eqnarray}
\e^2\dt u^{\e}+\e \vw\cdot\nabla_x u^{\e}+u^{\e}-\bar
u^{\e}&=&0,
\end{eqnarray}
is valid, which implies
\begin{eqnarray}
\e^3\dt g(0,\vx_0,\vw)+\e \vw\cdot\nabla_xh(\vx_0,\vw)+h(\vx_0,\vw)-\bar
h(\vx_0)&=&0.
\end{eqnarray}
Since $0<\e<<1$ can be arbitrary, we must have for $\vw\cdot\vec n<0$,
\begin{eqnarray}
\dt g(0,\vx_0,\vw)&=&0,\\
\vw\cdot\nabla_xh(\vx_0,\vw)&=&0,\\
h(\vx_0,\vw)-\bar
h(\vx_0)&=&0.
\end{eqnarray}
The above relations imply the improved compatibility condition
\begin{eqnarray}\label{improved compatibility condition}
h(\vx_0,\vw)=\pp[h](\vx_0)=\bar h(\vx_0),\ \ \text{and}\ \ \ g(0,\vx_0,\vw)=0,\ \text{for}\ \ \vw\cdot\vec n<0.
\end{eqnarray}
This fact is of great importance in the following analysis.

Then it is easy to check that the leading-order boundary layer must be zero, i.e.
\begin{eqnarray}
\ub^B_{0}(0,\eta,\tau,\phi)&=&0.
\end{eqnarray}
The leading-order initial layer is zero for $\vw\cdot\vn<0$, but is not necessarily zero for $\vw\cdot\vn>0$. However, above improved compatibility condition implies that
\begin{eqnarray}
\ub_0^I(\sigma,\vx_0,\vw)=\pp[\ub_0^I](\sigma,\vx_0),
\end{eqnarray}
which means that it has no effect on the remainder. Therefore, the leading-order initial-boundary layer is absent.

\subsection{Construction of Asymptotic Expansion}

The bridge between the interior solution, the initial layer, and the boundary layer,
is the initial and boundary condition of equation (\ref{transport}). To avoid the introduction of higher order initial-boundary layer, we only require the zeroth-order expansion of initial and first order expansion of boundary data be satisfied, i.e. we have
\begin{eqnarray}
\u_0(0,\vx,\vw)+\ub_{0}^I(0,\vx,\vw)&=&h(\vx,\vw),\\
\no\\
\u_0(t,\vx_0,\vw)+\ub^B_{0}(t,\tau,\phi)&=&\pp[\u_0(t,\vx_0,\vw)+\ub^B_{0}(t,\tau,\phi)],\\
\u_1(t,\vx_0,\vw)+\ub^B_{1}(t,\tau,\phi)&=&\pp[\u_1(t,\vx_0,\vw)+\ub^B_{1}(t,\tau,\phi)]+g(t,\tau,\phi).
\end{eqnarray}
Note the fact that $\bu_k=\pp[\bu_k]$, we can simplify above:
\begin{eqnarray}
\u_0(0,\vx,\vw)+\ub_{0}^I(0,\vx,\vw)&=&h(\vx,\vw),\\
\no\\
\ub^B_{0}(t,\tau,\phi)-\pp[\ub^B_{0}(t,\tau,\phi)]&=&0,\\
\ub^B_{1}(t,\tau,\phi)-\pp[\ub^B_{1}(t,\tau,\phi)]&=&\Big(\vw\cdot\u_0-\pp[\vw\cdot\u_0]\Big)(t,\vx_0,\vw)+g(t,\tau,\phi).
\end{eqnarray}
The construction of $\u_k$, $\ub^I_k$ and $\ub_k^B$ are as follows:\\
\ \\
Step 0: Preliminaries.\\
Assume the cut-off function $\Upsilon\in C^{\infty}[0,\infty)$ is defined as
\begin{eqnarray}\label{cut-off 2}
\Upsilon(\mu)=\left\{
\begin{array}{ll}
1&0\leq\mu\leq\dfrac{1}{4}R,\\
0&\dfrac{1}{2}R\leq\mu\leq\infty.
\end{array}
\right.
\end{eqnarray}
where
\begin{eqnarray}
R_{\min}=\min_{\tau}\rk(\tau).
\end{eqnarray}
Also, we define the force as
\begin{eqnarray}\label{force}
F(\eta,\tau)=-\frac{\e}{\rk(\tau)-\e\eta},
\end{eqnarray}
in the boundary layer length $L=R_{\min}\e^{-\frac{1}{2}}$. The reflexive boundary is
\begin{eqnarray}
\rr[\phi]=-\phi.
\end{eqnarray}
\ \\
Step 1: Construction of $\ub^B_{0}$.\\
Define the zeroth-order boundary layer as
\begin{eqnarray}\label{expansion temp 9}
\left\{
\begin{array}{l}
\ub^B_{0}(t,\eta,\tau,\phi)=\Upsilon(\e^{\frac{1}{2}}\eta)\bigg(f_0(t,\eta,\tau,\phi)-f_{0,L}(t,\tau)\bigg),\\\rule{0ex}{1.5em}
\sin\phi\dfrac{\p f_0}{\p\eta}+F(\eta,\tau)\cos\phi\dfrac{\p
f_0}{\p\phi}+f_0-\bar f_0=0,\\\rule{0ex}{1.5em}
f_0(t,0,\tau,\phi)=\pp[f_0](t,0,\tau)\ \ \text{for}\ \
\sin\phi>0,\\\rule{0ex}{1.5em}
f_0(t,L,\tau,\phi)=f_0(t,L,\tau,\rr[\phi]),
\end{array}
\right.
\end{eqnarray}
with the normalization condition
\begin{eqnarray}
\pp[f_0](t,0,\tau)=0.
\end{eqnarray}
By Theorem \ref{Milne theorem 2.}, $\ub^B_{0}$ is well-defined and
\begin{eqnarray}
f_{0,L}(t,\tau)=\frac{1}{\pi}\ds\int_{-\pi}^{\pi}f(t,L,\tau,\phi)\sin^2\phi\ud{\phi}.
\end{eqnarray}
It is obvious to see $f_0=f_{0,L}=0$ is the only solution.\\
\ \\
Step 2: Construction of $\ub_{0}^I$.\\
Define the zeroth-order initial layer as
\begin{eqnarray}\label{expansion temp 21}
\left\{
\begin{array}{l}
\ub_{0}^I(\sigma,\vx,\vw)=\ff_0(\sigma,\vx,\vw)-\ff_0(\infty,\vx)\\\rule{0ex}{1.0em}
\p_{\sigma}\bar\ff_0=0,\\\rule{0ex}{1.0em}
\ff_0(\sigma,\vx,\vw)=\ue^{-\sigma}\ff_0(0,\vx,\vw)+(1-\ue^{-\sigma})\bar\ff_0(0,\vx),\\\rule{0ex}{1.0em}
\ff_0(0,\vx,\vw)=h(\vx,\vw),\\\rule{0ex}{1.0em}
\lim_{\sigma\rt\infty}\ff_0(\sigma,\vx,\vw)=\ff_0(\infty,\vx).
\end{array}
\right.
\end{eqnarray}
We may directly solve that $\ff_0(\infty,\vx)=\bar\ff_0(0,\vx)=\dfrac{1}{2\pi}\ds\int_{\s^1}h(\vx,\vw)\ud{\vw}$. Hence,
we know $\ub_{0}^I(\sigma,\vx,\vw)=\ue^{-\sigma}\left(h(\vx,\vw)-\dfrac{1}{2\pi}\ds\int_{\s^1}h(\vx,\vw)\ud{\vw}\right)$. It is easy to see that
$\ub_{0}^I\in L^{\infty}$ is well-posed and can be explicitly solved.\\
\ \\
Step 3: Construction of $\ub^B_{1}$ and $\u_0$.\\
Define the first-order boundary layer as
\begin{eqnarray}\label{expansion temp 10}
\left\{
\begin{array}{l}
\ub_1(t,\eta,\tau,\phi)=\Upsilon(\e^{\frac{1}{2}}\eta)\bigg(f_1(t,\eta,\tau,\phi)-f_{1,L}(t,\tau)\bigg),\\\rule{0ex}{1.5em}
\sin\phi\dfrac{\p f_1}{\p\eta}+F(\eta,\tau)\cos\phi\dfrac{\p
f_1}{\p\phi}+f_1-\bar
f_1=\dfrac{1}{\rk-\e\eta}\cos\phi\dfrac{\p
\ub_0}{\p\tau},\\\rule{0ex}{1.5em}
f_1(t,0,\tau,\phi)=\pp
[f_1](t,0,\tau)+g_1(t,\tau,\phi)\ \ \text{for}\ \
\sin\phi>0,\\\rule{0ex}{1.5em}
f_1(t,L,\tau,\phi)=f_1(t,L,\tau,\rr[\phi]),
\end{array}
\right.
\end{eqnarray}
with the normalization condition
\begin{eqnarray}
\pp[f_1](t,0,\tau)=0,
\end{eqnarray}
and
\begin{eqnarray}
g_1(t,\tau,\phi)&=&\Big(\vw\cdot\nx\u_0(t,\vx_0)-\pp[\vw\cdot\nx\u_0](t,\vx_0)\Big)+g(t,\tau,\phi),
\end{eqnarray}
where $\vx_0$ is the same boundary point as $(0,\tau)$ and
\begin{eqnarray}
\vw&=&(-\sin(\phi-\tau),-\cos(\phi-\tau)),\\
\vn&=&(\cos\tau,\sin\tau).
\end{eqnarray}
To solve (\ref{expansion temp 10}), we require the compatibility
condition (\ref{Milne compatibility condition}) for the boundary
data
\begin{eqnarray}
\int_{\sin\phi>0}\bigg(g(t,\tau,\phi)+\vw\cdot\nx\u_0(t,\vx_0)-\pp[\vw\cdot\nx\u_0](t,\vx_0)\bigg)\sin\phi\ud{\phi}=0.
\end{eqnarray}
Note the fact
\begin{eqnarray}
&&\int_{\sin\phi>0}\bigg(\vw\cdot\nx\u_0(t,\vx_0)-\pp[\vw\cdot\nx\u_0](t,\vx_0)\bigg)\sin\phi\ud{\phi}\\
&=&
\int_{\sin\phi>0}\Big(\vw\cdot\nx\u_0(t,\vx_0)\Big)\sin\phi\ud{\phi}-2\pp[\vw\cdot\nx\u_0](t,\vx_0)\nonumber\\
&=&\int_{\sin\phi>0}\Big(\vw\cdot\nx\u_0(t,\vx_0)\Big)\sin\phi\ud{\phi}+\int_{\sin\phi<0}\Big(\vw\cdot\nx\u_0(t,\vx_0)\Big)\sin\phi\ud{\phi}\nonumber\\
&=&\int_{-\pi}^{\pi}\Big(\vw\cdot\nx\u_0(t,\vx_0)\Big)\sin\phi\ud{\phi}\nonumber\\
&=&-\pi\nx\bu_0(t,\vx_0)\cdot\vn=-\pi\frac{\p\bu_0(t,\vx_0)}{\p\vn}.\nonumber
\end{eqnarray}
We can simplify the compatibility condition as follows:
\begin{eqnarray}
\int_{\sin\phi>0}g(t,\tau,\phi)\sin\phi\ud{\phi}-\pi\frac{\p\bu_0(t,\vx_0)}{\p\vn}=0.
\end{eqnarray}
Then we have
\begin{eqnarray}
\frac{\p\bu_0(t,\vx_0)}{\p\vn}&=&\frac{1}{\pi}\int_{\sin\phi>0}g(t,\tau,\phi)\sin\phi\ud{\phi}.
\end{eqnarray}
Hence, we define the zeroth-order interior solution $\u_0(t,\vx,\vw)$ as
\begin{eqnarray}\label{expansion temp 11}
\left\{
\begin{array}{l}
\dt\u_0-\dfrac{1}{2}\Delta_x\u_0=0\ \ \text{for}\ \ (t,\vx)\in[0,\infty)\times\Omega,\\\rule{0ex}{2em}
\u_0(0,\vx)=\displaystyle\frac{1}{2\pi}\int_{\s^1}h(\vx,\vw)\ud{\vw}\
\ \text{for}\ \ \vx\in\Omega,\\\rule{0ex}{2em}
\dfrac{\p\u_0}{\p\vn}(t,\vx_0)=-\dfrac{1}{\pi}\displaystyle
\int_{\vw\cdot\vn<0}g(t,\vx_0,\vw)(\vw\cdot\vn)\ud{\vw}\ \ \text{for}\ \
(t,\vx_0)\in[0,\infty)\times\p\Omega.
\end{array}
\right.
\end{eqnarray}
\ \\
Step 4: Construction of $\ub_{1}^I$.\\
Define the first-order initial layer as
\begin{eqnarray}\label{expansion temp 22}
\left\{
\begin{array}{l}
\ub_{1}^I(\sigma,\vx,\vw)=\ff_1(\sigma,\vx,\vw)-\ff_1(\infty,\vx)\\\rule{0ex}{1.0em}
\p_{\sigma}\bar\ff_1=-\displaystyle\int_{\s^1}\bigg(\vw\cdot\nabla_x\ub_{0}^I\bigg)\ud{\vw},\\\rule{0ex}{1.5em}
\ff_1(\sigma,\vx,\vw)=\ue^{-\sigma}\ff_1(0,\vx,\vw)+\displaystyle\int_0^{\sigma}\bigg(\bar\ff_1-\vw\cdot\nabla_x\ub_{0}^I\bigg)(s,\vx,\vw)\ue^{s-\sigma}\ud{s},\\\rule{0ex}{1.0em}
\ff_1(0,\vx,\vw)=\vw\cdot\nx\u_0(0,\vx,\vw),\\\rule{0ex}{1.0em}
\lim_{\sigma\rt\infty}\ff_1(\sigma,\vx,\vw)=\ff_1(\infty,\vx).
\end{array}
\right.
\end{eqnarray}
This a first-order linear ordinary differential equation and we can easily see that $\ub_{1}^I\in L^{\infty}$ is well-posed.\\
\ \\
Step 5: Construction of $\u_1$.\\
We define the first-order interior solution $\u_1(t,\vx,\vw)$ as
\begin{eqnarray}\label{expansion temp 12}
\left\{
\begin{array}{l}
\u_1=\bu_1-\vw\cdot\nx\u_0\ \ \text{for}\ \ (t,\vx,\vw)\in[0,\infty)\times\Omega\times\s^1,\\\rule{0ex}{1.5em}
\dt\bu_1-\dfrac{1}{2}\Delta_x\bu_1=0\ \ \text{for}\ \ (t,\vx)\in[0,\infty)\times\Omega,\\\rule{0ex}{1.5em}
\bu_1(0,\vx)=0\ \ \text{for}\ \ \vx\in\Omega,\\\rule{0ex}{1.5em}
\dfrac{\p\bu_1}{\p\vn}=0\ \ \text{for}\ \ (t,\vx)\in[0,\infty)\times\p\Omega.
\end{array}
\right.
\end{eqnarray}
Note that here we only require the trivial initial and boundary condition since we cannot resort to the compatibility condition in $\e$-Milne problem with geometric correction. Based on \cite{AA003}, this might lead to $O(\e^2)$ error to the approximation. Since we focus on the leading-order terms, this error is acceptable.\\
\ \\
Step 6: Construction of $\u_2$.\\
By a similar fashion, we define the second order interior solution $\u_2(t,\vx,\vw)$ as
\begin{eqnarray}
\left\{
\begin{array}{l}
\u_2=\bu_2-\vw\cdot\nx\u_1\ \ \text{for}\ \ (t,\vx,\vw)\in[0,\infty)\times\Omega\times\s^1,\\\rule{0ex}{1.5em}
\dt\bu_2-\dfrac{1}{2}\Delta_x\bu_2=0\ \ \text{for}\ \ (t,\vx)\in[0,\infty)\times\Omega,\\\rule{0ex}{1.5em}
\bu_2(0,\vx)=0\ \ \text{for}\ \ \vx\in\Omega,\\\rule{0ex}{1.5em}
\dfrac{\p\bu_2}{\p\vn}=0\ \ \text{for}\ \ (t,\vx)\in[0,\infty)\times\p\Omega.
\end{array}
\right.
\end{eqnarray}
Similar to $\u_1$ case, here we only require the trivial initial and boundary condition. Based on \cite{AA003}, this might lead to $O(\e^3)$ error to the approximation. Since we focus on the leading-order terms, this error is acceptable.

\subsection{Diffusive Limit}

\begin{theorem}\label{diffusive limit}
Assume $g(t,\vx_0,\vw)\in C^2([0,\infty)\times\Gamma^-)$ and
$h(\vx,\vw)\in C^2(\Omega\times\s^1)$. Also, there exists $K_0>0$ such that $\ue^{K_0t}g(t,\vx_0,\vw)\in L^{\infty}([0,\infty)\times\Gamma^-)$. Then the unsteady neutron transport
equation (\ref{transport}) has a unique solution
$u^{\e}(t,\vx,\vw)\in L^{\infty}([0,\infty)\times\Omega\times\s^1)$
satisfying for some $0<K\leq K_0$,
\begin{eqnarray}\label{main theorem 2}
\nm{\ue^{Kt}\Big(u^{\e}-\u_0-\ub_{I,0}-\ub_{B,0}\Big)}_{L^{\infty}([0,\infty)\times\Omega\times\s^1)}\leq C\e^{\frac{1}{2}},
\end{eqnarray}
where the interior solution $\u_0$ is
defined in (\ref{expansion temp 11}), the initial layer $\ub_{0}^I$ is defined in (\ref{expansion temp 21}), and the boundary layer $\ub^B_{0}$ is defined in (\ref{expansion temp 9}).
\end{theorem}
\begin{proof}
We can divide the proof into several steps:\\
\ \\
Step 1: Well-posedness.\\
Based on Theorem \ref{LI estimate}, we directly obtain that there exists a unique solution $u^{\e}(t,\vx,\vw)\in L^{\infty}([0,\infty)\times\Omega\times\s^1)$ and satisfies
\begin{eqnarray}
&&\im{u^{\e}}{[0,\infty)\times\Omega\times\s^1}\\
&\leq& C \bigg(\frac{1}{\e^{\frac{m-1}{2m-1}}}\nm{h}_{L^{2m}(\Omega\times\s^1)}+\frac{1}{\e^{\frac{1}{2}-\frac{9m-4}{2m(2m-1)}}}\nm{\h}_{L^2(\Omega\times\s^1)}+\im{\h}{\Omega\times\s^1}\no\\
&&+\frac{1}{\e^{\frac{1}{m}}}\nm{g}_{L^m([0,\infty)\times\Gamma^+)}+\frac{1}{\e^{\frac{3}{2}-\frac{9m-4}{2m(2m-1)}}}\nm{\g}_{L^2([0,\infty)\times\Gamma^-)}
+\im{\g}{[0,\infty)\times\Gamma^-}\bigg).\nonumber
\end{eqnarray}
for any integer $m>2$. However, this estimate is not uniform in $\e$, so we resort to the expansions.\\
\ \\
Step 2: Remainder definitions.\\
We may rewrite the asymptotic expansion as follows:
\begin{eqnarray}
u^{\e}&\sim&\sum_{k=0}^{2}\e^k\u_k+\sum_{k=0}^{1}\e^k\ub^I_{k}+\sum_{k=0}^{1}\e^k\ub^B_{k}.
\end{eqnarray}
The remainder can be defined as
\begin{eqnarray}\label{pf 1}
R&=&u^{\e}-\sum_{k=0}^{2}\e^k\u_k-\sum_{k=0}^{1}\e^k\ub^I_{k}-\sum_{k=0}^{1}\e^k\ub^B_{k}=u^{\e}-\q-\qb^{I}-\qb^{B},
\end{eqnarray}
where
\begin{eqnarray}
\q&=&\sum_{k=0}^{2}\e^k\u_k,\\
\qb^{I}&=&\sum_{k=0}^{1}\e^k\ub^I_{k},\\
\qb^{B}&=&\sum_{k=0}^{1}\e^k\ub^B_{k}.
\end{eqnarray}
Noting the equation (\ref{transport}) is equivalent to the
equations (\ref{initial temp}) and (\ref{transport temp}), we write $\ll$ to denote the neutron
transport operator as follows:
\begin{eqnarray}
\ll u&=&\e^2\dt u+\e\vw\cdot\nx u+u-\bar u\\
&=&\p_{\sigma}u+\e\vw\cdot\nabla_xu+u-\bar u\no\\
&=&\e^2\frac{\p u}{\p t}+\sin\phi\frac{\p
u}{\p\eta}-\frac{\e}{\rk-\e\eta}\cos\phi\bigg(\frac{\p
u}{\p\phi}+\frac{\p u}{\p\tau}\bigg)+u-\bar u.\nonumber
\end{eqnarray}
\ \\
Step 3: Estimates of $\ll[\q]$.\\
The interior contribution can be estimated as
\begin{eqnarray}
\ll[\q]=\e^2\dt\q+\e\vw\cdot\nx \q+\q-\bar
\q&=&\e^3\dt\u_1+\e^4\dt\u_2+\e^{3}\vw\cdot\nx \u_2.
\end{eqnarray}
We have
\begin{eqnarray}
\abs{\e^3\dt\u_1}&\leq& C\e^{3},\\
\abs{\e^4\dt\u_2}&\leq& C\e^{4},\\
\abs{\e^{3}\vw\cdot\nx \u_2}&\leq& C\e^{3}\abs{\nx\u_2}\leq
C\e^{3}.
\end{eqnarray}
This implies
\begin{eqnarray}
\tm{\ll[\q]}{[0,\infty)\times\Omega\times\s^1}&\leq& C\e^{3},\label{pf 2}\\
\nm{\ll[\q]}_{L^{\frac{2m}{2m-1}}([0,\infty)\times\Omega\times\s^1)}&\leq& C\e^{3},\label{pf 2..}\\
\im{\ll[\q]}{[0,\infty)\times\Omega\times\s^1}&\leq& C\e^{3}.\label{pf 2.}
\end{eqnarray}
\ \\
Step 4: Estimates of $\ll \qb_{I}$.\\
The initial layer contribution can be estimated as
\begin{eqnarray}
\ll\qb_{I}&=&\p_{\tau}\qb_{I}+\e\vw\cdot\nabla_x\qb_{I}+\qb_{I}-\bar \qb_{I}=\e^{2}\nabla_x\ub_{1}^I.
\end{eqnarray}
It is easy to check that
\begin{eqnarray}
\im{\e^{2}\nabla_x\ub_{1}^I}{[0,\infty)\times\Omega\times\s^1}&\leq& C\e^{2}.
\end{eqnarray}
Note that $\ub_{1}^I$ decays exponentially in time and the scaling $\sigma=\dfrac{t}{\e^2}$, we have
\begin{eqnarray}
\tm{\e^{2}\nabla_x\ub_{1}^I}{[0,\infty)\times\Omega\times\s^1}&=&
\e^2\bigg(\int_0^{\infty}\int_{\Omega\times\s^1}(\nabla_x\ub_1)^2(\sigma,\vx,\vw)\ud{\vx}\ud{\vw}\ud{t}\bigg)^{\frac{1}{2}}\\
&\leq&C\e^2\bigg(\int_0^{\infty}\ue^{-\sigma}\ud{t}\bigg)^{\frac{1}{2}}\no\\
&=&C\e^3\bigg(\int_0^{\infty}\ue^{-\sigma}\ud{\sigma}\bigg)^{\frac{1}{2}}\leq C\e^3.\no
\end{eqnarray}
Similarly, we can show that
\begin{eqnarray}
\nm{\e^{2}\nabla_x\ub_{1}^I}_{L^{\frac{2m}{2m-1}}([0,\infty)\times\Omega\times\s^1)}&\leq& C\e^{4-\frac{1}{m}}.
\end{eqnarray}
In total, we know
\begin{eqnarray}
\tm{\ll\qb_{I}}{[0,\infty)\times\Omega\times\s^1}&\leq& C\e^{3},\label{pf 4}\\
\nm{\ll[\qb_I]}_{L^{\frac{2m}{2m-1}}([0,\infty)\times\Omega\times\s^1)}&\leq& C\e^{4-\frac{1}{m}},\label{pf 4..}\\
\im{\ll\qb_{I}}{[0,\infty)\times\Omega\times\s^1}&\leq& C\e^{2}.\label{pf 4.}
\end{eqnarray}
\ \\
Step 5: Estimates of $\ll \qb_{B}$.\\
Since $\ub_0^B=0$, we only need to estimate $\ub_1^B=(f_1-f_{1,L})\cdot\Upsilon=\v\Upsilon$ where
$f_1(\eta,\tau,\phi)$ solves the $\e$-Milne problem with geometric correction and $\v=f_1-f_{1,L}$. Hence, the boundary layer contribution can be
estimated as
\begin{eqnarray}\label{remainder temp 1}
&&\ll[\qb^{B}]\\
&=&\e^3\dt\ub^{B}_1+\sin\phi\frac{\p
(\e\ub^{B}_1)}{\p\eta}-\frac{\e}{\rk-\e\eta}\cos\phi\bigg(\frac{\p
(\e\ub^{B}_1)}{\p\phi}+\frac{\p(\e\ub^{B}_1)}{\p\tau}\bigg)+\e\ub^{B}_1-
\e\bub^{B}_1\no\\
&=&\e^3\dt\ub^B_{1}+\e\Bigg(\sin\phi\bigg(\Upsilon\frac{\p
\v}{\p\eta}+\v\frac{\p\Upsilon}{\p\eta}\bigg)-\frac{\Upsilon\e}{\rk-\e\eta}\cos\phi\bigg(\frac{\p
\v}{\p\phi}+\frac{\p \v}{\p\tau}\bigg)+\Upsilon\v-\Upsilon\bar\v\Bigg)\nonumber\\
&=&\e^3\dt\ub^B_{1}+\e\Upsilon\bigg(\sin\phi\frac{\p
\v}{\p\eta}-\frac{\e}{\rk-\e\eta}\cos\phi\frac{\p
\v}{\p\phi}+\v-\bar\v\bigg)+\e\Bigg(\sin\phi
\frac{\p\Upsilon}{\p\eta}\v-\frac{\Upsilon\e}{\rk-\e\eta}\cos\phi\frac{\p
\v}{\p\tau}\Bigg)\nonumber\\
&=&\e^3\dt\ub^B_{1}+\e\Bigg(\sin\phi
\frac{\p\Upsilon}{\p\eta}\v-\frac{\Upsilon\e}{\rk-\e\eta}\cos\phi\frac{\p
\v}{\p\tau}\Bigg)\nonumber.
\end{eqnarray}
We may directly bound the first term
\begin{eqnarray}
\im{\e^3\dt\ub^B_{1}}{[0,\infty)\times\Omega\times\s^1}&\leq& C\e^{3}.
\end{eqnarray}
Since $\Upsilon=1$ when $\eta\leq \dfrac{R_{\min}}{4\e^{\frac{1}{2}}}$, the effective region
of $\p_{\eta}\Upsilon$ is $\eta\geq \dfrac{R_{\min}}{4\e^{\frac{1}{2}}}$ which is further and further
from the origin as $\e\rt0$. By Theorem \ref{Milne theorem 3.}, the
second term in (\ref{remainder temp 1}) can be controlled as
\begin{eqnarray}
\im{\e\sin\phi\frac{\p\Upsilon}{\p\eta}\v}{[0,\infty)\times\Omega\times\s^1}&\leq&
C\ue^{-\frac{K_0}{\e^{\frac{1}{2}}}}\leq C\e^3.
\end{eqnarray}
For the third term in (\ref{remainder temp 1}), by Theorem \ref{Milne tangential.}, we have
\begin{eqnarray}
\im{-\frac{\Upsilon\e^2}{\rk-\e\eta}\cos\phi\frac{\p
\v}{\p\tau}}{[0,\infty)\times\Omega\times\s^1}&\leq&C\e^2\im{\frac{\p \v}{\p\tau}}{[0,\infty)\times\Omega\times\s^1}\leq C\e^2\abs{\ln{\e}}^8.
\end{eqnarray}
Also, the exponential decay of $\dfrac{\p\v}{\p\tau}$ by Theorem \ref{Milne tangential.} and the rescaling $\eta=\dfrac{\mu}{\e}$ implies
\begin{eqnarray}
\tm{-\frac{\Upsilon\e^2}{\rk-\e\eta}\cos\phi\frac{\p
\v}{\p\tau}}{[0,\infty)\times\Omega\times\s^1}&\leq& \e^2\tm{\frac{\p
\v}{\p\tau}}{[0,\infty)\times\Omega\times\s^1}\\
&\leq&\e^2\Bigg(\int_0^{\infty}\int_{-\pi}^{\pi}\int_0^{R_{\min}}(R_{\min}-\mu)\lnm{\frac{\p\v}{\p\tau}(t,\mu,\tau)}^2\ud{\mu}\ud{\tau}\ud{t}\Bigg)^{1/2}\no\\
&\leq&\e^{\frac{5}{2}}\Bigg(\int_0^{\infty}\int_{-\pi}^{\pi}\int_0^{R_{\min}\e^{-1}}(R_{\min}-\e\eta)\lnm{\frac{\p\v}{\p\tau}(\eta,\tau)}^2\ud{\eta}\ud{\tau}\ud{t}\Bigg)^{1/2}\no\\
&\leq&C\e^{\frac{5}{2}}\Bigg(\int_0^{\infty}\int_{-\pi}^{\pi}\int_0^{R_{\min}\e^{-1}}\ue^{-Kt}\ue^{-2K\eta}\abs{\ln(\e)}^{16}\ud{\eta}\ud{\tau}\ud{t}\Bigg)^{1/2}\no\\
&\leq& C\e^{\frac{5}{2}}\abs{\ln(\e)}^8.\no
\end{eqnarray}
Similarly, we have
\begin{eqnarray}
\nm{-\frac{\Upsilon\e^2}{\rk-\e\eta}\cos\phi\frac{\p
\v}{\p\tau}}_{L^{\frac{2m}{2m-1}}([0,\infty)\times\Omega\times\s^1)}&\leq&C\e^{3-\frac{1}{2m}}\abs{\ln(\e)}^8.
\end{eqnarray}
In total, we have
\begin{eqnarray}
\tm{\ll[\qb_{B}]}{[0,\infty)\times\Omega\times\s^1}&\leq& C\e^{\frac{5}{2}}\abs{\ln{\e}}^8,\label{pf 3}\\
\nm{\ll[\qb_B]}_{L^{\frac{2m}{2m-1}}([0,\infty)\times\Omega\times\s^1)}&\leq& C\e^{3-\frac{1}{2m}}\abs{\ln{\e}}^8,\label{pf 3..}\\
\im{\ll[\qb_{B}]}{[0,\infty)\times\Omega\times\s^1}&\leq& C\e^2\abs{\ln{\e}}^8.\label{pf 3..}
\end{eqnarray}
\ \\
Step 6: Diffusive Limit.\\
In summary, since $\ll[u^{\e}]=0$, collecting estimates in Step 2, Step 3 and Step 4, we can prove
\begin{eqnarray}
\tm{\ll[R]}{[0,\infty)\times\Omega\times\s^1}&\leq& C\e^{\frac{5}{2}}\abs{\ln{\e}}^8,\\
\nm{\ll[R]}_{L^{\frac{2m}{2m-1}}([0,\infty)\times\Omega\times\s^1)}&\leq& C\e^{3-\frac{1}{2m}}\abs{\ln{\e}}^8,\\
\im{\ll[R]}{[0,\infty)\times\Omega\times\s^1}&\leq& C\e^{2}\abs{\ln{\e}}^8.
\end{eqnarray}
Also, noting that $\ub^I_0(\sigma,\vx_0,\vw)=\pp[\ub_0^I](\sigma,\vx_0)$, based on our construction, at boundary $\p\Omega$, it is easy to see
\begin{eqnarray}
R-\pp[R]=-\e^2\Big(\vw\cdot\nx\u_1-\pp[\vw\cdot\nx\u_1]\Big)-\e\Big(\ub_1^I-\pp[\ub_1^I]\Big),
\end{eqnarray}
which, using the rescaling $\sigma=\dfrac{t}{\e^2}$, further implies
\begin{eqnarray}
\tm{R-\pp[R]}{[0,\infty)\times\Gamma^-}&\leq& C\e^2,\\
\nm{R-\pp[R]}_{L^m([0,\infty)\times\Gamma^-)}&\leq& C\e^{1+\frac{2}{m}},\\
\im{R-\pp[R]}{[0,\infty)\times\Gamma^-}&\leq& C\e.
\end{eqnarray}
On the other hand, at $t=0$, we have
\begin{eqnarray}
R_0(\vx,\vw)=R(0,\vx,\vw)=-\e\ub^B_1(0,\eta,\tau,\phi)-\e^2\u_2(0,\vx,\vw),
\end{eqnarray}
which, using the rescaling $\eta=\dfrac{\mu}{\e}$, further implies
\begin{eqnarray}
\tm{R_0}{\Omega\times\s^1}&\leq& C\e^{\frac{3}{2}},\\
\nm{R_0}_{L^{2m}(\Omega\times\s^1)}&\leq& C\e^{1+\frac{1}{2m}},\\
\im{R_0}{\Omega\times\s^1}&\leq& C\e.
\end{eqnarray}
Therefore, the remainder $R$ satisfies the equation
\begin{eqnarray}
\left\{
\begin{array}{rcl}
\e^2\dt R+\e \vw\cdot\nabla_x R+R-\bar
R&=&\ll[R]\ \ \ \text{for}\ \
(t,\vx,\vw)\in[0,\infty)\times\Omega\times\s^1,\\\rule{0ex}{1.0em}
R(0,\vx,\vw)&=&R_0(\vx,\vw)\ \ \text{for}\ \ (\vx,\vw)\in\Omega\times\s^1\\\rule{0ex}{1.0em}
(R-\pp[R])(t,\vx_0,\vw)&=&(R-\pp[R])(t,\vx_0,\vw)\ \ \text{for}\ \ t\in[0,\infty),\ \ \vx_0\in\p\Omega,\ \ \text{and}\ \ \vw\cdot\vn<0.
\end{array}
\right.
\end{eqnarray}
By Theorem \ref{LI estimate}, we have for integer $m\geq2$,
\begin{eqnarray}
&&\im{R}{[0,\infty)\times\Omega\times\s^1}\\
&\leq& C \bigg(\frac{1}{\e^{\frac{5}{2}-\frac{9m-4}{2m(2m-1)}}}\nm{\ll[R]}_{L^{\frac{2m}{2m-1}}([0,\infty)\times\Omega\times\s^1)}
+\frac{1}{\e^{\frac{3}{2}-\frac{9m-4}{2m(2m-1)}}}\tm{\ll[R]}{[0,\infty)\times\Omega\times\s^1}
+\im{\ll[R]}{[0,\infty)\times\Omega\times\s^1}\no\\
&&+\frac{1}{\e^{\frac{m-1}{2m-1}}}\nm{R_0}_{L^{2m}(\Omega\times\s^1)}+\frac{1}{\e^{\frac{1}{2}-\frac{9m-4}{2m(2m-1)}}}\nm{R_0}_{L^2(\Omega\times\s^1)}+\im{R_0}{\Omega\times\s^1}\no\\
&&+\frac{1}{\e^{\frac{1}{m}}}\nm{R-\pp[R]}_{L^m([0,\infty)\times\Gamma^+)}+\frac{1}{\e^{\frac{3}{2}-\frac{9m-4}{2m(2m-1)}}}\nm{R-\pp[R]}_{L^2([0,\infty)\times\Gamma^-)}
+\im{R-\pp[R]}{[0,\infty)\times\Gamma^-}\bigg)\no\\
&\leq& C \bigg(\frac{\e^{3-\frac{1}{2m}}\abs{\ln{\e}}^8}{\e^{\frac{5}{2}-\frac{9m-4}{2m(2m-1)}}}
+\frac{\e^{\frac{5}{2}}\abs{\ln{\e}}^8}{\e^{\frac{3}{2}-\frac{9m-4}{2m(2m-1)}}}
+\e^{2}\abs{\ln{\e}}^8+\frac{\e^{1+\frac{1}{2m}}}{\e^{\frac{m-1}{2m-1}}}+\frac{\e^{\frac{3}{2}}}{\e^{\frac{1}{2}-\frac{9m-4}{2m(2m-1)}}}+\e
+\frac{\e^{1+\frac{2}{m}}}{\e^{\frac{1}{m}}}+\frac{\e^2}{\e^{\frac{3}{2}-\frac{9m-4}{2m(2m-1)}}}
+\e\bigg)\no\\
&\leq&C\max\left\{\e^{\frac{1}{2}+\frac{7m-3}{2m(2m-1)}}\abs{\ln{\e}}^8,\e^{\frac{1}{2}+\frac{3m-1}{2m(2m-1)}}\right\}\no\\
&\leq&C\e^{\frac{1}{2}}.\no
\end{eqnarray}
Since it is obvious that
\begin{eqnarray}
\im{\e\u_1+\e^2\u_2+\e\ub_{1}^I+\e\ub^B_{1}}{[0,\infty)\times\Omega\times\s^1}\leq C\e,
\end{eqnarray}
our result naturally follows. The exponential decay can be easily derived following a similar argument using Theorem \ref{LI estimate..}, since their estimates are almost the same. This completes the proof of our main theorem.
\end{proof}

\section{Regularity of $\e$-Milne Problem with Geometric Correction}

We consider the $\e$-Milne problem with geometric correction for $f^{\e}(\eta,\tau,\phi)$ in
the domain $(\eta,\tau,\phi)\in[0,L]\times[-\pi,\pi)\times[-\pi,\pi)$ where $L=R_{\min}\e^{-\frac{1}{2}}$ as
\begin{eqnarray}\label{Milne problem.}
\left\{ \begin{array}{l}\displaystyle \sin\phi\frac{\p
f^{\e}}{\p\eta}+F(\eta,\tau)\cos\phi\frac{\p
f^{\e}}{\p\phi}+f^{\e}-\bar f^{\e}=S^{\e}(\eta,\tau,\phi),\\\rule{0ex}{1.0em}
f^{\e}(0,\tau,\phi)= h^{\e}(\tau,\phi)+\pp[f^{\e}](0,\tau)\ \ \text{for}\
\ \sin\phi>0,\\\rule{0ex}{1.0em}
f^{\e}(L,\tau,\phi)=f^{\e}(L,\tau,\rr[\phi]),
\end{array}
\right.
\end{eqnarray}
where $\rr[\phi]=-\phi$,
\begin{eqnarray}
\pp
[f^{\e}](0,\tau)=-\half\int_{\sin\phi<0}f^{\e}(0,\tau,\phi)\sin\phi\ud{\phi},
\end{eqnarray}
and
\begin{eqnarray}
F(\eta,\tau)=-\frac{\e}{\rk (\tau)-\e\eta},
\end{eqnarray}
Define a potential function $V(\eta,\tau)$ satisfying that $\p_{\eta}V=-F$ and $V(0,\tau)=0$. In this section, for convenience, we temporarily ignore the superscript on $\e$. We define the norms in the
space $(\eta,\phi)\in[0,L]\times[-\pi,\pi)$ as follows:
\begin{eqnarray}
\tnnm{f(\tau)}&=&\bigg(\int_0^{L}\int_{-\pi}^{\pi}\abs{f(\eta,\tau,\phi)}^2\ud{\phi}\ud{\eta}\bigg)^{1/2},\\
\lnnm{f(\tau)}&=&\sup_{(\eta,\phi)\in[0,L]\times[-\pi,\pi)}\abs{f(\eta,\tau,\phi)},
\end{eqnarray}
Similarly, we can define the norm at in-flow boundary as
\begin{eqnarray}
\tnm{f(0,\tau)}&=&\bigg(\int_{\sin\phi>0}\abs{f(0,\tau,\phi)}^2\ud{\phi}\bigg)^{1/2},\\
\lnm{f(0,\tau)}&=&\sup_{\sin\phi>0}\abs{f(0,\tau,\phi)},
\end{eqnarray}
We further assume
\begin{eqnarray}\label{Milne bounded}
\lnm{h(\tau)}+\lnm{\frac{\p h}{\p\phi}(\tau)}+\lnm{\frac{\p h}{\p\tau}(\tau)}\leq C,
\end{eqnarray}
and
\begin{eqnarray}\label{Milne decay}
\lnnm{\ue^{K_0\eta}S(\tau)}+\lnnm{\ue^{K_0\eta}\frac{\p S}{\p\eta}(\tau)}+\lnnm{\ue^{K_0\eta}\frac{\p S}{\p\phi}(\tau)}+\lnnm{\ue^{K_0\eta}\frac{\p S}{\p\tau}(\tau)}\leq C,
\end{eqnarray}
for $C>0$ and $K_0>0$ uniform in $\e$ and $\tau$. In \cite{AA007} and \cite[Section 6]{AA003}, it has been proved that
\begin{lemma}\label{Milne lemma 1.}
In order for the equation (\ref{Milne problem.}) to have a solution
$f(\eta,\tau,\phi)\in L^{\infty}([0,L]\times[-\pi,\pi)\times[-\pi,\pi))$, the boundary data $h$
and the source term $S$ must satisfy the compatibility condition
\begin{eqnarray}\label{Milne compatibility condition}
\int_{\sin\phi>0}h(\tau,\phi)\sin\phi\ud{\phi}
+\int_0^{L}\int_{-\pi}^{\pi}\ue^{-V(s)}S(s,\tau,\phi)\ud{\phi}\ud{s}=0.
\end{eqnarray}
In particular, if $S=0$, then the compatibility condition reduces to
\begin{eqnarray}\label{Milne reduced compatibility condition}
\int_{\sin\phi>0}h(\tau,\phi)\sin\phi\ud{\phi}=0.
\end{eqnarray}
\end{lemma}
It is easy to see if $f$ is a solution to (\ref{Milne problem.}),
then $f+C$ is also a solution for any constant $C$. Hence, in order
to obtain a unique solution, we need a normalization condition
\begin{eqnarray}\label{Milne normalization}
\pp[f](0,\tau)=0.
\end{eqnarray}
Hence, based on \cite{AA007}, we have the well-posedness and regularity results.
\begin{theorem}\label{Milne theorem 1.}
There exists a unique solution $f(\eta,\tau,\phi)$ to the $\e$-Milne problem
(\ref{Milne problem.}) with the normalization condition (\ref{Milne
normalization}) satisfying
\begin{eqnarray}
\tnnm{f(\eta,\tau,\phi)-f_L(\tau)}\leq C,
\end{eqnarray}
for some $f_L\in\r$.
\end{theorem}
\begin{theorem}\label{Milne theorem 2.}
The unique solution $f(\eta,\tau,\phi)$ to the $\e$-Milne problem
(\ref{Milne problem.}) with the normalization condition (\ref{Milne
normalization}) satisfies
\begin{eqnarray}
\lnnm{f(\eta,\tau,\phi)-f_L(\tau)}\leq C.
\end{eqnarray}
\end{theorem}
\begin{theorem}\label{Milne theorem 3.}
There exists $K>0$ such that the solution $f(\eta,\tau,\phi)$ to the
$\e$-Milne problem (\ref{Milne problem.}) with the normalization
condition (\ref{Milne normalization}) satisfies
\begin{eqnarray}
\lnnm{\ue^{K\eta}\bigg(f(\eta,\tau,\phi)-f_L(\tau)\bigg)}\leq C.
\end{eqnarray}
\end{theorem}
\begin{theorem}\label{Milne tangential.}
There exists $K>0$ such that the solution $f(\eta,\tau,\phi)$ to the
$\e$-Milne problem (\ref{Milne problem.}) with the normalization
condition (\ref{Milne normalization}) satisfies
\begin{eqnarray}
\lnnm{\ue^{K\eta}\frac{\p(f-f_L)}{\p\tau}(\eta,\tau,\phi)}\leq C\abs{\ln(\e)}^8.
\end{eqnarray}
\end{theorem}

\section{Remainder Estimate}

In this section, we consider the remainder equation for $u(t,\vx,\vw)$ as
\begin{eqnarray}\label{neutron}
\left\{
\begin{array}{l}
\e^2\dt u+\e \vw\cdot\nabla_x u+u-\bar
u=\ss\ \ \ \text{for}\ \
(t,\vx,\vw)\in[0,\infty)\times\Omega\times\s^1,\\\rule{0ex}{1.0em}
u(0,\vx,\vw)=\h(\vx,\vw)\ \ \text{for}\ \ (\vx,\vw)\in\Omega\times\s^1\\\rule{0ex}{1.0em}
u(t,\vx_0,\vw)-\pp[u](t,\vx_0)=\g(t,\vx_0,\vw)\ \ \text{for}\ \ t\in[0,\infty),\ \ \vx_0\in\p\Omega\ \ \text{and}\ \ \vw\cdot\vn<0,
\end{array}
\right.
\end{eqnarray}
where
\begin{eqnarray}
\bar u(t,\vx)=\frac{1}{2\pi}\int_{\s^1}u(t,\vx,\vw)\ud{\vw},
\end{eqnarray}
\begin{eqnarray}
\pp[u](t,\vx_0)=\frac{1}{2}\int_{\vw\cdot\vn>0}u(t,\vx_0,\vw)(\vw\cdot\vn)\ud{\vw},
\end{eqnarray}
$\vn$ is the outward unit normal vector, with the Knudsen number $0<\e<<1$. The initial and boundary data satisfy the
compatibility condition
\begin{eqnarray}
\h(\vx_0,\vw)-\pp[\h](\vx_0)=\g(0,\vx_0,\vw)\ \ \text{for}\ \ \vx_0\in\p\Omega\ \ \text{and}\ \ \vw\cdot\vn<0.
\end{eqnarray}
We define the $L^p$ norm with $1\leq p<\infty$ and $L^{\infty}$ norm in $\Omega\times\s^1$ as
usual:
\begin{eqnarray}
\nm{f}_{L^p(\Omega\times\s^1)}&=&\bigg(\int_{\Omega}\int_{\s^1}\abs{f(\vx,\vw)}^p\ud{\vw}\ud{\vx}\bigg)^{1/p},\\
\nm{f}_{L^{\infty}(\Omega\times\s^1)}&=&\sup_{(\vx,\vw)\in\Omega\times\s^1}\abs{f(\vx,\vw)}.
\end{eqnarray}
Define the $L^p$ norm with $1\leq p<\infty$ and $L^{\infty}$ norm on the boundary as follows:
\begin{eqnarray}
\nm{f}_{L^p(\Gamma)}&=&\bigg(\iint_{\Gamma}\abs{f(\vx,\vw)}^p\abs{\vw\cdot\vn}\ud{\vw}\ud{\vx}\bigg)^{1/p},\\
\nm{f}_{L^p(\Gamma^{\pm})}&=&\bigg(\iint_{\Gamma^{\pm}}\abs{f(\vx,\vw)}^p\abs{\vw\cdot\vn}\ud{\vw}\ud{\vx}\bigg)^{1/p},\\
\nm{f}_{L^{\infty}(\Gamma)}&=&\sup_{(\vx,\vw)\in\Gamma}\abs{f(\vx,\vw)},\\
\nm{f}_{L^{\infty}(\Gamma^{\pm})}&=&\sup_{(\vx,\vw)\in\Gamma^{\pm}}\abs{f(\vx,\vw)}.
\end{eqnarray}
Similar notation also applies to the space
$[0,\infty)\times\Omega\times\s^1$, $[0,\infty)\times\Gamma$, and
$[0,\infty)\times\Gamma^{\pm}$.

\subsection{Preliminaries}

In order to show the $L^{\infty}$ estimates of the
equation (\ref{neutron}), we start with some preparations with the
transport equation.
\begin{lemma}\label{well-posedness lemma 1}
Assume $\ss(t,\vx,\vw)\in
L^{\infty}([0,\infty)\times\Omega\times\s^1)$, $\h(\vx,\vw)\in
L^{\infty}(\Omega\times\s^1)$ and $\g(t,x_0,\vw)\in
L^{\infty}([0,\infty)\times\Gamma^-)$. Then the
transport equation
\begin{eqnarray}\label{penalty equation}
\left\{
\begin{array}{l}
\e^2\dt u+\e \vw\cdot\nabla_x u+u=\ss\ \ \ \text{for}\ \
(t,\vx,\vw)\in[0,\infty)\times\Omega\times\s^1,\\\rule{0ex}{1.0em}
u(0,\vx,\vw)=\h(\vx,\vw)\ \ \text{for}\ \ (\vx,\vw)\in\Omega\times\s^1\\\rule{0ex}{1.0em}
u(t,\vx_0,\vw)=\g(t,\vx_0,\vw)\ \ \text{for}\ \ t\in[0,\infty),\ \ \vx_0\in\p\Omega\ \ \text{and}\ \ \vw\cdot\vn<0,
\end{array}
\right.
\end{eqnarray}
has a solution
$u(t,\vx,\vw)\in L^{\infty}([0,\infty)\times\Omega\times\s^1)$
satisfying
\begin{eqnarray}
\im{u}{[0,\infty)\times\Omega\times\s^1}\leq
\im{\g}{[0,\infty)\times\Gamma^-}+\im{\h}{\Omega\times\s^1}+\im{\ss}{[0,\infty)\times\Omega\times\s^1}.
\end{eqnarray}
\end{lemma}
\begin{proof}
The characteristics $\Big(T(s),X(s),W(s)\Big)$ of the equation (\ref{penalty
equation}) which goes through $(t,\vx,\vw)$ is defined by
\begin{eqnarray}\label{character}
\left\{
\begin{array}{rcl}
\Big(T(0),X(0),W(0)\Big)&=&(t,\vx,\vw)\\\rule{0ex}{2.0em}
\dfrac{\ud{T(s)}}{\ud{s}}&=&\e^2,\\\rule{0ex}{2.0em}
\dfrac{\ud{X(s)}}{\ud{s}}&=&\e W(s),\\\rule{0ex}{2.0em}
\dfrac{\ud{W(s)}}{\ud{s}}&=&0.
\end{array}
\right.
\end{eqnarray}
which implies
\begin{eqnarray}
\left\{
\begin{array}{rcl}
T(s)&=&t+\e^2s,\\
X(s)&=&\vx+(\e\vw)s,\\
W(s)&=&\vw,
\end{array}
\right.
\end{eqnarray}
Hence, we can rewrite the equation (\ref{penalty equation}) along
the characteristics as
\begin{eqnarray}\label{well-posedness temp 31}
&&u(t,\vx,\vw)\\
&=&{\bf 1}_{\{t\geq \e^2t_b\}}\bigg(\g\Big(t-\e^2t_b,\vx-(\e
\vw)t_b,\vw\Big)\ue^{-t_b}+\int_{0}^{t_b}\ss\Big(t-\e^2(t_b-s),\vx-\e(t_b-s)\vw,\vw\Big)\ue^{-(t_b-s)}\ud{s}\bigg)\no\\
&&+{\bf 1}_{\{t\leq \e^2t_b\}}\bigg(\h\left(\vx-(\e\vw)\frac{
t}{\e^2},\vw\right)\ue^{-\frac{t}{\e^2}}+\int_{0}^{\frac{t}{\e^2}}\ss\left(\e^2s,\vx-\e\left(\frac{t}{\e^2}-s\right)\vw,\vw\right)\ue^{-(\frac{t}{\e^2}-s)}\ud{s}\bigg)\no,
\end{eqnarray}
where the backward exit time $t_b$ is defined as
\begin{equation}\label{exit time}
t_b(\vx,\vw)=\inf\{s\geq0: (\vx-\e s\vw,\vw)\in\Gamma^-\}.
\end{equation}
Then we can naturally estimate
\begin{eqnarray}
\im{u}{[0,\infty)\times\Omega\times\s^1}
&\leq&{\bf 1}_{\{t\geq \e^2t_b\}}\bigg(\ue^{-t_b}\im{\g}{[0,\infty)\times\Gamma^-}+(1-\ue^{t_b})\im{\ss}{[0,\infty)\times\Omega\times\s^1}\bigg)\\
&&+{\bf 1}_{\{t\leq \e^2t_b\}}\bigg(\ue^{-\frac{t}{\e^2}}\im{\h}{\Omega\times\s^1}+(1-\ue^{\frac{t}{\e^2}})\im{\ss}{[0,\infty)\times\Omega\times\s^1}\bigg)\no\\
&\leq&\im{\g}{[0,\infty)\times\Gamma^-}+\im{\h}{\Omega\times\s^1}+\im{\ss}{[0,\infty)\times\Omega\times\s^1}\nonumber.
\end{eqnarray}
Since $u$ can be explicitly tracked back to the initial or
boundary data, the existence naturally follows from above estimate.
\end{proof}

\subsection{$L^2$ Estimate}

In this section, we start from the preliminary equation (\ref{penalty equation}) and take $\bar u$ and $\pp[u]$ into consideration.
\begin{lemma}\label{well-posedness lemma 2}
Define the near-grazing set of $\Gamma^+$ or $\Gamma^-$ as
\begin{eqnarray}
\Gamma_{\pm}^{\delta}=\left\{(\vx,\vw)\in\Gamma^{\pm}:
\abs{\vn(\vx)\cdot\vw}\leq\delta\right\}.
\end{eqnarray}
Then
\begin{eqnarray}
&&\int_s^t\nm{f(r){\bf{1}}_{\{\Gamma^{\pm}\backslash\Gamma_{\pm}^{\delta}\}}}_{L^1(\Gamma^{\pm})}\ud{r}\\
&\leq&
C(\delta)\bigg(\e\nm{f(s)}_{L^1(\Omega\times\s^1)}+\int_s^t\left(\nm{f(r)}_{L^1(\Omega\times\s^1)}
+\nm{(\e\dt+\vw\cdot\nx)f(r)}_{L^1(\Omega\times\s^1)}\right)\ud{r}\bigg).\no
\end{eqnarray}
\end{lemma}
\begin{proof}
See the proof of \cite[Lemma 2.1]{Esposito.Guo.Kim.Marra2013} with a standard scaling argument.
\end{proof}
\begin{lemma}(Green's Identity)\label{well-posedness lemma 3}
Assume $f(t,\vx,\vw),\ g(t,\vx,\vw)\in
L^{\infty}([0,\infty)\times\Omega\times\s^1)$ and $\dt f+\vw\cdot\nx
f,\ \dt g+\vw\cdot\nx g\in L^2([0,\infty)\times\Omega\times\s^1)$
with $f,\ g\in L^2([0,\infty)\times\Gamma)$. Then for almost all
$s,t\in[0,\infty)$,
\begin{eqnarray}
&&\int_s^t\iint_{\Omega\times\s^1}\bigg((\dt f+\vw\cdot\nx f)g+(\dt
g+\vw\cdot\nx
g)f\bigg)\ud{\vx}\ud{\vw}\ud{r}\\
&=&\int_s^t\int_{\Gamma}fg\ud{\gamma}\ud{r}+\iint_{\Omega\times\s^1}f(t)g(t)\ud{\vx}\ud{\vw}-\iint_{\Omega\times\s^1}f(s)g(s)\ud{\vx}\ud{\vw},\no
\end{eqnarray}
where $\ud{\gamma}=(\vw\cdot\vn)\ud{s}$ on the boundary.
\end{lemma}
\begin{proof}
See \cite[Chapter 9]{Cercignani.Illner.Pulvirenti1994} and
\cite{Esposito.Guo.Kim.Marra2013}.
\end{proof}
\begin{theorem}\label{LT estimate}
Assume $\ss(t,\vx,\vw)\in
L^{\infty}([0,\infty)\times\Omega\times\s^1)$, $\h(\vx,\vw)\in
L^{\infty}(\Omega\times\s^1)$ and $\g(t,x_0,\vw)\in
L^{\infty}([0,\infty)\times\Gamma^-)$. Then the neutron
transport equation (\ref{neutron}) has a unique solution
$u(t,\vx,\vw)\in L^2([0,\infty)\times\Omega\times\s^1)$ satisfying
\begin{eqnarray}
&&\nm{u(t)}_{L^2(\Omega\times\s^1)}+\frac{1}{\e^{\frac{1}{2}}}\nm{(1-\pp)[u]}_{L^2([0,\infty)\times\Gamma^-)}+\nm{u}_{L^2([0,\infty)\times\Omega\times\s^1)}\\
&\leq& C\bigg( \frac{1}{\e^2}\nm{\ss}_{L^2([0,\infty)\times\Omega\times\s^1)}
+\nm{\h}_{L^2(\Omega\times\s^1)}+\frac{1}{\e}\nm{\g}_{L^2([0,\infty)\times\Gamma^-)}\bigg).\no
\end{eqnarray}
\end{theorem}
\begin{proof}
We divide the proof into several steps:\\
\ \\
Step 1: Penalized equation.\\
We first consider the penalized equation for $u_{j,\l}$,
\begin{eqnarray}\label{penalty neutron}
\left\{
\begin{array}{l}\displaystyle
\e^2\dt u_{j,\l}+\e\vw\cdot\nx u_{j,\l}+(1+\l)u_{j,\l}-\bar
u_{j,\l}=\ss(t,\vx,\vw)\ \ \text{for}\ \ (t,\vx,\vw)\in[0,\infty)\times\Omega\times\s^1,\\\rule{0ex}{1.0em}
u_{j,\l}(0,\vx,\vw)=h(\vx,\vw)\ \
\text{for}\ \ (\vx,\vw)\in\Omega\times\s^1\\\rule{0ex}{1.0em}
u_{j,\l}(t,\vx_0,\vw)-\left(1-\dfrac{1}{j}\right)\pp[u_{j,\l}](t,\vx_0)=\g(t,\vx_0,\vw)\ \ \text{for}\ \ t\in[0,\infty),\ \
\vx_0\in\p\Omega\ \ \text{and}\ \ \vw\cdot\vn<0,
\end{array}
\right.
\end{eqnarray}
for $\l>0$, $j\in\mathbb{N}$ and $j\geq \dfrac{2}{\l}$. We iteratively construct
an approximating sequence $\{u^k_{j,\l}\}_{k=0}^{\infty}$ where $u^0_{j,\l}=0$ and
\begin{eqnarray}\label{penalty iteration}
\left\{
\begin{array}{l}\displaystyle
\e^2\dt u^k_{\j,\l}+\e\vw\cdot\nx u_{j,\l}^k+(1+\l)u_{j,\l}^k-\bar u_{j,\l}^{k-1}=\ss(t,\vx,\vw)\ \ \text{for}\ \ (t,\vx,\vw)\in[0,\infty)\times\Omega\times\s^1,\\\rule{0ex}{1.0em} u_{j}^k(0,\vx,\vw)=h(\vx,\vw)\ \
\text{for}\ \ (\vx,\vw)\in\Omega\times\s^1\\\rule{0ex}{1.0em}
u_{j,\l}^k(t,\vx_0,\vw)-\left(1-\dfrac{1}{j}\right)\pp[u_{j,\l}^{k-1}](t,\vx_0)=\g(t,\vx_0,\vw)
\ \ \text{for}\ \ t\in[0,\infty),\ \
\vx_0\in\p\Omega\ \ \text{and}\ \ \vw\cdot\vn<0.
\end{array}
\right.
\end{eqnarray}
By Lemma \ref{well-posedness lemma 1}, this sequence is
well-defined and $\im{u_{j,\l}^k}{[0,\infty)\times\Omega\times\s^1}<\infty$.
We rewrite equation (\ref{penalty iteration}) along the
characteristics as
\begin{eqnarray}
u_{j,\l}^k(t,\vx,\vw)
&=&{\bf 1}_{\{t\geq \e^2t_b\}}\Bigg( \left(\g+\left(1-\dfrac{1}{j}\right)\pp[u_{j,\l}^{k-1}]\right)\Big(t-\e^2t_b,\vx-(\e\vw)
t_b,\vw\Big)\ue^{-(1+\l)t_b}\\
&&+\int_{0}^{t_b}(\bar u_{j,\l}^{k-1}+\ss)\Big(t-\e^2(t_b-s),\vx-\e(t_b-s)\vw,\vw\Big)\ue^{-(1+\l)(t_b-s)}\ud{s}\Bigg)\no\\
&&+{\bf 1}_{\{t\leq \e^2t_b\}}\Bigg(\h\left(\vx-(\e\vw)\frac{
t}{\e^2},\vw\right)\ue^{-(1+\l)\frac{t}{\e^2}}\no\\
&&+\int_{0}^{\frac{t}{\e^2}}(\bar
u_{j,\l}^{k-1}+\ss)\left(\e^2s,\vx-\e\left(\frac{t}{\e^2}-s\right)\vw,\vw\right)\ue^{-(1+\l)(\frac{t}{\e^2}-s)}\ud{s}\Bigg)\no,
\end{eqnarray}
We define the difference $v^k_{j,\l}=u^{k}_{j,\l}-u^{k-1}_{j,\l}$ for $k\geq1$. Then
$v^k_{j,\l}$ satisfies
\begin{eqnarray}
v^{k}_{j,\l}(t,\vx,\vw)
&=&{\bf 1}_{\{t\geq \e^2t_b\}}\Bigg( \left(1-\dfrac{1}{j}\right)\pp[v_{j,\l}^{k-1}]\Big(t-\e^2t_b,\vx-(\e\vw)
t_b\Big)\ue^{-(1+\l)t_b}\\
&&+\int_{0}^{t_b}\bar v_{j,\l}^{k-1}\Big(t-\e^2(t_b-s),\vx-\e(t_b-s)\vw,\vw\Big)\ue^{-(1+\l)(t_b-s)}\ud{s}\Bigg)\no\\
&&+{\bf 1}_{\{t\leq \e^2t_b\}}\Bigg(\int_{0}^{\frac{t}{\e^2}}\bar
v_{j,\l}^{k-1}\left(\e^2s,\vx-\e\left(\frac{t}{\e^2}-s\right)\vw,\vw\right)\ue^{-(1+\l)(\frac{t}{\e^2}-s)}\ud{s}\Bigg)\no.
\end{eqnarray}
Since $\im{\bar
v^k_{j,\l}}{[0,\infty)\times\Omega\times\s^1}\leq\im{v^k_{j,\l}}{[0,\infty)\times\Omega\times\s^1}$ and $\im{\pp
[v^k_{j,\l}]}{[0,\infty)\times\Gamma^+}\leq \im{v^k_{j,\l}}{[0,\infty)\times\Omega\times\s^1}$, we can directly estimate
\begin{eqnarray}
&&\im{v^{k}_{j,\l}}{[0,\infty)\times\Omega\times\s^1}\\
&\leq&{\bf 1}_{\{t\geq \e^2t_b\}}\Bigg( \left(1-\dfrac{1}{j}\right)\ue^{-(1+\l)t_b}\im{v^{k-1}_{j,\l}}{[0,\infty)\times\Omega\times\s^1}
+\im{v^{k-1}_{j,\l}}{[0,\infty)\times\Omega\times\s^1}\int_{0}^{t_b}\ue^{-(1+\l)(t_b-s)}\ud{s}\Bigg)\no\\
&&+{\bf 1}_{\{t\leq \e^2t_b\}}\Bigg(\im{v^{k-1}_{j,\l}}{[0,\infty)\times\Omega\times\s^1}\int_{0}^{\frac{t}{\e^2}}\ue^{-(1+\l)(\frac{t}{\e^2}-s)}\ud{s}\Bigg)\no,\\
&\leq&\left(1-\dfrac{1}{j}\right)\ue^{-(1+\l)t_b}\im{v^{k-1}_{j,\l}}{[0,\infty)\times\Omega\times\s^1}
+\frac{1}{1+\l}\bigg(1-\ue^{-(1+\l)t_b}\bigg)\im{v^{k-1}_{j,\l}}{[0,\infty)\times\Omega\times\s^1}\no\\
&\leq&\left(1-\dfrac{1}{j}\right)\im{v^{k-1}_{j,\l}}{[0,\infty)\times\Omega\times\s^1},\no
\end{eqnarray}
since $j\geq\dfrac{2}{\l}$.
Thus, this is a contractive iteration. Considering $v^1_{j,\l}=u^1_{j,\l}$, we
have
\begin{eqnarray}
\im{v^{k}_{j,\l}}{[0,\infty)\times\Omega\times\s^1}\leq\left(1-\dfrac{1}{j}\right)^{k-1}\im{u^{1}_{j,\l}}{[0,\infty)\times\Omega\times\s^1}.
\end{eqnarray}
for $k\geq1$. Therefore, $u^k_{j,\l}$ converges strongly in $L^{\infty}$ to
the limiting solution $u_{j,\l}$ satisfying
\begin{eqnarray}\label{well-posedness temp 1}
\im{u_{j,\l}}{[0,\infty)\times\Omega\times\s^1}\leq\sum_{k=1}^{\infty}\im{v^{k}_{j,\l}}{[0,\infty)\times\Omega\times\s^1}\leq j\im{u^1_{j,\l}}{[0,\infty)\times\Omega\times\s^1}.
\end{eqnarray}
Since $u^1_{j,\l}$ can be expressed as
\begin{eqnarray}
u_{j,\l}^1(t,\vx,\vw)
&=&{\bf 1}_{\{t\geq \e^2t_b\}}\bigg(\g\Big(t-\e^2t_b,\vx-(\e\vw)
t_b,\vw\Big)\ue^{-(1+\l)t_b}\\
&&+\int_{0}^{t_b}\ss\Big(t-\e^2(t_b-s),\vx-\e(t_b-s)\vw,\vw\Big)\ue^{-(1+\l)(t_b-s)}\ud{s}\bigg)\no\\
&&+{\bf 1}_{\{t\leq \e^2t_b\}}\bigg(\h\left(\vx-(\e\vw)\frac{
t}{\e^2},\vw\right)\ue^{-(1+\l)\frac{t}{\e^2}}\no\\
&&+\int_{0}^{\frac{t}{\e^2}}\ss\left(\e^2s,\vx-\e\left(\frac{t}{\e^2}-s\right)\vw,\vw\right)\ue^{-(1+\l)(\frac{t}{\e^2}-s)}\ud{s}\bigg)\no,
\end{eqnarray}
Based on Lemma \ref{well-posedness lemma 1}, we can directly
estimate
\begin{eqnarray}\label{well-posedness temp 2}
\im{u_{j,\l}^1}{[0,\infty)\times\Omega\times\s^1}\leq
\im{\g}{[0,\infty)\times\Gamma^-}+\im{\h}{\Omega\times\s^1}+\im{\ss}{[0,\infty)\times\Omega\times\s^1}.
\end{eqnarray}
Combining (\ref{well-posedness temp 1}) and (\ref{well-posedness
temp 2}), we can naturally obtain the estimate
\begin{eqnarray}
\im{u_{j,\l}}{[0,\infty)\times\Omega\times\s^1}\leq j\bigg(\im{\g}{[0,\infty)\times\Gamma^-}+\im{\h}{\Omega\times\s^1}+\im{\ss}{[0,\infty)\times\Omega\times\s^1}\bigg).
\end{eqnarray}
However, this estimate is not uniform in $j$, so we cannot directly take limit $j\rt\infty$.\\
\ \\
Step 2: Energy Estimate of $u_{j,\l}$.\\
Multiplying $u_{j,\l}$ on both sides of (\ref{penalty neutron}) and integrating over $[0,t]\times\Omega\times\s^1$, by Lemma \ref{well-posedness lemma 3}, we get the energy estimate
\begin{eqnarray}
\frac{\e^2}{2}\nm{u_{j,\l}(t)}_{L^2(\Omega\times\s^1)}^2-\frac{\e^2}{2}\nm{u_{j,\l}(0)}_{L^2(\Omega\times\s^1)}^2+\frac{\e}{2}\iint_{[0,t)\times\Gamma}\abs{u_{j,\l}}^2\ud{\gamma}\\
+\l\nm{u_{j,\l}}_{L^2([0,t)\times\Omega\times\s^1)}^2+\nm{u_{j,\l}-\bar
u_{j,\l}}_{L^2([0,t)\times\Omega\times\s^1)}^2&=&\iint_{[0,t)\times\Omega\times\s^1}\ss u_{j,\l}.\no
\end{eqnarray}
A direct computation shows
\begin{eqnarray}
\frac{\e}{2}\int_{[0,t)\times\Gamma}\abs{u_{j,\l}}^2\ud{\gamma}
&=&\frac{\e}{2}\nm{u_{j,\l}}^2_{L^2([0,t)\times\Gamma^+)}-\frac{\e}{2}\nm{\left(1-\dfrac{1}{j}\right)\pp[u_{j,\l}]
+\g}^2_{L^2([0,t)\times\Gamma^-)}\\
&=&\frac{\e}{2}\bigg(\nm{u_{j,\l}}^2_{L^2([0,t)\times\Gamma^+)}-\nm{\left(1-\dfrac{1}{j}\right)\pp[u_{j,\l}]}^2_{L^2([0,t)\times\Gamma^+)}\bigg)\no\\
&&-\frac{\e}{2}\nm{\g}_{L^2([0,t)\times\Gamma^-)}^2-\e\left(1-\dfrac{1}{j}\right)\int_{[0,t)\times\Gamma^-}\g\pp[u_{j,\l}]\abs{\vw\cdot\vn}\ud{\gamma}\ud{t}.\nonumber
\end{eqnarray}
Hence, we have
\begin{eqnarray}
&&\frac{\e^2}{2}\nm{u_{j,\l}(t)}_{L^2(\Omega\times\s^1)}^2+\frac{\e}{2}\bigg(\nm{u_{j,\l}}^2_{L^2([0,t)\times\Gamma^+)}
-\nm{\left(1-\dfrac{1}{j}\right)\pp[u_{j,\l}]}^2_{L^2([0,t)\times\Gamma^+)}\bigg)\\
&&+\l\nm{u_{j,\l}}_{L^2([0,t)\times\Omega\times\s^1)}^2+\nm{u_{j,\l}-\bar
u_{j,\l}}_{L^2([0,t)\times\Omega\times\s^1)}^2\no\\
&=&\iint_{[0,t)\times\Omega\times\s^1}\ss u_{j,\l}+\frac{\e^2}{2}\nm{\h}_{L^2(\Omega\times\s^1)}^2+\frac{\e}{2}\nm{\g}_{L^2([0,t)\times\Gamma^-)}^2
+\e\left(1-\dfrac{1}{j}\right)\int_{[0,t)\times\Gamma^-}\g\pp[u_{j,\l}]\abs{\vw\cdot\vn}\ud{\gamma}\ud{t}.\no
\end{eqnarray}
Noting the fact that
\begin{eqnarray}
\e\bigg(\nm{u_{j,\l}}^2_{L^2([0,t)\times\Gamma^+)}-\nm{\pp[u_{j,\l}]}^2_{L^2([0,t)\times\Gamma^+)}\bigg)=\e\nm{(1-\pp)[u_{j,\l}]}^2_{L^2([0,t)\times\Gamma^+)},
\end{eqnarray}
we deduce
\begin{eqnarray}
\\
&&\frac{\e^2}{2}\nm{u_{j,\l}(t)}_{L^2(\Omega\times\s^1)}^2+\frac{\e}{2}\nm{(1-\pp)[u_{j,\l}]}^2_{L^2([0,t)\times\Gamma^+)}
+\l\nm{u_{j,\l}}_{L^2([0,t)\times\Omega\times\s^1)}^2+\nm{u_{j,\l}-\bar
u_{j,\l}}_{L^2([0,t)\times\Omega\times\s^1)}^2\no\\
&\leq&\iint_{[0,t)\times\Omega\times\s^1}\ss u_{j,\l}+\frac{\e^2}{2}\nm{\h}_{L^2(\Omega\times\s^1)}^2+\frac{\e}{2}\nm{\g}_{L^2([0,t)\times\Gamma^-)}^2
+\e\int_{[0,t)\times\Gamma^-}\g\pp[u_{j,\l}]\abs{\vw\cdot\vn}\ud{\gamma}\ud{t}.\no
\end{eqnarray}
Applying Cauchy's inequality, we obtain that for $\eta>0$,
\begin{eqnarray}\label{wt 02}
\\
&&\frac{\e^2}{2}\nm{u_{j,\l}(t)}_{L^2(\Omega\times\s^1)}^2+\frac{\e}{2}\nm{(1-\pp)[u_{j,\l}]}^2_{L^2([0,t)\times\Gamma^+)}
+\l\nm{u_{j,\l}}_{L^2([0,t)\times\Omega\times\s^1)}^2+\nm{u_{j,\l}-\bar
u_{j,\l}}_{L^2([0,t)\times\Omega\times\s^1)}^2\no\\
&\leq&\iint_{[0,t)\times\Omega\times\s^1}\ss u_{j,\l}+\frac{\e^2}{2}\nm{\h}_{L^2(\Omega\times\s^1)}^2+\bigg(1+\frac{4}{\eta}\bigg)\nm{\g}_{L^2([0,t)\times\Gamma^-)}^2
+\e^2\eta\nm{\pp[u_{j,\l}]}^2_{L^2([0,t)\times\Gamma^+)}.\no
\end{eqnarray}
Now the only difficulty is $\e^2\eta\nm{\pp[u_{j,\l}]}^2_{L^2([0,t)\times\Gamma^+)}$, which we cannot bound directly.\\
\ \\
Step 3: Estimate of $\nm{\pp[u_{j,\l}]}^2_{L^2([0,t)\times\Gamma^+)}$.\\
Multiplying $u_{j,\l}$ on both sides of (\ref{penalty neutron}), we have
\begin{eqnarray}\label{wt 03}
\frac{\e^2}{2}\dt(u_{j,\l}^2)+\frac{\e}{2}\vw\cdot\nx(u_{j,\l}^2)=-\l u_{j,\l}^2-u_{j,\l}(u_{j,\l}-\bar u_{j,\l})+fu_{j,\l}.
\end{eqnarray}
Taking absolute value on both sides of (\ref{wt 03}) and integrating over $[0,t]\times\Omega\times\s^1$, we get
\begin{eqnarray}
&&\nm{(\e\dt+\vw\cdot\nx)(u_{j,\l}^2)}_{L^1([0,t]\times\Omega\times\s^1)}\\
&\leq& \frac{2\l}{\e}\nm{u_{j,\l}}_{L^2([0,t]\times\Omega\times\s^1)}^2+\frac{2}{\e}\nm{u_{j,\l}-\bar
u_{j,\l}}_{L^2([0,t]\times\Omega\times\s^1)}^2+\frac{2}{\e}\iint_{[0,t]\times\Omega\times\s^1}\ss u_{j,\l}.\no
\end{eqnarray}
Based on (\ref{wt 02}), we can further obtain
\begin{eqnarray}
\nm{(\e\dt+\vw\cdot\nx)(u_{j,\l}^2)}_{L^1([0,t]\times\Omega\times\s^1)}&\leq& \frac{1}{\e}\bigg(1+\frac{4}{\eta}\bigg)\nm{\g}_{L^2([0,t]\times\Gamma^-)}^2+\e\nm{\h}_{L^2(\Omega\times\s^1)}^2\\
&&+\e\eta\nm{\pp[u_{j,\l}]}^2_{L^2([0,t]\times\Gamma^+)}+\frac{4}{\e}\iint_{[0,t]\times\Omega\times\s^1}\ss u_{j,\l}.\no
\end{eqnarray}
Hence, by Lemma \ref{well-posedness lemma 2} and (\ref{wt 02}), we know for given $\d>0$
\begin{eqnarray}\label{wt 04}
&&\nm{u_{j,\l}^2{\bf{1}}_{\{\Gamma^{\pm}\backslash\Gamma_{\pm}^{\delta}\}}}_{L^1([0,t]\times\Gamma^{\pm})}\\
&\leq&
C(\delta)\bigg(\e\nm{\h}_{L^2(\Omega\times\s^1)}^2+\nm{u_{j,\l}}_{L^2([0,t]\times\Omega\times\s^1)}^2
+\nm{(\e\dt+\vw\cdot\nx)(u_{j,\l}^2)}_{L^1([0,t]\times\Omega\times\s^1)}\bigg)\no\\
&\leq&C(\delta)\Bigg(\left(\frac{1}{\e}+\frac{1}{\l}\right)\e^2\nm{\h}_{L^2(\Omega\times\s^1)}^2
+\left(\frac{1}{\e}+\frac{1}{\l}\right)\bigg(1+\frac{1}{\eta}\bigg)\nm{\g}_{L^2([0,t]\times\Gamma^-)}^2\no\\
&&+\left(\frac{1}{\e}+\frac{1}{\l}\right)\e^2\eta\nm{\pp[u_{j,\l}]}^2_{L^2([0,t]\times\Gamma^+)}
+\left(\frac{1}{\e}+\frac{1}{\l}\right)\iint_{[0,t]\times\Omega\times\s^1}fu_{j,\l}.\Bigg).\no
\end{eqnarray}
Noting the fact that
\begin{eqnarray}
\nm{\pp[u_{j,\l}{\bf{1}}_{\{\Gamma_{\pm}\backslash\Gamma_{\pm}^{\delta}\}}]}_{L^2([0,t]\times\Gamma^+)}\leq \nm{
u_{j,\l}{\bf{1}}_{\{\Gamma_{\pm}\backslash\Gamma_{\pm}^{\delta}\}}}_{L^2([0,t]\times\Gamma^+)},
\end{eqnarray}
and for $\d$ sufficiently small, we have
\begin{eqnarray}
\nm{\pp[u_{j,\l}{\bf{1}}_{\{\Gamma_{\pm}\backslash\Gamma_{\pm}^{\delta}\}}]}_{L^2([0,t]\times\Gamma^+)}\geq \half\nm{\pp[u_{j,\l}]}_{L^2([0,t]\times\Gamma^+)}.
\end{eqnarray}
Combining with (\ref{wt 04}), we naturally obtain
\begin{eqnarray}
\nm{\pp[u_{j,\l}]}_{L^2([0,t]\times\Gamma^+)}^2&\leq&2\nm{\pp[u_{j,\l}{\bf{1}}_{\{\Gamma_{\pm}\backslash\Gamma_{\pm}^{\delta}\}}]}_{L^2([0,t]\times\Gamma^+)}^2\leq2\nm{
u_{j,\l}{\bf{1}}_{\{\Gamma_{\pm}\backslash\Gamma_{\pm}^{\delta}\}}}_{L^2([0,t]\times\Gamma^+)}^2\\
&\leq&C(\delta)\Bigg(\left(\frac{1}{\e}+\frac{1}{\l}\right)\e^2\nm{\h}_{L^2(\Omega\times\s^1)}^2
+\left(\frac{1}{\e}+\frac{1}{\l}\right)\bigg(1+\frac{1}{\eta}\bigg)\nm{\g}_{L^2([0,t]\times\Gamma^-)}^2\no\\
&&+\left(\frac{1}{\e}+\frac{1}{\l}\right)\e^2\eta\nm{\pp[u_{j,\l}]}^2_{L^2([0,t]\times\Gamma^+)}
+\left(\frac{1}{\e}+\frac{1}{\l}\right)\iint_{[0,t]\times\Omega\times\s^1}fu_{j,\l}.\Bigg).\no
\end{eqnarray}
For fixed $\d$, taking $\eta>0$ sufficiently small, we obtain
\begin{eqnarray}\label{wt 07}
\nm{\pp[u_{j,\l}]}_{L^2([0,t]\times\Gamma^+)}^2
&\leq&C\left(\frac{1}{\e}+\frac{1}{\l}\right)\Bigg(\e^2\nm{\h}_{L^2(\Omega\times\s^1)}^2+\nm{\g}_{L^2([0,t]\times\Gamma^-)}^2
+\iint_{[0,t]\times\Omega\times\s^1}fu_{j,\l}.\Bigg).
\end{eqnarray}
\ \\
Step 4: Limit $j\rt\infty$.\\
Plugging (\ref{wt 07}) into (\ref{wt 02}), we deduce
\begin{eqnarray}\label{wt 08}
\\
&&\e^2\nm{u_{j,\l}(t)}_{L^2(\Omega\times\s^1)}^2+\e\nm{(1-\pp)[u_{j,\l}]}^2_{L^2([0,t)\times\Gamma^+)}+\l\nm{u_{j,\l}}_{L^2([0,t)\times\Omega\times\s^1)}^2+\nm{u_{j,\l}-\bar
u_{j,\l}}_{L^2([0,t)\times\Omega\times\s^1)}^2\no\\
&\leq&\frac{C}{\l}\bigg(\iint_{[0,t)\times\Omega\times\s^1}\ss u_{j,\l}+\e^2\nm{\h}_{L^2(\Omega\times\s^1)}^2+\nm{\g}_{L^2([0,t)\times\Gamma^-)}^2\bigg).\no
\end{eqnarray}
Applying Cauchy's inequality, we have for $C_0>0$ small,
\begin{eqnarray}
\frac{1}{\l}\iint_{[0,t)\times\Omega\times\s^1}\ss u_{j,\l}\leq \frac{1}{C_0\l^2}\nm{\ss}_{L^2([0,t)\times\Omega\times\s^1)}^2+C_0\nm{u_{j,\l}}_{L^2([0,t)\times\Omega\times\s^1)}^2.
\end{eqnarray}
Therefore, absorbing $\dfrac{C_0}{\l}\nm{u_{j,\l}}_{L^2([0,t)\times\Omega\times\s^1)}^2$ into the left-hand side, we obtain
\begin{eqnarray}
\\
&&\e^2\nm{u_{j,\l}(t)}_{L^2(\Omega\times\s^1)}^2+\e\nm{(1-\pp)[u_{j,\l}]}^2_{L^2([0,t)\times\Gamma^+)}+\l\nm{u_{j,\l}}_{L^2([0,t)\times\Omega\times\s^1)}^2+\nm{u_{j,\l}-\bar
u_{j,\l}}_{L^2([0,t)\times\Omega\times\s^1)}^2\no\\
&\leq&C\bigg(\frac{1}{\l^2}\nm{\ss}_{L^2([0,t)\times\Omega\times\s^1)}^2+\frac{\e^2}{\l}\nm{\h}_{L^2(\Omega\times\s^1)}^2+\frac{1}{\l}\nm{\g}_{L^2([0,t)\times\Gamma^-)}^2\bigg).\no
\end{eqnarray}
This is a uniform estimate in $j$. We may take weak limit $u_{j,\l}\rightharpoonup u_{\l}$ in $L^2([0,t)\times\Omega\times\s^1)$ as $j\rt\infty$. Then by the weak formulation and the weak lower semi-continuity of $L^2$ norms, there exists a solution $u_{\l}$ to the penalized equation
\begin{eqnarray}\label{penalty neutron.}
\left\{
\begin{array}{l}\displaystyle
\e^2\dt u_{\l}+\e\vw\cdot\nx u_{\l}+(1+\l)u_{\l}-\bar
u_{\l}=\ss(t,\vx,\vw)\ \ \text{for}\ \ (t,\vx,\vw)\in[0,\infty)\times\Omega\times\s^1,\\\rule{0ex}{1.0em}
u_{\l}(0,\vx,\vw)=h(\vx,\vw)\ \
\text{for}\ \ (\vx,\vw)\in\Omega\times\s^1\\\rule{0ex}{1.0em}
u_{\l}(\vx_0,\vw)-\pp[u_{\l}](\vx_0)=\g(t,\vx_0,\vw)\ \ \text{for}\ \ t\in[0,\infty),\ \
\vx_0\in\p\Omega\ \ \text{and}\ \ \vw\cdot\vn<0,
\end{array}
\right.
\end{eqnarray}
and satisfies the estimate
\begin{eqnarray}
\\
&&\e^2\nm{u_{\l}(t)}_{L^2(\Omega\times\s^1)}^2+\e\nm{(1-\pp)[u_{\l}]}^2_{L^2([0,t)\times\Gamma^+)}+\l\nm{u_{\l}}_{L^2([0,t)\times\Omega\times\s^1)}^2+\nm{u_{\l}-\bar
u_{j,\l}}_{L^2([0,t)\times\Omega\times\s^1)}^2\no\\
&\leq&C\bigg(\frac{1}{\l^2}\nm{\ss}_{L^2([0,t)\times\Omega\times\s^1)}^2+\frac{\e^2}{\l}\nm{\h}_{L^2(\Omega\times\s^1)}^2+\frac{1}{\l}\nm{\g}_{L^2([0,t)\times\Gamma^-)}^2\bigg).\no
\end{eqnarray}
However, this estimate still blows up when $\l\rt0$, so we need to find a uniform estimate in $\l$.\\
\ \\
Step 5: Kernel Estimate.\\
Applying Lemma \ref{well-posedness lemma 3} to the
equation (\ref{penalty neutron.}). Then for any
$\phi\in L^{2}([0,t)\times\Omega\times\s^1)$ satisfying
$\e\dt\phi+\vw\cdot\nx\phi\in L^2([0,t)\times\Omega\times\s^1)$
and $\phi\in L^{2}([0,t)\times\Gamma)$, we have
\begin{eqnarray}\label{well-posedness temp 4}
\\
&&
-\e^2\int_0^t\iint_{\Omega\times\s^1}\dt\phi
u_{\l}-\e\int_0^t\iint_{\Omega\times\s^1}(\vw\cdot\nx\phi)u_{\l}+\l\int_0^t\iint_{\Omega\times\s^1}u_{\l}\phi+\int_0^t\iint_{\Omega\times\s^1}(u_{\l}-\bar
u_{\l})\phi\no\\
&=&-\e^2\iint_{\Omega\times\s^1}u_{\l}(t)\phi(t)+
\e^2\iint_{\Omega\times\s^1}u_{\l}(0)\phi(0)-\e\int_0^t\int_{\Gamma}u_{\l}\phi\ud{\gamma}+\int_0^t\iint_{\Omega\times\s^1}\ss\phi.\no
\end{eqnarray}
Our goal is to choose a particular test function $\phi$. We first
construct an auxiliary function $\zeta(t)$. Since $u_{\l}(t)\in
L^{2}(\Omega\times\s^1)$, it naturally implies $\bar
u_{\l}(t)\in L^{2}(\Omega)$. We define $\zeta(t,\vx)$ on $\Omega$
satisfying
\begin{eqnarray}\label{test temp 1}
\left\{
\begin{array}{rcl}
\Delta_x \zeta(t)&=&\bar u_{\l}(t,\vx)\ \ \text{in}\ \
\Omega,\\\rule{0ex}{1.0em} \zeta(t)&=&0\ \ \text{on}\ \ \p\Omega.
\end{array}
\right.
\end{eqnarray}
In the bounded domain $\Omega$, based on the standard elliptic
estimates, we have
\begin{eqnarray}\label{test temp 3}
\nm{\zeta(t)}_{H^2(\Omega)}\leq C\nm{\bar
u_{\l}(t)}_{L^2(\Omega)}.
\end{eqnarray}
We plug the test function
\begin{eqnarray}\label{test temp 2}
\phi(t)=-\vw\cdot\nx\zeta(t)
\end{eqnarray}
into the weak formulation (\ref{well-posedness temp 4}) and estimate
each term there. Naturally, we have
\begin{eqnarray}\label{test temp 4}
\nm{\phi(t)}_{L^2(\Omega)}\leq C\nm{\zeta(t)}_{H^1(\Omega)}\leq
C \nm{\bar u_{\l}(t)}_{L^2(\Omega)}.
\end{eqnarray}
Easily we can decompose
\begin{eqnarray}\label{test temp 5}
-\e\int_0^t\iint_{\Omega\times\s^1}(\vw\cdot\nx\phi)u_{\l}&=&-\e\int_0^t\iint_{\Omega\times\s^1}(\vw\cdot\nx\phi)\bar
u_{\l}-\e\int_0^t\iint_{\Omega\times\s^1}(\vw\cdot\nx\phi)(u_{\l}-\bar
u_{\l}).
\end{eqnarray}
We estimate the two term on the right-hand side of (\ref{test temp
5}) separately. By (\ref{test temp 1}) and (\ref{test temp 2}), we
have
\begin{eqnarray}\label{wellposed temp 1}
-\e\int_0^t\iint_{\Omega\times\s^1}(\vw\cdot\nx\phi)\bar
u_{\l}&=&\e\int_0^t\iint_{\Omega\times\s^1}\bar
u_{\l}\bigg(w_1(w_1\p_{11}\zeta+w_2\p_{12}\zeta)+w_2(w_1\p_{12}\zeta+w_2\p_{22}\zeta)\bigg)\\
&=&\e\int_0^t\iint_{\Omega\times\s^1}\bar
u_{\l}\bigg(w_1^2\p_{11}\zeta+w_2^2\p_{22}\zeta\bigg)\nonumber\\
&=&\e\pi\int_0^t\int_{\Omega}\bar u_{\l}(\p_{11}\zeta+\p_{22}\zeta)\nonumber\\
&=&\e\pi\nm{\bar u_{\l}}_{L^2([0,t]\times\Omega)}^2\nonumber\\
&=&\frac{\e}{2}\nm{\bar
u_{\l}}_{L^2([0,t]\times\Omega\times\s^1)}^2\nonumber.
\end{eqnarray}
In the second equality, above cross terms vanish due to the symmetry
of the integral over $\s^1$. On the other hand, for the second term
in (\ref{test temp 5}), H\"older's inequality and the elliptic
estimate imply
\begin{eqnarray}\label{wellposed temp 2}
-\e\int_0^t\iint_{\Omega\times\s^1}(\vw\cdot\nx\phi)(u_{\l}-\bar
u_{\l})&\leq&C \e\nm{u_{\l}-\bar u_{\l}}_{L^2([0,t]\times\Omega\times\s^1)}\bigg(\int_0^t\nm{\zeta}^2_{H^2(\Omega)}\bigg)^{1/2}\\
&\leq&C \e\nm{u_{\l}-\bar
u_{\l}}_{L^2([0,t]\times\Omega\times\s^1)}\nm{\bar
u_{\l}}_{L^2([0,t]\times\Omega\times\s^1)}\nonumber.
\end{eqnarray}
Based on (\ref{test temp 4}), the trace theorem and
H\"older's inequality, we have
\begin{eqnarray}\label{wellposed temp 3}
\e\int_0^t\int_{\Gamma}u_{\l}\phi\ud{\gamma}
&=&\e\int_0^t\int_{\Gamma}\pp[u_{\l}]\phi\ud{\gamma}+\e\int_0^t\int_{\Gamma^-}\g\phi\ud{\gamma}+\e\int_0^t\int_{\Gamma^+}(1-\pp)[u_{\l}]\phi\ud{\gamma}
\\
&=&\e\int_0^t\int_{\Gamma^+}(1-\pp)[u_{\l}]\phi\ud{\gamma}
+\e\int_0^t\int_{\Gamma^-}\g\phi\ud{\gamma}\no\\
&\leq&\e\nm{\phi}_{L^2([0,t]\times\Gamma)}\bigg(\nm{(1-\pp)[u_{\l}]}_{L^2([0,t]\times\Gamma^+)}+\nm{\g}_{L^2([0,t]\times\Gamma^-)}\bigg)\no\\
&\leq&\e\nm{\bar u_{\l}}_{L^2([0,t]\times\Omega\times\s^1)}\bigg(\nm{(1-\pp)[u_{\l}]}_{L^2([0,t]\times\Gamma^+)}+\nm{\g}_{L^2([0,t]\times\Gamma^-)}\bigg).\no
\end{eqnarray}
Also, we obtain
\begin{eqnarray}\label{wellposed temp 4}
\l\int_0^t\iint_{\Omega\times\s^1}u_{\l}\phi&=&\l\int_0^t\iint_{\Omega\times\s^1}\bar
u_{\l}\phi+\l\int_0^t\iint_{\Omega\times\s^1}(u_{\l}-\bar u_{\l})\phi=\l\int_0^t\iint_{\Omega\times\s^1}(u_{\l}-\bar u_{\l})\phi\\
&\leq&C \l\nm{\bar
u_{\l}}_{L^2([0,t]\times\Omega\times\s^1)}\nm{u_{\l}-\bar
u_{\l}}_{L^2([0,t]\times\Omega\times\s^1)}\nonumber,
\end{eqnarray}
\begin{eqnarray}\label{wellposed temp 5}
\int_0^t\iint_{\Omega\times\s^1}(u_{\l}-\bar u_{\l})\phi\leq
C \nm{\bar
u_{\l}}_{L^2([0,t]\times\Omega\times\s^1)}\nm{u_{\l}-\bar
u_{\l}}_{L^2([0,t]\times\Omega\times\s^1)},
\end{eqnarray}
\begin{eqnarray}\label{wellposed temp 6}
\int_0^t\iint_{\Omega\times\s^1}\ss\phi\leq C \nm{\bar
u_{\l}}_{L^2([0,t]\times\Omega\times\s^1)}\nm{\ss}_{L^2([0,t]\times\Omega\times\s^1)}.
\end{eqnarray}
On the other hand, we may directly estimate
\begin{eqnarray}\label{wellposed temp 7}
\e^2\iint_{\Omega\times\s^1}u_{\l}(t)\phi(t)&\leq& C\e^2\nm{\phi(t)}_{L^2(\Omega\times\s^1)}\nm{u_{\l}(t)}_{L^2(\Omega\times\s^1)}\\
&\leq& C\e^2\nm{\bar
u_{\l}(t)}_{L^2(\Omega\times\s^1)}\nm{u_{\l}(t)}_{L^2(\Omega\times\s^1)}\leq C\e^2\nm{u_{\l}(t)}^2_{L^2(\Omega\times\s^1)}.\no
\end{eqnarray}
Similarly, we know
\begin{eqnarray}
\e^2\iint_{\Omega\times\s^1}u_{\l}(0)\phi(0)\leq  C\e^2\nm{h}^2_{L^2(\Omega\times\s^1)}.
\end{eqnarray}
Then the only remaining term is
\begin{eqnarray}\label{wellposed temp 8}
-\e^2\int_0^t\iint_{\Omega\times\s^1}\dt\phi u_{\l}&=&-\e^2\int_0^t\iint_{\Omega\times\s^1}\dt\phi (u_{\l}-\bar u_{\l})\\
&\leq&\e^2\nm{\dt\nabla\zeta}_{L^2([0,t]\times\Omega\times\s^1)}\nm{u_{\l}-\bar
u_{\l}}_{L^2([0,t]\times\Omega\times\s^1)}.\no
\end{eqnarray}
Now we have to tackle $\nm{\dt\nabla\zeta}_{L^2([0,t]\times\Omega\times\s^1)}$.
For test function $\phi(\vx,\vw)$ which is independent of time $t$,
in time interval $[t-\delta,t]$ the weak formulation in
(\ref{well-posedness temp 4}) can be simplified as
\begin{eqnarray}\label{test temp 6}
&&\e^2\iint_{\Omega\times\s^1}u_{\l}(t)\phi-\e^2\iint_{\Omega\times\s^1}u_{\l}({t-\delta})\phi\\
&&
-\e\int_{t-\delta}^t\iint_{\Omega\times\s^1}(\vw\cdot\nx\phi)u_{\l}+\l\int_{t-\delta}^t\iint_{\Omega\times\s^1}u_{\l}\phi+\int_{t-\delta}^t\iint_{\Omega\times\s^1}(u_{\l}-\bar
u_{\l})\phi\no\\
&=&-\e\int_{t-\delta}^t\int_{\Gamma}u_{\l}\phi\ud{\gamma}
+\int_{t-\delta}^t\iint_{\Omega\times\s^1}\ss\phi.\no
\end{eqnarray}
Taking difference quotient as $\delta\rt0$, we know
\begin{eqnarray}
\frac{\e^2\displaystyle\iint_{\Omega\times\s^1}u_{\l}(t)\phi-\e^2\displaystyle\iint_{\Omega\times\s^1}u_{\l}({t-\delta})\phi}{\delta}\rt
\e^2\iint_{\Omega\times\s^1}\dt
u_{\l}(t)\phi.
\end{eqnarray}
Then (\ref{test temp 6}) can be simplified into
\begin{eqnarray}\label{test temp 7}
\e^2\iint_{\Omega\times\s^1}\dt u_{\l}(t)\phi
&=&-\l\iint_{\Omega\times\s^1}u_{\l}(t)\phi
+\e\iint_{\Omega\times\s^1}(\vw\cdot\nx\phi)u_{\l}(t)-\iint_{\Omega\times\s^1}(u_{\l}-\bar
u_{\l})(t)\phi\\
&&-\e\int_{\Gamma}u_{\l}(t)\phi\ud{\gamma}+\iint_{\Omega\times\s^1}f(t)\phi.\no
\end{eqnarray}
For fixed $t$, taking $\phi=-\Phi(\vx)$ which satisfies
\begin{eqnarray}
\left\{
\begin{array}{rcl}
\Delta \Phi&=&\dt\bar u(t,\vx)\ \ \text{in}\ \
\Omega,\\\rule{0ex}{1.0em} \Phi(t)&=&0\ \ \text{on}\ \ \p\Omega,
\end{array}
\right.
\end{eqnarray}
which further implies $\Phi=\dt\zeta$.
Then the left-hand side of (\ref{test temp 7}) is actually
\begin{eqnarray}
LHS&=&-\e^2\iint_{\Omega\times\s^1}\Phi\dt u_{\l}(t)=-\e^2\iint_{\Omega\times\s^1}\Phi\dt\bar u_{\l}\\
&=&-\e^2\iint_{\Omega\times\s^1}\Phi\Delta\Phi=\e^2\iint_{\Omega\times\s^1}\abs{\nabla\Phi}^2\no\\
&=&\e^2\nm{\dt\nabla\zeta(t)}_{L^2(\Omega\times\s^1)}^2.\no
\end{eqnarray}
By a similar argument as above and the Poincar\'e inequality, the right-hand side of
(\ref{test temp 7}) can be bounded as
\begin{eqnarray}
RHS&\leq& \nm{\dt\nabla\zeta(t)}_{L^2(\Omega\times\s^1)}\bigg(\nm{u_{\l}(t)-\bar
u_{\l}(t)}_{L^2(\Omega\times\s^1)}+\l\nm{\bar
u_{\l}(t)}_{L^2(\Omega\times\s^1)}
+\nm{\ss(t)}_{L^2(\Omega\times\s^1)}\bigg).
\end{eqnarray}
Note that the boundary terms vanish due to the construction of $\Phi$. Therefore, we have
\begin{eqnarray}
\e^2\nm{\dt\nabla\zeta(t)}_{L^2(\Omega\times\s^1)}&\leq&
\nm{u_{\l}(t)-\bar u_{\l}(t)}_{L^2(\Omega\times\s^1)}+\l\nm{\bar
u_{\l}(t)}_{L^2(\Omega\times\s^1)}
+\nm{\ss(t)}_{L^2(\Omega\times\s^1)}.
\end{eqnarray}
For all $t$, we can further integrate over $[0,t]$ to obtain
\begin{eqnarray}\label{wellposed temp 9}
\e^2\nm{\dt\nabla\zeta}_{L^2([0,t]\times\Omega\times\s^1)}
&\leq& \nm{u_{\l}-\bar
u_{\l}}_{L^2([0,t]\times\Omega\times\s^1)}+\l\nm{\bar
u_{\l}}_{L^2([0,t]\times\Omega\times\s^1)}
+\nm{\ss}_{L^2([0,t]\times\Omega\times\s^1)}.
\end{eqnarray}
Collecting terms in (\ref{wellposed temp 1}), (\ref{wellposed temp
2}), (\ref{wellposed temp 3}), (\ref{wellposed temp 4}),
(\ref{wellposed temp 5}), (\ref{wellposed temp 6}), (\ref{wellposed
temp 7}), (\ref{wellposed temp 8}), and (\ref{wellposed temp 9}), and using Cauchy's inequality, we
obtain
\begin{eqnarray}\label{well-posedness temp 3}
\\
\e\nm{\bar u_{\l}}_{L^2([0,t]\times\Omega\times\s^1)}
&\leq& C \bigg(\nm{u_{\l}-\bar
u_{\l}}_{L^2([0,t]\times\Omega\times\s^1)}+\e\nm{(1-\pp)[u_{\l}]}_{L^2([0,t]\times\Gamma^+)}+\nm{\ss}_{L^2([0,t]\times\Omega\times\s^1)}\no\\
&&+\e\tm{\g}{[0,t]\times\Gamma^-}+\e^{\frac{3}{2}}\nm{u_{\l}(t)}_{L^2(\Omega\times\s^1)}+\e^{\frac{3}{2}}\nm{h}_{L^2(\Omega\times\s^1)}\bigg).\no
\end{eqnarray}
When $0<\l<1$ and $0<\e<1$, we get the desired uniform estimate
with respect to $\lambda$.\\
\ \\
Step 6: Limit $\l\rt0$.\\
In the weak formulation (\ref{well-posedness temp 4}), we may take
the test function $\phi=u_{\l}$ to get the energy estimate
\begin{eqnarray}
\\
&&\frac{\e^2}{2}\nm{u_{\l}(t)}_{L^2(\Omega\times\s^1)}^2-\frac{\e^2}{2}\nm{h}_{L^2(\Omega\times\s^1)}^2
+\frac{\e}{2}\int_0^t\int_{\Gamma}\abs{u_{\l}}^2\ud{\gamma}+\l\nm{u_{\l}}_{L^2([0,t]\times\Omega\times\s^1)}^2+\nm{u_{\l}-\bar
u_{\l}}_{L^2([0,t]\times\Omega\times\s^1)}^2\no\\
&=&\int_0^t\iint_{\Omega\times\s^1}fu_{\l}.\no
\end{eqnarray}
Hence, this naturally implies
\begin{eqnarray}
\\
&&\frac{\e^2}{2}\nm{u_{\l}(t)}_{L^2(\Omega\times\s^1)}^2+\frac{\e}{2}\nm{(1-\pp)[u_{\l}]}^2_{L^2([0,t)\times\Gamma^+)}
+\l\nm{u_{\l}}_{L^2([0,t)\times\Omega\times\s^1)}^2+\nm{u_{\l}-\bar
u_{\l}}_{L^2([0,t)\times\Omega\times\s^1)}^2\no\\
&\leq&\iint_{[0,t)\times\Omega\times\s^1}\ss u_{\l}+\frac{\e^2}{2}\nm{\h}_{L^2(\Omega\times\s^1)}^2+\bigg(1+\frac{4}{\eta}\bigg)\nm{\g}_{L^2([0,t)\times\Gamma^-)}^2
+\e^2\eta\nm{\pp[u_{\l}]}^2_{L^2([0,t)\times\Gamma^+)}.\no
\end{eqnarray}
Also, as in Step 3, we know
\begin{eqnarray}
\\
\nm{\pp[u_{\l}]}_{L^2([0,t]\times\Gamma^+)}^2&\leq&C\bigg(\e\nm{\h}_{L^2(\Omega\times\s^1)}^2+\nm{u_{\l}}_{L^2([0,t]\times\Omega\times\s^1)}^2
+\nm{(\e\dt+\vw\cdot\nx)(u_{\l}^2)}_{L^1([0,t]\times\Omega\times\s^1)}\bigg)\no\\
&\leq&C\bigg(\frac{1}{\e}\iint_{[0,t)\times\Omega\times\s^1}\ss u_{\l}+\e\nm{\h}_{L^2(\Omega\times\s^1)}^2+\frac{1}{\e}\nm{\g}_{L^2([0,t)\times\Gamma^-)}^2\no\\
&&+\nm{u_{\l}}_{L^2([0,t)\times\Omega\times\s^1)}^2+\frac{1}{\e}\nm{u_{\l}-\bar u_{\l}}_{L^2([0,t)\times\Omega\times\s^1)}^2\bigg)\no
\end{eqnarray}
Note that here we keep $\nm{u_{\l}}_{L^2([0,t)\times\Omega\times\s^1)}^2$ on the right-hand side.
Then we have
\begin{eqnarray}\label{well-posedness temp 5}
\\
&&\e^2\nm{u_{\l}(t)}_{L^2(\Omega\times\s^1)}^2+\e\nm{(1-\pp)[u_{\l}]}^2_{L^2([0,t)\times\Gamma^+)}
+\l\nm{u_{\l}}_{L^2([0,t)\times\Omega\times\s^1)}^2+\nm{u_{\l}-\bar
u_{\l}}_{L^2([0,t)\times\Omega\times\s^1)}^2\no\\
&\leq&\iint_{[0,t)\times\Omega\times\s^1}\ss u_{\l}+\e^2\nm{\h}_{L^2(\Omega\times\s^1)}^2+\nm{\g}_{L^2([0,t)\times\Gamma^-)}^2
+\e^2\eta\nm{\bar u_{\l}}_{L^2([0,t)\times\Omega\times\s^1)}^2.\no
\end{eqnarray}
On the other hand, we can square on both sides of
(\ref{well-posedness temp 3}) to obtain
\begin{eqnarray}\label{well-posedness temp 6}
\\
\e^2\tm{\bar u_{\l}}{[0,t]\times\Omega\times\s^1}^2
&\leq& C \bigg( \tm{u_{\l}-\bar
u_{\l}}{[0,t]\times\Omega\times\s^1}^2+\tm{\ss}{[0,t]\times\Omega\times\s^1}^2+\e^2\tm{(1-\pp)[u_{\l}]}{[0,t]\times\Gamma^{+}}^2\no\\
&&+\e^2\tm{\g}{[0,t]\times\Gamma^-}^2+\e^3\tm{u_{\l}(t)}{\Omega\times\s^1}^2+\e^3\tm{\h}{\Omega\times\s^1}^2\bigg).\nonumber
\end{eqnarray}
Taking $\eta$ sufficiently small, multiplying a sufficiently small constant on both sides of
(\ref{well-posedness temp 6}) and adding it to (\ref{well-posedness temp 5}) to absorb $\e^2\nm{\bar u_{\l}}_{L^2([0,t)\times\Omega\times\s^1)}^2$, $\e^2\nm{(1-\pp)[u_{\l}]}_{L^2(\Gamma^+)}^2$,
$\e^3\tm{u_{\l}(t)}{\Omega\times\s^1}^2$ and $\nm{u_{\l}-\bar
u_{\l}}_{L^2([0,t]\times\Omega\times\s^1)}^2$, we deduce
\begin{eqnarray}
\\
&&\e^2\nm{u_{\l}(t)}_{L^2(\Omega\times\s^1)}^2+\e\nm{(1-\pp)[u_{\l}]}_{L^2([0,t]\times\Gamma^+)}^2+\e^2\nm{\bar
u_{\l}}_{L^2([0,t]\times\Omega\times\s^1)}^2+\nm{u_{\l}-\bar
u_{\l}}_{L^2([0,t]\times\Omega\times\s^1)}^2\no\\&\leq&
C \bigg(\tm{\ss}{[0,t]\times\Omega\times\s^1}^2+
\int_0^t\iint_{\Omega\times\s^1}\ss u_{\l}+\e^2\nm{\h}_{L^2(\Omega\times\s^1)}^2+\nm{\g}_{L^2([0,t]\times\Gamma^-)}^2\bigg).\nonumber
\end{eqnarray}
Hence, we have
\begin{eqnarray}\label{well-posedness temp 7}
&&\e^2\nm{u_{\l}(t)}_{L^2(\Omega\times\s^1)}^2+\e\nm{(1-\pp)[u_{\l}]}_{L^2([0,t]\times\Gamma^+)}^2+\e^2\nm{u_{\l}}_{L^2([0,t]\times\Omega\times\s^1)}^2\\
&\leq& C \bigg(\tm{\ss}{[0,t]\times\Omega\times\s^1}^2+
\int_0^t\iint_{\Omega\times\s^1}\ss u_{\l}+\e^2\nm{\h}_{L^2(\Omega\times\s^1)}^2+\nm{\g}_{L^2([0,t]\times\Gamma^-)}^2\bigg).\no
\end{eqnarray}
A simple application of Cauchy's inequality leads to
\begin{eqnarray}
\int_0^t\iint_{\Omega\times\s^1}fu_{\l}\leq\frac{1}{4C\e^2}\tm{f}{[0,t]\times\Omega\times\s^1}^2+C\e^2\tm{u_{\l}}{[0,t]\times\Omega\times\s^1}^2.
\end{eqnarray}
Taking $C$ sufficiently small, we can divide (\ref{well-posedness
temp 7}) by $\e^2$ to obtain
\begin{eqnarray}\label{well-posedness temp 21}
&&\nm{u_{\l}(t)}_{L^2(\Omega\times\s^1)}^2+\frac{1}{\e}\nm{(1-\pp)[u_{\l}]}_{L^2([0,t]\times\Gamma^+)}^2+\nm{u_{\l}}_{L^2([0,t]\times\Omega\times\s^1)}^2\\
&\leq& C \bigg(
\frac{1}{\e^4}\nm{\ss}_{L^2([0,t]\times\Omega\times\s^1)}^2+\nm{\h}_{L^2(\Omega\times\s^1)}^2+\frac{1}{\e^2}\nm{\g}_{L^2([0,t]\times\Gamma^-)}^2\bigg).\no
\end{eqnarray}
Since above estimate does not depend on $\l$, it gives a uniform
estimate for the penalized neutron transport equation
(\ref{penalty neutron.}). Thus, we can extract a
weakly convergent subsequence $u_{\l}\rt u$ as $\l\rt0$. The weak
lower semi-continuity of $L^2$ norms implies that $u$ also satisfies
the estimate (\ref{well-posedness temp 21}). Hence, in the weak
formulation (\ref{well-posedness temp 4}), we can take $\l\rt0$ to
deduce that $u$ satisfies equation (\ref{neutron}). Also $u_{\l}-u$
satisfies the equation
\begin{eqnarray}
\left\{
\begin{array}{l}
\e^2\dt(u_{\l}-u)+\e\vec w\cdot\nabla_x(u_{\l}-u)+(u_{\l}-u)-(\bar
u_{\l}-\bar u)=-\l u_{\l}\ \ \text{for}\ \
(t,\vx,\vw)\in[0,\infty)\times\Omega\times\s^1\label{remainder},\\\rule{0ex}{1.0em}
(u_{\l}-u)(0,\vx,\vw)=0\ \ \text{for}\ \ (\vx,\vw)\in\Omega\times\s^1,\\\rule{0ex}{1.0em}
(u_{\l}-u)(\vec x_0,\vec w)=0\ \ \text{for}\ \ t\in[0,\infty),\ \ \vx_0\in\p\Omega\ \ \text{and}\ \ \vw\cdot\vn<0.
\end{array}
\right.
\end{eqnarray}
By a similar argument as above, we can achieve
\begin{eqnarray}
\nm{u_{\l}-u}_{L^2([0,t]\times\Omega\times\s^1)}^2\leq
C \bigg(\frac{\l}{\e^4}\nm{u_{\l}}_{L^2([0,t]\times\Omega\times\s^1)}^2\bigg).
\end{eqnarray}
When $\l\rt0$, the right-hand side approaches zero, which implies
the convergence is actually in the strong sense. The uniqueness
easily follows from the energy estimates.

\end{proof}

\subsection{$L^{\infty}$ Estimate - First Round}

In this section, we will prove the $L^{\infty}$ well-posedness.
\begin{definition}(Stochastic Cycle)
For fixed point $(t,\vx,\vw)$ with $(\vx,\vw)\notin\Gamma^0$, let
$(t_0,\vx_0,\vw_0)=(0,\vx,\vw)$. For $\vw_{k+1}$ such that
$\vw_{k+1}\cdot\vn(\vx_{k+1})>0$, define the $(k+1)$-component of
the back-time cycle as
\begin{eqnarray}
(t_{k+1},\vx_{k+1},\vw_{k+1})=(t_k+t_b(\vx_k,\vw_k),\vx_b(\vx_k,\vw_k),\vw_{k+1})
\end{eqnarray}
where
\begin{eqnarray}
t_b(\vx,\vw)&=&\inf\{t>0:\vx-\e t\vw\notin\Omega\}\\
x_b(\vx,\vw)&=&\vx-\e t_b(\vx,\vw)\vw\notin\Omega
\end{eqnarray}
Set
\begin{eqnarray}
\xc(s;t,\vx,\vw)&=&\sum_{k}\id_{\{t_{k+1}\leq s<t_k\}}\bigg(\vx_k-\e(t_k-s)\vw_k\bigg)\\
\wc(s;t,\vx,\vw)&=&\sum_{k}\id_{\{t_{k+1}\leq s<t_k\}}\vw_k
\end{eqnarray}
Define $\mu_{k+1}=\{\vw\in \s^1:\vw\cdot\vn(\vx_{k+1})>0\}$, and
let the iterated integral for $k\geq2$ be defined as
\begin{eqnarray}
\int_{\prod_{k=1}^{k-1}\mu_j}\prod_{j=1}^{k-1}\ud{\sigma_j}=\int_{\mu_1}\ldots\bigg(\int_{\mu_{k-1}}\ud{\sigma_{k-1}}\bigg)\ldots\ud{\sigma_1}
\end{eqnarray}
where $\ud{\sigma_j}=\Big(\vn(\vx_j)\cdot\vw\Big)\ud{\vw}$ is a
probability measure.
\end{definition}
\begin{lemma}\label{well-posedness lemma 4}
For $T_0>0$ sufficiently large, there exists constants $C_1,C_2>0$
independent of $T_0$, such that for $k=C_1T_0^{\frac{5}{4}}$ and $t\in[0,\e^{-1} T_0]$,
\begin{eqnarray}
\int_{\prod_{j=1}^{k-1}\mu_j}{\bf{1}}_{\{t_k(t,\vx,\vw,\vw_1,\ldots,\vw_{k-1})>0\}}\prod_{j=1}^{k-1}\ud{\sigma_j}\leq
\bigg(\frac{1}{2}\bigg)^{C_2T_0^{\frac{5}{4}}}
\end{eqnarray}
\end{lemma}
\begin{proof}
See \cite[Lemma 4.1]{Esposito.Guo.Kim.Marra2013} and \cite[Lemma 3.13]{Esposito.Guo.Kim.Marra2015}.
\end{proof}
\begin{theorem}\label{LI estimate.}
Assume $\ss(t,\vx,\vw)\in
L^{\infty}([0,\infty)\times\Omega\times\s^1)$, $\h(\vx,\vw)\in
L^{\infty}(\Omega\times\s^1)$ and $\g(t,x_0,\vw)\in
L^{\infty}([0,\infty)\times\Gamma^-)$. Then the solution $u(t,\vx,\vw)$ to the neutron transport
equation (\ref{neutron}) satisfies
\begin{eqnarray}
\im{u}{[0,\infty)\times\Omega\times\s^1}&\leq& C \bigg(\frac{1}{\e^3}\nm{\ss}_{L^2([0,\infty)\times\Omega\times\s^1)}
+\im{\ss}{[0,\infty)\times\Omega\times\s^1}\\
&&+\frac{1}{\e}\nm{\h}_{L^2(\Omega\times\s^1)}+\im{\h}{[0,\infty)\times\Omega\times\s^1}\bigg)\no\\
&&+\frac{1}{\e^2}\nm{\g}_{L^2([0,\infty)\times\Gamma^-)}+\im{g}{[0,\infty)\times\Gamma^-}.\no
\end{eqnarray}
\end{theorem}
\begin{proof}
We divide the proof into several steps:\\
\ \\
Step 1: Mild formulation.\\
We rewrite the equation (\ref{neutron}) along the characteristics as
\begin{eqnarray}
u(t,\vx,\vw)
&=&{\bf 1}_{\{t\geq \e^2t_b\}}\bigg( \left(\g+\pp[u]\right)\Big(t-\e^2t_b,\vx-(\e\vw)
t_b,\vw\Big)\ue^{-t_b}\\
&&+\int_{0}^{t_b}(\bar u+\ss)\Big(t-\e^2(t_b-s),\vx-\e(t_b-s)\vw,\vw\Big)\ue^{-(t_b-s)}\ud{s}\bigg)\no\\
&&+{\bf 1}_{\{t\leq \e^2t_b\}}\bigg(\h\left(\vx-(\e\vw)\frac{
t}{\e^2},\vw\right)\ue^{-\frac{t}{\e^2}}\no\\
&&+\int_{0}^{\frac{t}{\e^2}}(\bar
u+\ss)\left(\e^2s,\vx-\e\left(\frac{t}{\e^2}-s\right)\vw,\vw\right)\ue^{-(\frac{t}{\e^2}-s)}\ud{s}\bigg)\no,
\end{eqnarray}
Note that here $\pp[u]$ is an integral over $\mu_1$ at $\vx_1$, we may rewrite it again along the characteristics to $\vx_2$. This process can continue to arbitrary $\vx_k$. Then we get
\begin{eqnarray}\label{wt 21}
u(t,\vx,\vw)&=&H\ue^{- t_1\wedge \frac{t}{\e^2}}+\sum_{l=1}^{k-1}\int_{\prod_{j=1}^l}G_l\ue^{- t_{l+1}\wedge(\frac{t}{\e^2}-\sum_{j=1}^{l}t_j)}\prod_{j=1}^l\ud{\sigma_j}\\
&&+{\bf 1}_{\{t\geq \e^2\sum_{j=1}^kt_j\}}\int_{\prod_{j=1}^{k-1}}\pp[u]\left(t-\e^2\sum_{l=1}^kt_l,\vx_k,\vw_{k-1}\right)\ue^{- \sum_{j=1}^kt_j}\prod_{j=1}^{k-1}\ud{\sigma_j}\no\\
&=&I+II+III.\no
\end{eqnarray}
where
\begin{eqnarray}
H&=&{\bf 1}_{\{t\geq \e^2t_1\}}\g\Big(t-\e^2t_1,\vx-(\e\vw)t_1,\vw\Big)+{\bf 1}_{\{t\leq \e^2t_1\}}\h\left(\vx-(\e\vw)\frac{t}{\e^2},\vw\right)\\
&&+\int_0^{t_1\wedge\frac{t}{\e^2}}\ss\Big(t-\e^2(t_1-s_1),\vx-\e(t_1-s_1)\vw,\vw\Big)\ue^{ s_1}\ud{s_1}\no\\
&&+\int_0^{t_1\wedge\frac{t}{\e^2}}\bar u\Big(t-\e^2(t_1-s_1),\vx-\e(t_1-s_1)\vw\Big)\ue^{ s_1}\ud{s_1},\no
\end{eqnarray}
and
\begin{eqnarray}
G_l&=&{\bf 1}_{\{t-\sum_{j=1}^l\e^2t_j\geq \e^2t_{l+1}\}}\g\left(t-\e^2\sum_{j=1}^lt_j,\vx_l-(\e\vw_l)t_{l+1},\vw_l\right)\\
&&+{\bf 1}_{\{t-\sum_{j=1}^l\e^2t_j\leq \e^2t_{l+1}\}}\h\left(\vx_l+\sum_{j=1}^lt_j\e\vw_{j}-(\e\vw_l)\frac{t}{\e^2},\vw_l\right)\no\\
&&+\int_0^{t_l\wedge(\frac{t}{\e^2}-\sum_{j=1}^{l}t_j)}\ss\Big(t-\sum_{j=1}^{l}\e^2t_j-\e^2(t_{l+1}-s_{l+1}),\vx_l-\e(t_{l+1}-s_{l+1})\vw_l,\vw_l\Big)\ue^{ s_{l+1}}\ud{s_{l+1}}\no\\
&&+\int_0^{t_l\wedge(\frac{t}{\e^2}-\sum_{j=1}^{l}t_j)}\bar u\Big(t-\sum_{j=1}^{l}\e^2t_j-\e^2(t_{l+1}-s_{l+1}),\vx_l-\e(t_{l+1}-s_{l+1})\vw_l\Big)\ue^{ s_{l+1}}\ud{s_{l+1}}.\no
\end{eqnarray}
We need to estimate each term on the right-hand side of (\ref{wt 21}).\\
\ \\
Step 2: Estimates in mild formulation.\\
We first consider $III$. We may decompose it as
\begin{eqnarray}
III&\leq&\int_{\prod_{j=1}^{k-1}}\pp[u]\left(t-\e^2\sum_{l=1}^kt_l,\vx_k,\vw_{k-1}\right)\ue^{- \sum_{j=1}^kt_j}\prod_{j=1}^{k-1}\ud{\sigma_j}\\
&=&\int_{\prod_{j=1}^{k-1}}{\bf{1}}_{\{\e^2\sum_{j=1}^kt_j<\e T_0\}}\pp[u]\left(t-\e^2\sum_{l=1}^kt_l,\vx_k,\vw_{k-1}\right)\ue^{- \sum_{j=1}^kt_j}\prod_{j=1}^{k-1}\ud{\sigma_j}\no\\
&&+\int_{\prod_{j=1}^{k-1}}{\bf{1}}_{\{\e^2\sum_{j=1}^kt_j>\e T_0\}}\pp[u]\left(t-\e^2\sum_{l=1}^kt_l,\vx_k,\vw_{k-1}\right)\ue^{- \sum_{j=1}^kt_j}\prod_{j=1}^{k-1}\ud{\sigma_j}\no\\
&=&III_1+III_2,\no
\end{eqnarray}
where $T_0>0$ is defined as in Lemma \ref{well-posedness lemma 4}.
Then we take $k=C_1T_0^{\frac{5}{4}}$. By Lemma \ref{well-posedness lemma 4}, we deduce
\begin{eqnarray}
\abs{III_1}&\leq&C\bigg(\frac{1}{2}\bigg)^{C_2T_0^{\frac{5}{4}}}\im{u}{[0,\infty)\times\Omega\times\s^1}.
\end{eqnarray}
Also, $\ds\e^2\sum_{j=1}^kt_j>\e^{-1} T_0$ implies that $\ds\sum_{j=1}^kt_j\geq\dfrac{T_0}{\e^3}\geq T_0$ for $0<\e<<1$, so we may directly estimate
\begin{eqnarray}
\abs{III_2}&\leq&C\ue^{- T_0}\im{u}{[0,\infty)\times\Omega\times\s^1}.
\end{eqnarray}
Therefore, for $T_0$ sufficiently large, we know
\begin{eqnarray}\label{wt 22}
\abs{III}\leq \delta\im{u}{[0,t]\times\Omega\times\s^1},
\end{eqnarray}
for some $\d>0$ small. On the other hand, we may directly estimate the terms in $I$ and $II$ related to $\g$, $\h$ and $\ss$, which we denote as $I_1$ and $II_1$. For fixed $T$, it is easy to see
\begin{eqnarray}\label{wt 23}
\abs{I_1}+\abs{II_1}\leq \im{\ss}{[0,t]\times\Omega\times\s^1}+\im{\g}{[0,t]\times\Gamma^-}+\im{\h}{\Omega\times\s^1}.
\end{eqnarray}
Hence, the remaining terms are all related to $\bar u$.\\
\ \\
Step 3: Estimate of $\bar u$ term.\\
Collecting the results in (\ref{wt 22}) and (\ref{wt 23}), we obtain
\begin{eqnarray}
\abs{u}&\leq&A+\abs{\int_0^{t_1}\bar u\Big(t-\e^2(t_1-s_1),\vx-\e(t_1-s_1)\vw\Big)\ue^{- (t_1-s_1)}\ud{s_1}}\\
&&+\abs{\sum_{l=1}^{k-1}\int_{\prod_{j=1}^l}\bigg(\int_0^{t_l}\bar u\Big(t-\e^2\sum_{j=1}^lt_l-\e^2(t_{l+1}-s_{l+1}),\vx_l-\e(t_{l+1}-s_{l+1}\Big)\vw_l)\ue^{- (t_{l+1}-s_{l+1})}\ud{s_{l+1}}\bigg)\prod_{j=1}^l\ud{\sigma_j}},\no\\
&=&A+I_2+II_2,\no
\end{eqnarray}
where
\begin{eqnarray}
A=\im{\ss}{[0,t]\times\Omega\times\s^1}+\im{\g}{[0,t]\times\Gamma^-}+\im{\h}{\Omega\times\s^1}+\delta\im{u}{[0,t]\times\Omega\times\s^1}.
\end{eqnarray}
By definition, we know
\begin{eqnarray}
\abs{I_2}&=&\abs{\int_0^{t_1}\bigg(\int_{\s^1}u(t-\e^2(t_1-s_1),\vx-\e(t_1-s_1)\vw,\vw_{s_1})\ud{\vw_{s_1}}\bigg)\ue^{- (t_1-s_1)}\ud{s_1}},
\end{eqnarray}
where $\vw_{s_1}\in\s^1$ is a dummy variable.
Then we can utilize the mild formulation (\ref{wt 21}) to rewrite $u(\vx-\e(t_1-s_1)\vw,\vw_{s_1})$ along the characteristics. We denote the stochastic cycle as $(t_k',\vx_k',\vw_k')$ correspondingly and $(t_0',\vx_0',\vw_0')=\Big(t-\e^2(t_1-s_1),\vx-\e(t_1-s_1)\vw,\vw_{s_1}\Big)$. Then
\begin{eqnarray}
\abs{I_2}&\leq& \abs{\int_0^{t_1}\bigg(\int_{\s^1}A\ud{\vw_{s_1}}\bigg)\ue^{- (t_1-s_1)}\ud{s_1}}\\
&&+\abs{\int_0^{t_1}\bigg(\int_{\s^1}\int_0^{t_1'}\bar u(t_0'-\e^2(t_1'-s_1'),\vx'_0-\e(t_1'-s_1')\vw_{s_1})\ue^{- (t_1'-s_1')}\ud{s_1'}\ud{\vw_{s_1}}\bigg)\ue^{- (t_1-s_1)}\ud{s_1}}\no\\
&&+\Bigg\vert\int_0^{t_1}\bigg(\int_{\s^1}\sum_{l'=1}^{k-1}\int_{\prod_{j'=1}^{l'}}\bigg(\int_0^{t_{l'}'}\bar u(t_0'-\e^2\sum_{j'=1}^{l'}t_{j'}'-\e^2(t_{l'+1}'-s_{l'+1}'),\vx_{l'}-\e(t_{l'+1}'-s_{l'+1}')\vw_{l'})\no\\
&&\ue^{- (t_{l'+1}'-s_{l'+1}')}\ud{s_{l'+1}'}\bigg)\prod_{j'=1}^{l'}
\ud{\sigma_{j'}}\ud{\vw_{s_1}}\bigg)\no\\
&&\ue^{- (t_1-s_1)}\ud{s_1}\Bigg\vert,\no\\
&=&\abs{I_{2,1}}+\abs{I_{2,2}}+\abs{I_{2,3}}.\no
\end{eqnarray}
It is obvious that
\begin{eqnarray}
\abs{I_{2,1}}&=&\abs{\int_0^{t_1}\bigg(\int_{\s^1}A\ud{\vw_{s_1}}\bigg)\ue^{- (t_1-s_1)}\ud{s_1}}\leq A\\
&\leq& \im{\ss}{[0,t]\times\Omega\times\s^1}+\im{\g}{[0,t]\times\Gamma^-}+\im{\h}{\Omega\times\s^1}+\delta\im{u}{[0,t]\times\Omega\times\s^1}.\no
\end{eqnarray}
For $I_{2,2}$, we use mild formulation again to rewrite $\bar u(t-\e^2(t_1-s_1),\vx'-\e(t_1'-s_1')\vw_{s_1})$. Denote the stochastic cycle as $(t_k'',\vx_k'',\vw_k'')$. For convenience, we only write out the key term for estimating as
\begin{eqnarray}
I_{2,2}&=&\int_0^{t_1}\int_{\s^1}\int_0^{t_1'}\int_{\s^1}\int_0^{t_1''}\bar u(t_0'-\e^2(t_1'-s_1'),\vx'_0-\e(t_1'-s_1')\vw_{s_1}-\e(t_1''-s_1'')\vw_{s_1'})\\
&&\ue^{- (t_1''-s_1'')}\ue^{- (t_1'-s_1')}\ue^{- (t_1-s_1)}\ud{s_1}\ud{s_1''}\ud{\vw_{s_1'}}\ud{s_1'}\ud{\vw_{s_1}}.\no
\end{eqnarray}
Then we decompose
\begin{eqnarray}
I_{2,2}&=&\int_0^{t_1}\int_{\s^1}\int_{t_1'-\d}^{t_1'}\int_{\s^1}\int_0^{t_1''}+\int_0^{t_1}\int_{\s^1}\int_0^{t_1'-\d}\int_{\s^1}\int_{t_1''-\d}^{t_1''}\\
&&+\int_0^{t_1}\int_{\s^1}\int_0^{t_1'-\d}\int_{\s^1}\int_0^{t_1''-d}{\bf{1}}_{\vw_{s_1}\cdot\vw_{s_1'}\leq\d}
+\int_0^{t_1}\int_{\s^1}\int_0^{t_1'-\d}\int_{\s^1}\int_0^{t_1''-\d}{\bf{1}}_{\vw_{s_1}\cdot\vw_{s_1'}\geq\d}\no\\
&=&I_{2,2,1}+I_{2,2,2}+I_{2,2,3}+I_{2,2,4}.\no
\end{eqnarray}
The first three terms are only restricted to small domains, so we can directly obtain
\begin{eqnarray}
\abs{I_{2,2,1}}+\abs{I_{2,2,2}}+\abs{I_{2,2,3}}\leq \delta\im{u}{[0,t]\times\Omega\times\s^1}.
\end{eqnarray}
We turn to the last and most difficult term $I_{2,2,4}$.
Note $\vw_{s_1},\vw_{s_1'}\in\s^1$, which are essentially one-dimensional
variables. Thus, we may write them in new variablea $\psi$ and $\phi$ as
$\vw_{s_1}=(\cos\psi,\sin\psi)$ and $\vw_{s_1'}=(\cos\phi,\sin\phi)$. Then we define the change of variable
$[-\pi,\pi)^2\rt \Omega: (\psi,\phi)\rt(y_1,y_2)=(\vx'-\e(t_1'-s_1')\vw_{s_1}-\e(t_1''-s_1'')\vw_{s_1'})$, i.e.
\begin{eqnarray}
\left\{
\begin{array}{rcl}
y_1&=&x_1-\e(t_1-s_1)w_1-\e(t_1'-s_1')\cos\psi-\e(t_1''-s_1'')\cos\phi,\\
y_2&=&x_2-\e(t_1-s_1)w_2-\e(t_1'-s_1')\sin\psi-\e(t_1''-s_1'')\sin\phi.
\end{array}
\right.
\end{eqnarray}
Therefore, we can directly compute the
Jacobian
\begin{eqnarray}
\\
\abs{\frac{\p{(y_1,y_2)}}{\p{(\psi,\phi)}}}=\abs{\abs{\begin{array}{cc}
\e(t_1'-s_1')\sin\psi&\e(t_1''-s_1'')\sin\phi\\
-\e(t_1'-s_1')\cos\psi&-\e(t_1''-s_1'')\sin\phi
\end{array}}}=\e^2(t_1'-s_1')(t_1''-s_1'')\Big(\vw_{s_1}\cdot\vw_{s_1'}\Big)\geq\e^2\d^3.\no
\end{eqnarray}
Hence, using H\"{o}lder's inequality, we have
\begin{eqnarray}
I_{2,2,2}&\leq&C\bigg(\int_{0}^{t}\int_{\Omega}\frac{1}{\e^2\delta^3}\abs{\bar u^2(s,\vec y)}\ud{\vec y}\ud{s}\bigg)^{\frac{1}{2}}.
\end{eqnarray}
Therefore, we have shown
\begin{eqnarray}
\abs{I_{2,2}}\leq \delta\im{u}{[0,t]\times\Omega\times\s^1}+\frac{1}{\d^{\frac{3}{2}}\e}\nm{\bar u}_{L^2([0,t]\times\Omega\times\s^1)}.
\end{eqnarray}
After a similar but tedious computation, we can show
\begin{eqnarray}
\abs{I_{2,3}}\leq \delta\im{u}{[0,t]\times\Omega\times\s^1}+\frac{1}{\d^{\frac{3}{2}}\e}\nm{\bar u}_{L^2([0,t]\times\Omega\times\s^1)}.
\end{eqnarray}
Hence, we have proved
\begin{eqnarray}
\abs{I_{2}}&\leq& \delta\im{u}{[0,t]\times\Omega\times\s^1}+\frac{1}{\d^{\frac{3}{2}}\e}\nm{\bar u}_{L^2([0,t]\times\Omega\times\s^1)}\\
&&+\im{\ss}{[0,t]\times\Omega\times\s^1}+\im{\g}{[0,t]\times\Gamma^-}+\im{h}{\Omega\times\s^1}.\no
\end{eqnarray}
In a similar fashion, we can show
\begin{eqnarray}
\abs{II_{2}}&\leq& \delta\im{u}{[0,t]\times\Omega\times\s^1}+\frac{1}{\d^{\frac{3}{2}}\e}\nm{\bar u}_{L^2([0,t]\times\Omega\times\s^1)}\\
&&+\im{\ss}{[0,t]\times\Omega\times\s^1}+\im{\g}{[0,t]\times\Gamma^-}+\im{h}{\Omega\times\s^1}.\no
\end{eqnarray}
\ \\
Step 4: Synthesis.\\
Summarizing all above, we have shown
\begin{eqnarray}
\abs{u}&\leq& \delta\im{u}{[0,t]\times\Omega\times\s^1}+\frac{1}{\d^{\frac{3}{2}}\e}\nm{\bar u}_{L^2([0,t]\times\Omega\times\s^1)}\\
&&+\im{\ss}{[0,t]\times\Omega\times\s^1}+\im{\g}{[0,t]\times\Gamma^-}+\im{\h}{\Omega\times\s^1}.\no
\end{eqnarray}
Since $(t,\vx,\vw)$ are arbitrary and $\d$ is small, we have
\begin{eqnarray}
\\
\im{u}{[0,t]\times\Omega\times\s^1}&\leq& C\bigg(\frac{1}{\e}\nm{\bar u}_{L^2([0,t]\times\Omega\times\s^1)}+\im{\ss}{[0,t]\times\Omega\times\s^1}+\im{\g}{[0,t]\times\Gamma^-}+\im{\h}{\Omega\times\s^1}\bigg).\no
\end{eqnarray}
Then using Theorem \ref{LT estimate}, we get the desired result.
\end{proof}

\subsection{$L^{2m}$ Estimate}

In this section, we try to improve previous estimates. In the following, we assume $m\geq2$ is an integer and let $o(1)$ denote a sufficiently small constant.
\begin{lemma}\label{LN estimate}
Assume $\ss(t,\vx,\vw)\in
L^{\infty}([0,\infty)\times\Omega\times\s^1)$, $\h(\vx,\vw)\in
L^{\infty}(\Omega\times\s^1)$ and $\g(t,x_0,\vw)\in
L^{\infty}([0,\infty)\times\Gamma^-)$. Then the solution $u(t,\vx,\vw)$ to the neutron transport
equation (\ref{neutron}) satisfies
\begin{eqnarray}
\\
&&\nm{u(t)}_{L^2(\Omega\times\s^1)}+\frac{1}{\e^{\frac{1}{2}}}\nm{(1-\pp)[u]}_{L^2([0,\infty)\times\Gamma^+)}+\nm{\bar
u}_{L^2([0,\infty)\times\Omega\times\s^1)}+\frac{1}{\e}\nm{u-\bar
u}_{L^2([0,\infty)\times\Omega\times\s^1)}\no\\
&\leq&
C \bigg(o(1)\e^{\frac{1}{m}}\nm{u}_{L^{\infty}([0,\infty)\times\Omega\times\s^1)}\no\\
&&+\frac{1}{\e^{\frac{3}{2}-\frac{5m-2}{2m(2m-1)}}}\tm{\ss}{[0,\infty)\times\Omega\times\s^1}
+\frac{1}{\e^{\frac{5}{2}-\frac{5m-2}{2m(2m-1)}}}\nm{\ss}_{L^{\frac{2m}{2m-1}}([0,\infty)\times\Omega\times\s^1)}\no\\
&&+\frac{1}{\e^{\frac{1}{2}-\frac{5m-2}{2m(2m-1)}}}\nm{\h}_{L^2(\Omega\times\s^1)}+\e^{\frac{1}{2m-1}}\nm{h}_{L^{2m}(\Omega\times\s^1)}\no\\
&&+\frac{1}{\e^{\frac{3}{2}-\frac{5m-2}{2m(2m-1)}}}\nm{\g}_{L^2([0,\infty)\times\Gamma^-)}+\nm{g}_{L^m([0,\infty)\times\Gamma^+)}\bigg).\nonumber
\end{eqnarray}
\end{lemma}
\begin{proof}
We divide the proof into several steps:\\
\ \\
Step 1: Kernel Estimate.\\
Applying Lemma \ref{well-posedness lemma 3} to the
equation (\ref{neutron}). Then for any
$\phi\in L^{2}([0,t)\times\Omega\times\s^1)$ satisfying
$\e\dt\phi+\vw\cdot\nx\phi\in L^2([0,t)\times\Omega\times\s^1)$
and $\phi\in L^{2}([0,t)\times\Gamma)$, we have
\begin{eqnarray}\label{well-posedness temp 4.}
&&
-\e^2\int_0^t\iint_{\Omega\times\s^1}\dt\phi
u-\e\int_0^t\iint_{\Omega\times\s^1}(\vw\cdot\nx\phi)u+\int_0^t\iint_{\Omega\times\s^1}(u-\bar
u)\phi\\
&=&-\e\int_0^t\int_{\Gamma}u\phi\ud{\gamma}-\e^2\iint_{\Omega\times\s^1}u(t)\phi(t)+\e^2\iint_{\Omega\times\s^1}u(0)\phi(0)+\int_0^t\iint_{\Omega\times\s^1}\ss\phi.\no
\end{eqnarray}
Our goal is to choose a particular test function $\phi$. We first
construct an auxiliary function $\zeta(t)$. Since $u(t)\in
L^{\infty}(\Omega\times\s^1)$, it naturally implies $\bar
u(t)\in L^{\infty}(\Omega)$. We define $\zeta(t,\vx)$ on $\Omega$
satisfying
\begin{eqnarray}\label{test temp 1.}
\left\{
\begin{array}{rcl}
\Delta \zeta(t)&=&(\bar u)^{2m-1}(t,\vx)\ud{\vx}\ \ \text{in}\ \
\Omega,\\\rule{0ex}{1.0em} \zeta(t)&=&0\ \ \text{on}\ \ \p\Omega.
\end{array}
\right.
\end{eqnarray}
In the bounded domain $\Omega$, based on the standard elliptic
estimate, we have
\begin{eqnarray}\label{test temp 3.}
\nm{\zeta}_{W^{2,\frac{2m}{2m-1}}(\Omega)}\leq C\nm{(\bar
u)^{2m-1}}_{L^{\frac{2m}{2m-1}}(\Omega)}= C\nm{\bar
u}_{L^{2m}(\Omega)}^{2m-1}.
\end{eqnarray}
We plug the test function
\begin{eqnarray}\label{test temp 2.}
\phi=-\vw\cdot\nx\zeta
\end{eqnarray}
into the weak formulation (\ref{well-posedness temp 4.}) and estimate
each term there. By Sobolev embedding theorem, we have
\begin{eqnarray}
\nm{\phi}_{L^2(\Omega)}&\leq& C\nm{\zeta}_{H^1(\Omega)}\leq C\nm{\zeta}_{W^{2,\frac{2m}{2m-1}}(\Omega)}\leq
C\nm{\bar
u}_{L^{2m}(\Omega)}^{2m-1},\label{test temp 6.}\\
\nm{\phi}_{L^{\frac{2m}{2m-1}}(\Omega)}&\leq&C\nm{\zeta}_{W^{1,\frac{2m}{2m-1}}(\Omega)}\leq
C\nm{\bar
u}_{L^{2m}(\Omega)}^{2m-1}.\label{test temp 4.}
\end{eqnarray}
Easily we can decompose
\begin{eqnarray}\label{test temp 5.}
-\e\int_0^t\iint_{\Omega\times\s^1}(\vw\cdot\nx\phi)u&=&-\e\int_0^t\iint_{\Omega\times\s^1}(\vw\cdot\nx\phi)\bar
u-\e\int_0^t\iint_{\Omega\times\s^1}(\vw\cdot\nx\phi)(u-\bar
u).
\end{eqnarray}
We estimate the two term on the right-hand side of (\ref{test temp
5.}) separately. By (\ref{test temp 1.}) and (\ref{test temp 2.}), we
have
\begin{eqnarray}\label{wellposed temp 1.}
-\e\int_0^t\iint_{\Omega\times\s^1}(\vw\cdot\nx\phi)\bar
u&=&\e\int_0^t\iint_{\Omega\times\s^1}\bar
u\bigg(w_1(w_1\p_{11}\zeta+w_2\p_{12}\zeta)+w_2(w_1\p_{12}\zeta+w_2\p_{22}\zeta)\bigg)\\
&=&\e\int_0^t\iint_{\Omega\times\s^1}\bar
u\bigg(w_1^2\p_{11}\zeta+w_2^2\p_{22}\zeta\bigg)\nonumber\\
&=&\e\pi\int_0^t\int_{\Omega}\bar u(\p_{11}\zeta+\p_{22}\zeta)\nonumber\\
&=&\e\pi\nm{\bar u}_{L^{2m}([0,t]\times\Omega)}^{2m}.\nonumber
\end{eqnarray}
In the second equality, above cross terms vanish due to the symmetry
of the integral over $\s^1$. On the other hand, for the second term
in (\ref{test temp 5}), H\"older's inequality and the elliptic
estimate imply
\begin{eqnarray}\label{wellposed temp 2.}
-\e\int_0^t\iint_{\Omega\times\s^1}(\vw\cdot\nx\phi)(u-\bar
u)&\leq&C \e\nm{u-\bar u}_{L^{2m}([0,t]\times\Omega\times\s^1)}\bigg(\int_0^t\nm{\zeta}^2_{W^{2,\frac{2m}{2m-1}}(\Omega)}\bigg)^{1/2}\\
&\leq&C \e\nm{u-\bar
u}_{L^{2m}([0,t]\times\Omega\times\s^1)}\nm{\bar
u}_{L^{2m}([0,t]\times\Omega\times\s^1)}^{2m-1}\nonumber.
\end{eqnarray}
Based on (\ref{test temp 3.}), (\ref{test temp 6.}), (\ref{test temp 4.}), Sobolev embedding theorem and the trace theorem, we have
\begin{eqnarray}
\\
\nm{\nx\zeta}_{L^{\frac{m}{m-1}}(\Gamma)}\leq C\nm{\nx\zeta}_{W^{\frac{1}{2m},\frac{2m}{2m-1}}(\Gamma)}\leq C\nm{\nx\zeta}_{W^{1,\frac{2m}{2m-1}}(\Omega)}\leq C\nm{\zeta}_{W^{2,\frac{2m}{2m-1}}(\Omega)}\leq
C\nm{\bar
u}_{L^{2m}(\Omega)}^{2m-1}.\no
\end{eqnarray}
Based on (\ref{test temp 3.}), (\ref{test temp 4.}), the trace theorem and
H\"older's inequality, we have
\begin{eqnarray}\label{wellposed temp 3.}
\e\int_0^t\int_{\Gamma}u\phi\ud{\gamma}&=&\e\int_0^t\int_{\Gamma}\pp[u]\phi\ud{\gamma}+\e\int_0^t\int_{\Gamma^+}(1-\pp)[u]\phi\ud{\gamma}
+\e\int_0^t\int_{\Gamma^-}\g\phi\ud{\gamma}\\
&=&\e\int_0^t\int_{\Gamma^+}(1-\pp)[u]\phi\ud{\gamma}
+\e\int_0^t\int_{\Gamma^-}\g\phi\ud{\gamma}\no\\
&\leq&\e\nm{\phi}_{L^{\frac{m}{m-1}}(\Gamma)}\bigg(\nm{(1-\pp)[u]}_{L^m([0,t]\times\Gamma^+)}+\nm{\g}_{L^m([0,t]\times\Gamma^-)}\bigg)\no\\
&\leq&\e\nm{\bar u}_{L^{2m}([0,t]\times\Omega\times\s^1)}^{2m-1}\bigg(\nm{(1-\pp)[u]}_{L^m([0,t]\times\Gamma^+)}+\nm{\g}_{L^m([0,t]\times\Gamma^-)}\bigg).\no
\end{eqnarray}
Also, we obtain
\begin{eqnarray}\label{wellposed temp 5.}
\int_0^t\iint_{\Omega\times\s^1}(u-\bar u)\phi&\leq&\nm{\phi}_{L^{2}([0,t]\times\Omega\times\s^1)}\nm{u-\bar
u}_{L^2([0,t]\times\Omega\times\s^1)}\\
&\leq&
C \nm{\bar
u}_{L^{2m}([0,t]\times\Omega\times\s^1)}^{2m-1}\nm{u-\bar
u}_{L^2([0,t]\times\Omega\times\s^1)},\no
\end{eqnarray}
\begin{eqnarray}\label{wellposed temp 6.}
\int_0^t\iint_{\Omega\times\s^1}\ss\phi\leq C &\leq&\nm{\phi}_{L^{2}([0,t]\times\Omega\times\s^1)}\nm{\ss}_{L^2([0,t]\times\Omega\times\s^1)}\\
&\leq&
C \nm{\bar
u}_{L^{2m}([0,t]\times\Omega\times\s^1)}^{2m-1}\nm{\ss}_{L^2([0,t]\times\Omega\times\s^1)}.\no
\end{eqnarray}
On the other hand, we may apply H\"{o}lder's inequality and Young's inequality to directly estimate
\begin{eqnarray}\label{wellposed temp 7.}
\e^2\iint_{\Omega\times\s^1}u(t)\phi(t)&\leq&C\e^2
\nm{\phi(t)
}_{L^{2}(\Omega\times\s^1)}\nm{u(t)}_{L^{2}(\Omega\times\s^1)}\\
&\leq&C\e^2\nm{\bar u(t)
}^{2m-1}_{L^{2m}(\Omega\times\s^1)}\nm{u(t)}_{L^{2}(\Omega\times\s^1)}\leq C\bigg(\e^{\frac{4m-1}{2m-1}}\nm{u(t)}^{2m}_{L^{2m}(\Omega\times\s^1)}+\e\nm{u(t)}^{2m}_{L^{2}(\Omega\times\s^1)}\bigg).\no
\end{eqnarray}
Similarly, we have
\begin{eqnarray}
\e^2\iint_{\Omega\times\s^1}u(0)\phi(0)&\leq& C\bigg(\e^{\frac{4m-1}{2m-1}}\nm{h}^{2m}_{L^{2m}(\Omega\times\s^1)}+\e\nm{h}^{2m}_{L^{2}(\Omega\times\s^1)}\bigg).
\end{eqnarray}
Then the only remaining term is
\begin{eqnarray}\label{wellposed temp 8.}
-\e^2\int_0^t\iint_{\Omega\times\s^1}\dt\phi u&=&-\e^2\int_0^t\iint_{\Omega\times\s^1}\dt\phi (u-\bar u)\\
&\leq&\nm{\dt\nabla\zeta}_{L^{2}([0,t]\times\Omega\times\s^1)}\nm{u-\bar
u}_{L^2([0,t]\times\Omega\times\s^1)}.\no
\end{eqnarray}
Now we have to tackle $\nm{\dt\nabla\zeta}_{L^{2}([0,t]\times\Omega\times\s^1)}$.
Similar to the $L^2$ estimate, for test function $\phi(\vx,\vw)$ which is independent of time $t$, we have the weak formulation
\begin{eqnarray}\label{test temp 7.}
\e^2\iint_{\Omega\times\s^1}\dt u(t)\phi
&=&\e\iint_{\Omega\times\s^1}(\vw\cdot\nx\phi)u(t)-\iint_{\Omega\times\s^1}(u(t)-\bar
u(t))\phi\\
&&-\e\int_{\Gamma}u(t)\phi\ud{\gamma}+\iint_{\Omega\times\s^1}f(t)\phi.\no
\end{eqnarray}
For fixed $t$, taking $\phi=-\Phi(\vx)$ which satisfies
\begin{eqnarray}
\left\{
\begin{array}{rcl}
\Delta \Phi&=&\dt\bar u(t,\vx)\Big(\bar u(t,\vx)\Big)^{2m-2}\ \ \text{in}\ \
\Omega,\\\rule{0ex}{1.0em} \Phi(t)&=&0\ \ \text{on}\ \ \p\Omega,
\end{array}
\right.
\end{eqnarray}
which further implies $\Phi=\dt\zeta$.
Then the left-hand side of (\ref{test temp 7.}) is actually
\begin{eqnarray}
LHS&=&-\e^2\iint_{\Omega\times\s^1}\Phi\dt u(t)=-\e^2\iint_{\Omega\times\s^1}\Phi\dt\bar u\\
&=&-\e^2\iint_{\Omega\times\s^1}\Phi\Delta\Phi=\e^2\iint_{\Omega\times\s^1}\abs{\nabla\Phi}^2\no\\
&=&\e^2\nm{\dt\nabla\zeta(t)}_{L^{2}(\Omega\times\s^1)}^{2}.\no
\end{eqnarray}
By a similar argument as above and the Poincar\'e inequality, the right-hand side of
(\ref{test temp 7.}) can be bounded as
\begin{eqnarray}
RHS&\leq& \nm{\dt\nabla\zeta(t)}_{L^{2}(\Omega\times\s^1)}\bigg(\nm{u(t)-\bar
u(t)}_{L^{2}(\Omega\times\s^1)}
+\nm{\ss(t)}_{L^2(\Omega\times\s^1)}\bigg).
\end{eqnarray}
Note that the boundary terms vanish due to the construction of $\Phi$.
Therefore, we have
\begin{eqnarray}
\e^2\nm{\dt\nabla\zeta(t)}_{L^{2}(\Omega\times\s^1)}&\leq&
\nm{u(t)-\bar u(t)}_{L^{2}(\Omega\times\s^1)}
+\nm{\ss(t)}_{L^2(\Omega\times\s^1)}.
\end{eqnarray}
For all $t$, we can further integrate over $[0,t]$ to obtain
\begin{eqnarray}\label{wellposed temp 9.}
\e^2\nm{\dt\nabla\zeta}_{L^{2}([0,t]\times\Omega\times\s^1)}
&\leq& \nm{u-\bar
u}_{L^{2}([0,t]\times\Omega\times\s^1)}
+\nm{\ss}_{L^2([0,t]\times\Omega\times\s^1)}.
\end{eqnarray}
Collecting all terms above and using Young's inequality, we
obtain
\begin{eqnarray}\label{well-posedness temp 3.}
\e\nm{\bar u}_{L^{2m}([0,t]\times\Omega\times\s^1)}
&\leq& C \bigg(\e\nm{u-\bar
u}_{L^{2m}([0,t]\times\Omega\times\s^1)}+\e\nm{(1-\pp)[u]}_{L^m([0,t]\times\Gamma^+)}\\
&&+\nm{\ss}_{L^2([0,t]\times\Omega\times\s^1)}+\e\nm{\g}_{L^m([0,t]\times\Gamma^-)}\no\\
&&+\e^{\frac{2m}{2m-1}}\nm{u(t)}_{L^{2m}(\Omega\times\s^1)}+\e\nm{u(t)}_{L^{2}(\Omega\times\s^1)}
+\e^{\frac{2m}{2m-1}}\nm{h}_{L^{2m}(\Omega\times\s^1)}+\e\nm{h}_{L^{2}(\Omega\times\s^1)}\bigg).\no
\end{eqnarray}
\ \\
Step 2: Energy Estimate.\\
As before, we get energy estimate
\begin{eqnarray}\label{well-posedness temp 5.}
&&\frac{\e^2}{2}\nm{u(t)}_{L^2(\Omega\times\s^1)}^2+\frac{\e}{2}\nm{(1-\pp)[u]}^2_{L^2([0,t)\times\Gamma^+)}+\nm{u-\bar
u}_{L^2([0,t)\times\Omega\times\s^1)}^2\\
&\leq&\iint_{[0,t)\times\Omega\times\s^1}\ss u+\frac{\e^2}{2}\nm{\h}_{L^2(\Omega\times\s^1)}^2+\nm{\g}_{L^2([0,t)\times\Gamma^-)}^2
+\e^2\eta\nm{\bar u}_{L^2([0,t)\times\Omega\times\s^1)}^2.\no
\end{eqnarray}
On the other hand, we can square on both sides of
(\ref{well-posedness temp 3.}) to obtain
\begin{eqnarray}\label{well-posedness temp 6.}
&&\e^2\nm{\bar u}^2_{L^{2m}([0,t]\times\Omega\times\s^1)}\\
&\leq& C \bigg(\e^2\nm{u-\bar
u}^2_{L^{2m}([0,t]\times\Omega\times\s^1)}+\e^2\nm{(1-\pp)[u]}^2_{L^m([0,t]\times\Gamma^+)}\no\\
&&+\nm{\ss}^2_{L^2([0,t]\times\Omega\times\s^1)}+\e^2\nm{\g}^2_{L^m([0,t]\times\Gamma^-)}\no\\
&&+\e^{\frac{4m}{2m-1}}\nm{u(t)}^{2}_{L^{2m}(\Omega\times\s^1)}+\e^2\nm{u(t)}^{2}_{L^{2}(\Omega\times\s^1)}
+\e^{\frac{4m}{2m-1}}\nm{h}^{2}_{L^{2m}(\Omega\times\s^1)}+\e^2\nm{h}^{2}_{L^{2}(\Omega\times\s^1)}\bigg).\no
\end{eqnarray}
Taking $\eta$ sufficiently small, multiplying a sufficiently small constant on both sides of
(\ref{well-posedness temp 6.}) and adding it to (\ref{well-posedness temp 5.}) to absorb $\e^2\eta\nm{\bar u}_{L^2([0,t)\times\Omega\times\s^1)}^2$ and $\e^2\nm{u(t)}^{2}_{L^{2}(\Omega\times\s^1)}$, we deduce
\begin{eqnarray}\label{wt 03.}
&&\e^2\nm{u(t)}_{L^2(\Omega\times\s^1)}^2+\e\nm{(1-\pp)[u]}_{L^2([0,t]\times\Gamma^+)}^2+\e^2\nm{\bar
u}_{L^2([0,t]\times\Omega\times\s^1)}^2+\nm{u-\bar
u}_{L^2([0,t]\times\Omega\times\s^1)}^2\\
&\leq&
C \bigg(\int_0^t\iint_{\Omega\times\s^1}\ss u+\e^2\nm{u-\bar
u}^2_{L^{2m}([0,t]\times\Omega\times\s^1)}+\e^2\nm{(1-\pp)[u]}^2_{L^m([0,t]\times\Gamma^+)}\no\\
&&+\nm{\ss}^2_{L^2([0,t]\times\Omega\times\s^1)}+\e^2\nm{\g}^2_{L^m([0,t]\times\Gamma^-)}+\nm{\g}_{L^2([0,t)\times\Gamma^-)}^2\no\\
&&+\e^{\frac{4m}{2m-1}}\nm{u(t)}^2_{L^{2m}(\Omega\times\s^1)}+\e^2\nm{h}^2_{L^{2}(\Omega\times\s^1)}+\e^{\frac{4m}{2m-1}}\nm{h}^2_{L^{2m}(\Omega\times\s^1)}\bigg).\no
\end{eqnarray}
By interpolation estimate and Young's inequality, we have
\begin{eqnarray}
\nm{(1-\pp)[u]}_{L^{m}([0,t]\times\Gamma^+)}&\leq&\nm{(1-\pp)[u]}_{L^2([0,t]\times\Gamma^+)}^{\frac{2}{m}}\nm{(1-\pp)[u]}_{L^{\infty}([0,t]\times\Gamma^+)}^{\frac{m-2}{m}}\\
&=&\bigg(\frac{1}{\e^{\frac{m-2}{m^2}}}\nm{(1-\pp)[u]}_{L^2([0,t]\times\Gamma^+)}^{\frac{2}{m}}\bigg)
\bigg(\e^{\frac{m-2}{m^2}}\nm{(1-\pp)[u]}_{L^{\infty}([0,t]\times\Gamma^+)}^{\frac{m-2}{m}}\bigg)\no\\
&\leq&C\bigg(\frac{1}{\e^{\frac{m-2}{m^2}}}\nm{(1-\pp)[u]}_{L^2([0,t]\times\Gamma^+)}^{\frac{2}{m}}\bigg)^{\frac{m}{2}}+o(1)
\bigg(\e^{\frac{m-2}{m^2}}\nm{(1-\pp)[u]}_{L^{\infty}([0,t]\times\Gamma^+)}^{\frac{m-2}{m}}\bigg)^{\frac{m}{m-2}}\no\\
&\leq&\frac{C}{\e^{\frac{m-2}{2m}}}\nm{(1-\pp)[u]}_{L^2([0,t]\times\Gamma^+)}+o(1)\e^{\frac{1}{m}}\nm{(1-\pp)[u]}_{L^{\infty}([0,t]\times\Gamma^+)}\no\\
&\leq&\frac{C}{\e^{\frac{m-2}{2m}}}\nm{(1-\pp)[u]}_{L^2([0,t]\times\Gamma^+)}+o(1)\e^{\frac{1}{m}}\nm{u}_{L^{\infty}([0,t]\times\Omega\times\s^1)}.\no
\end{eqnarray}
Similarly, we have
\begin{eqnarray}
\nm{u-\bar u}_{L^{2m}([0,t]\times\Omega\times\s^1)}&\leq&\nm{u-\bar u}_{L^2([0,t]\times\Omega\times\s^1)}^{\frac{1}{m}}\nm{u-\bar u}_{L^{\infty}([0,t]\times\Omega\times\s^1)}^{\frac{m-1}{m}}\\
&=&\bigg(\frac{1}{\e^{\frac{m-1}{m^2}}}\nm{u-\bar u}_{L^2([0,t]\times\Omega\times\s^1)}^{\frac{1}{m}}\bigg)\bigg(\e^{\frac{m-1}{m^2}}\nm{u-\bar u}_{L^{\infty}([0,t]\times\Omega\times\s^1)}^{\frac{m-1}{m}}\bigg)\no\\
&\leq&C\bigg(\frac{1}{\e^{\frac{m-1}{m^2}}}\nm{u-\bar u}_{L^2([0,t]\times\Omega\times\s^1)}^{\frac{1}{m}}\bigg)^{m}+o(1)\bigg(\e^{\frac{m-1}{m^2}}\nm{u-\bar u}_{L^{\infty}([0,t]\times\Omega\times\s^1)}^{\frac{m-1}{m}}\bigg)^{\frac{m}{m-1}}\no\\
&\leq&\frac{C}{\e^{\frac{m-1}{m}}}\nm{u-\bar u}_{L^2([0,t]\times\Omega\times\s^1)}+o(1)\e^{\frac{1}{m}}\nm{u-\bar u}_{L^{\infty}([0,t]\times\Omega\times\s^1)}.\no
\end{eqnarray}
Also, we know
\begin{eqnarray}
\nm{u(t)}_{L^{2m}(\Omega\times\s^1)}&\leq&\nm{u(t)}_{L^2(\Omega\times\s^1)}^{\frac{1}{m}}\nm{u(t)}_{L^{\infty}(\Omega\times\s^1)}^{\frac{m-1}{m}}\\
&=&\bigg(\frac{1}{\e^{\frac{(m-1)^2}{m^2(2m-1)}}}\nm{u(t)}_{L^2(\Omega\times\s^1)}^{\frac{1}{m}}\bigg)
\bigg(\e^{\frac{(m-1)^2}{m^2(2m-1)}}\nm{u(t)}_{L^{\infty}(\Omega\times\s^1)}^{\frac{m-1}{m}}\bigg)\no\\
&\leq&C\bigg(\frac{1}{\e^{\frac{(m-1)^2}{m^2(2m-1)}}}\nm{u(t)}_{L^2(\Omega\times\s^1)}^{\frac{1}{m}}\bigg)^{m}
+o(1)\bigg(\e^{\frac{(m-1)^2}{m^2(2m-1)}}\nm{u(t)}_{L^{\infty}(\Omega\times\s^1)}^{\frac{m-1}{m}}\bigg)^{\frac{m}{m-1}}\no\\
&\leq&\frac{C}{\e^{\frac{(m-1)^2}{m(2m-1)}}}\nm{u(t)}_{L^2(\Omega\times\s^1)}+o(1)\e^{\frac{m-1}{m(2m-1)}}\nm{u(t)}_{L^{\infty}(\Omega\times\s^1)}.\no
\end{eqnarray}
We need this extra $\e^{\frac{1}{m}}$ for the convenience of $L^{\infty}$ estimate.
Then we know for sufficiently small $\e$,
\begin{eqnarray}
\e^2\nm{(1-\pp)[u]}_{L^{m}([0,t]\times\Gamma^+)}^2
&\leq&C\e^{2-\frac{m-2}{m}}\nm{(1-\pp)[u]}_{L^2([0,t]\times\Gamma^+)}^2+o(1)\e^{2+\frac{2}{m}}\nm{u}_{L^{\infty}([0,t]\times\Gamma^+)}^2\\
&\leq&o(1)\e\nm{(1-\pp)[u]}_{L^2([0,t]\times\Gamma^+)}^2+o(1)\e^{2+\frac{2}{m}}\nm{u}_{L^{\infty}([0,t]\times\Gamma^+)}^2.\no
\end{eqnarray}
Similarly, we have
\begin{eqnarray}
\e^2\nm{u-\bar
u}_{L^{2m}([0,t]\times\Omega\times\s^1)}^2&\leq&\e^{2-\frac{2m-2}{m}}\nm{u-\bar u}_{L^2([0,t]\times\Omega\times\s^1)}^2+o(1)\e^{2+\frac{2}{m}}\nm{u}_{L^{\infty}([0,t]\times\Omega\times\s^1)}^2\\
&\leq& o(1)\nm{u-\bar u}_{L^2([0,t]\times\Omega\times\s^1)}^2+o(1)\e^{2+\frac{2}{m}}\nm{u}_{L^{\infty}([0,t]\times\Omega\times\s^1)}^2.\no
\end{eqnarray}
Also, we know
\begin{eqnarray}
\e^{\frac{4m}{2m-1}}\nm{u(t)}^2_{L^{2m}(\Omega\times\s^1)}&\leq& \e^{\frac{4m}{2m-1}}\bigg(\frac{C}{\e^{\frac{2(m-1)^2}{m(2m-1)}}}\nm{u(t)}^2_{L^2(\Omega\times\s^1)}+o(1)\e^{\frac{2(m-1)}{m(2m-1)}}\nm{u(t)}^2_{L^{\infty}(\Omega\times\s^1)}\bigg)\\
&=&C\e^{1+\frac{5m-2}{m(2m-1)}}\nm{u(t)}^2_{L^2(\Omega\times\s^1)}+o(1)\e^{1+\frac{2}{m}}\nm{u(t)}^2_{L^{\infty}(\Omega\times\s^1)}\no\\
&\leq&C\e^{1+\frac{5m-2}{m(2m-1)}}\nm{u(t)}^2_{L^2(\Omega\times\s^1)}+o(1)\e^{1+\frac{2}{m}}\nm{u}^2_{L^{\infty}([0,\infty)\times\Omega\times\s^1)}.\no
\end{eqnarray}
In (\ref{wt 03.}), we can absorb $\nm{u-\bar u}_{L^2(\Omega\times\s^1)}$ and $\e\nm{(1-\pp)[u]}_{L^2(\Gamma^+)}^2$ into left-hand side to obtain
\begin{eqnarray}\label{wt 04.}
&&\e^2\nm{u(t)}_{L^2(\Omega\times\s^1)}^2+\e\nm{(1-\pp)[u]}_{L^2([0,t]\times\Gamma^+)}^2+\e^2\nm{\bar
u}_{L^2([0,t]\times\Omega\times\s^1)}^2+\nm{u-\bar
u}_{L^2([0,t]\times\Omega\times\s^1)}^2\\
&\leq&
C \bigg(o(1)\e^{1+\frac{2}{m}}\nm{u}^2_{L^{\infty}([0,\infty)\times\Omega\times\s^1)}+\int_0^t\iint_{\Omega\times\s^1}\ss u+\nm{\ss}^2_{L^2([0,t]\times\Omega\times\s^1)}\no\\
&&+\e^2\nm{\g}^2_{L^m([0,t]\times\Gamma^-)}+\nm{\g}_{L^2([0,t)\times\Gamma^-)}^2\no\\
&&+\e^{1+\frac{5m-2}{m(2m-1)}}\nm{u(t)}^2_{L^2(\Omega\times\s^1)}+\e^2\nm{h}^2_{L^{2}(\Omega\times\s^1)}+\e^{\frac{4m}{2m-1}}\nm{h}^2_{L^{2m}(\Omega\times\s^1)}\bigg).\no
\end{eqnarray}
Note that we cannot further absorb $\e^{1+\frac{5m-2}{m(2m-1)}}\nm{u(t)}^2_{L^2(\Omega\times\s^1)}$ into the left-hand side since its power of $\e$ is insufficient. This is the critical value and makes the whole estimate worse.

Based on (\ref{well-posedness temp 7}), we have
\begin{eqnarray}
\\
\nm{u(t)}_{L^2(\Omega\times\s^1)}^2
&\leq& C \bigg(\frac{1}{\e^2}\tm{\ss}{[0,t]\times\Omega\times\s^1}^2+
\frac{1}{\e^2}\int_0^t\iint_{\Omega\times\s^1}\ss u_{\l}+\nm{\h}_{L^2(\Omega\times\s^1)}^2+\frac{1}{\e^2}\nm{\g}_{L^2([0,t]\times\Gamma^-)}^2\bigg).\no
\end{eqnarray}
Plugging it into the right-hand side of (\ref{wt 04.}), we can conclude that
\begin{eqnarray}\label{wt 04.}
&&\e^2\nm{u(t)}_{L^2(\Omega\times\s^1)}^2+\e\nm{(1-\pp)[u]}_{L^2([0,t]\times\Gamma^+)}^2+\e^2\nm{\bar
u}_{L^2([0,t]\times\Omega\times\s^1)}^2+\nm{u-\bar
u}_{L^2([0,t]\times\Omega\times\s^1)}^2\\
&\leq&
C \bigg(o(1)\e^{1+\frac{2}{m}}\nm{u}^2_{L^{\infty}([0,\infty)\times\Omega\times\s^1)}+\e^{-1+\frac{5m-2}{m(2m-1)}}\int_0^t\iint_{\Omega\times\s^1}\ss u+\e^{-1+\frac{5m-2}{m(2m-1)}}\nm{\ss}^2_{L^2([0,t]\times\Omega\times\s^1)}\no\\
&&+\e^2\nm{\g}^2_{L^m([0,t]\times\Gamma^-)}+\e^{-1+\frac{5m-2}{m(2m-1)}}\nm{\g}^2_{L^2([0,t]\times\Gamma^-)}\no\\
&&+\e^{1+\frac{5m-2}{m(2m-1)}}\nm{h}^2_{L^{2}(\Omega\times\s^1)}+\e^{\frac{4m}{2m-1}}\nm{h}^2_{L^{2m}(\Omega\times\s^1)}\bigg).\no
\end{eqnarray}
We can decompose
\begin{eqnarray}
\iint_{\Omega\times\s^1}fu=\iint_{\Omega\times\s^1}f\bar u+\iint_{\Omega\times\s^1}f(u-\bar u).
\end{eqnarray}
H\"older's inequality and Cauchy's inequality imply
\begin{eqnarray}
\e^{-1+\frac{5m-2}{m(2m-1)}}\iint_{\Omega\times\s^1}f\bar u&\leq&\e^{-1+\frac{5m-2}{m(2m-1)}}\nm{f}_{L^{\frac{2m}{2m-1}}(\Omega\times\s^1)}\nm{\bar u}_{L^{2m}(\Omega\times\s^1)}\\
&\leq&\frac{C}{\e^{3-\frac{5m-2}{m(2m-1)}}}\nm{f}_{L^{\frac{2m}{2m-1}}(\Omega\times\s^1)}^2+o(1)\e^2\nm{\bar u}_{L^{2m}(\Omega\times\s^1)}^2,\no
\end{eqnarray}
and
\begin{eqnarray}
\e^{-1+\frac{5m-2}{m(2m-1)}}\iint_{\Omega\times\s^1}f(u-\bar u)&\leq& C\frac{1}{\e^{1-\frac{5m-2}{m(2m-1)}}}\nm{f}_{L^{2}(\Omega\times\s^1)}^2+o(1)\nm{u-\bar u}_{L^2(\Omega\times\s^1)}^2.
\end{eqnarray}
Hence, absorbing $\e^2\nm{\bar u}_{L^{2m}(\Omega\times\s^1)}^2$ and $\nm{u-\bar u}_{L^2(\Omega\times\s^1)}^2$ into left-hand side of (\ref{wt 04.}), we get
\begin{eqnarray}\label{wt 06.}
&&\e^2\nm{u(t)}_{L^2(\Omega\times\s^1)}^2+\e\nm{(1-\pp)[u]}_{L^2([0,t]\times\Gamma^+)}^2+\e^2\nm{\bar
u}_{L^2([0,t]\times\Omega\times\s^1)}^2+\nm{u-\bar
u}_{L^2([0,t]\times\Omega\times\s^1)}^2\\
&\leq&
C \bigg(o(1)\e^{2+\frac{2}{m}}\nm{u}_{L^{\infty}([0,t]\times\Omega\times\s^1)}^2+\frac{1}{\e^{1-\frac{5m-2}{m(2m-1)}}}\tm{\ss}{[0,t]\times\Omega\times\s^1}^2
+\frac{1}{\e^{3-\frac{5m-2}{m(2m-1)}}}\nm{\ss}_{L^{\frac{2m}{2m-1}}([0,t]\times\Omega\times\s^1)}^2\no\\
&&+\e^{1+\frac{5m-2}{m(2m-1)}}\nm{\h}_{L^2(\Omega\times\s^1)}^2+\e^{\frac{4m}{2m-1}}\nm{h}^2_{L^{2m}(\Omega\times\s^1)}\no\\
&&+\frac{1}{\e^{1-\frac{5m-2}{m(2m-1)}}}\nm{\g}_{L^2([0,t]\times\Gamma^-)}^2+\e^2\nm{g}_{L^m([0,t]\times\Gamma^+)}^2\bigg).\nonumber
\end{eqnarray}
which implies
\begin{eqnarray}\label{wt 07.}
&&\nm{u(t)}_{L^2(\Omega\times\s^1)}+\frac{1}{\e^{\frac{1}{2}}}\nm{(1-\pp)[u]}_{L^2([0,t]\times\Gamma^+)}+\nm{\bar
u}_{L^2([0,t]\times\Omega\times\s^1)}+\frac{1}{\e}\nm{u-\bar
u}_{L^2([0,t]\times\Omega\times\s^1)}\\
&\leq&
C \bigg(o(1)\e^{\frac{1}{m}}\nm{u}_{L^{\infty}([0,t]\times\Omega\times\s^1)}+\frac{1}{\e^{\frac{3}{2}-\frac{5m-2}{2m(2m-1)}}}\tm{\ss}{[0,t]\times\Omega\times\s^1}
+\frac{1}{\e^{\frac{5}{2}-\frac{5m-2}{2m(2m-1)}}}\nm{\ss}_{L^{\frac{2m}{2m-1}}([0,t]\times\Omega\times\s^1)}\no\\
&&+\frac{1}{\e^{\frac{1}{2}-\frac{5m-2}{2m(2m-1)}}}\nm{\h}_{L^2(\Omega\times\s^1)}+\e^{\frac{1}{2m-1}}\nm{h}_{L^{2m}(\Omega\times\s^1)}\no\\
&&+\frac{1}{\e^{\frac{3}{2}-\frac{5m-2}{2m(2m-1)}}}\nm{\g}_{L^2([0,t]\times\Gamma^-)}+\nm{g}_{L^m([0,t]\times\Gamma^+)}\bigg).\nonumber
\end{eqnarray}

\end{proof}

\subsection{$L^{\infty}$ Estimate - Second Round}

\begin{theorem}\label{LI estimate}
Assume $\ss(t,\vx,\vw)\in
L^{\infty}([0,\infty)\times\Omega\times\s^1)$, $\h(\vx,\vw)\in
L^{\infty}(\Omega\times\s^1)$ and $\g(t,x_0,\vw)\in
L^{\infty}([0,\infty)\times\Gamma^-)$. Then the solution $u(t,\vx,\vw)$ to the neutron transport
equation (\ref{neutron}) satisfies
\begin{eqnarray}
&&\im{u}{[0,\infty)\times\Omega\times\s^1}\\
&\leq& C \bigg(\frac{1}{\e^{\frac{5}{2}-\frac{9m-4}{2m(2m-1)}}}\nm{\ss}_{L^{\frac{2m}{2m-1}}([0,\infty)\times\Omega\times\s^1)}
+\frac{1}{\e^{\frac{3}{2}-\frac{9m-4}{2m(2m-1)}}}\tm{\ss}{[0,\infty)\times\Omega\times\s^1}
+\im{\ss}{[0,\infty)\times\Omega\times\s^1}\no\\
&&+\frac{1}{\e^{\frac{m-1}{2m-1}}}\nm{h}_{L^{2m}(\Omega\times\s^1)}+\frac{1}{\e^{\frac{1}{2}-\frac{9m-4}{2m(2m-1)}}}\nm{\h}_{L^2(\Omega\times\s^1)}+\im{\h}{\Omega\times\s^1}\no\\
&&+\frac{1}{\e^{\frac{1}{m}}}\nm{g}_{L^m([0,\infty)\times\Gamma^+)}+\frac{1}{\e^{\frac{3}{2}-\frac{9m-4}{2m(2m-1)}}}\nm{\g}_{L^2([0,\infty)\times\Gamma^-)}
+\im{\g}{[0,\infty)\times\Gamma^-}\bigg).\nonumber
\end{eqnarray}
\end{theorem}
\begin{proof}
Based on the analysis in proving Theorem \ref{LI estimate.}, the key step is the estimate of $I_{2,2,2}$. Here, we use H\"{o}lder's inequality with a different exponent to obtain
\begin{eqnarray}
I_{2,2,2}&\leq&C\bigg(\int_{0}^{t}\int_{\Omega}\frac{1}{\e^2\delta^3}\abs{\bar u^{2m}(s,\vec y)}\ud{\vec y}\ud{s}\bigg)^{\frac{1}{2m}}.
\end{eqnarray}
Therefore, we have shown
\begin{eqnarray}
\abs{I_{2,2}}\leq \delta\im{u}{[0,t]\times\Omega\times\s^1}+\frac{1}{\d^{\frac{3}{2m}}\e^{\frac{1}{m}}}\nm{\bar u}_{L^{2m}([0,t]\times\Omega\times\s^1)}.
\end{eqnarray}
In a similar fashion, we can show that
\begin{eqnarray}
\abs{u}&\leq& \delta\im{u}{[0,t]\times\Omega\times\s^1}+\frac{1}{\d^{\frac{3}{2m}}\e^{\frac{1}{m}}}\nm{\bar u}_{L^{2m}([0,t]\times\Omega\times\s^1)}\\
&&+\im{\ss}{[0,\infty)\times\Omega\times\s^1}+\im{\g}{[0,\infty)\times\Gamma^-}+\im{\h}{\Omega\times\s^1}.\no
\end{eqnarray}
This naturally implies that
\begin{eqnarray}
\im{u}{[0,t]\times\Omega\times\s^1}&\leq& \frac{1}{\d^{\frac{3}{2m}}\e^{\frac{1}{m}}}\nm{\bar u}_{L^{2m}([0,t]\times\Omega\times\s^1)}\\
&&+\im{\ss}{[0,t]\times\Omega\times\s^1}+\im{\g}{[0,t]\times\Gamma^-}+\im{\h}{\Omega\times\s^1}.\no
\end{eqnarray}
Then using Theorem \ref{LN estimate}, we get
\begin{eqnarray}
&&\im{u}{[0,t]\times\Omega\times\s^1}\\
&\leq& C \bigg(o(1)\nm{u}_{L^{\infty}([0,t]\times\Omega\times\s^1)}\no\\
&&+\frac{1}{\e^{\frac{5}{2}-\frac{9m-4}{2m(2m-1)}}}\nm{\ss}_{L^{\frac{2m}{2m-1}}([0,t]\times\Omega\times\s^1)}
+\frac{1}{\e^{\frac{3}{2}-\frac{9m-4}{2m(2m-1)}}}\tm{\ss}{[0,t]\times\Omega\times\s^1}
+\im{\ss}{[0,t]\times\Omega\times\s^1}\no\\
&&+\frac{1}{\e^{\frac{m-1}{2m-1}}}\nm{h}_{L^{2m}(\Omega\times\s^1)}+\frac{1}{\e^{\frac{1}{2}-\frac{9m-4}{2m(2m-1)}}}\nm{\h}_{L^2(\Omega\times\s^1)}+\im{\h}{\Omega\times\s^1}\no\\
&&+\frac{1}{\e^{\frac{1}{m}}}\nm{g}_{L^m([0,t]\times\Gamma^+)}+\frac{1}{\e^{\frac{3}{2}-\frac{9m-4}{2m(2m-1)}}}\nm{\g}_{L^2([0,t]\times\Gamma^-)}
+\im{\g}{[0,t]\times\Gamma^-}\bigg).\nonumber
\end{eqnarray}
Absorbing $o(1)$ term into the left-hand side, we obtain the desired result.
\end{proof}

\begin{theorem}\label{LI estimate..}
Assume $\ue^{K_0t}\ss(t,\vx,\vw)\in
L^{\infty}([0,\infty)\times\Omega\times\s^1)$, $\h(\vx,\vw)\in
L^{\infty}(\Omega\times\s^1)$ and $\ue^{K_0 t}\g(t,x_0,\vw)\in
L^{\infty}([0,\infty)\times\Gamma^-)$ for some $K_0>0$. Then there exists $0<K\leq K_0$ such that the solution
$u(t,\vx,\vw)$ to the neutron transport equation (\ref{neutron}) satisfies
\begin{eqnarray}
&&\im{\ue^{Kt}u}{[0,\infty)\times\Omega\times\s^1}\\
&\leq& C \bigg(\frac{1}{\e^{\frac{5}{2}-\frac{9m-4}{2m(2m-1)}}}\nm{\ue^{Kt}\ss}_{L^{\frac{2m}{2m-1}}([0,\infty)\times\Omega\times\s^1)}
+\frac{1}{\e^{\frac{3}{2}-\frac{9m-4}{2m(2m-1)}}}\tm{\ue^{Kt}\ss}{[0,\infty)\times\Omega\times\s^1}
+\im{\ue^{Kt}\ss}{[0,\infty)\times\Omega\times\s^1}\no\\
&&+\frac{1}{\e^{\frac{m-1}{2m-1}}}\nm{h}_{L^{2m}(\Omega\times\s^1)}+\frac{1}{\e^{\frac{1}{2}-\frac{9m-4}{2m(2m-1)}}}\nm{\h}_{L^2(\Omega\times\s^1)}+\im{\h}{\Omega\times\s^1}\no\\
&&+\frac{1}{\e^{\frac{1}{m}}}\nm{\ue^{Kt}g}_{L^m([0,\infty)\times\Gamma^+)}+\frac{1}{\e^{\frac{3}{2}-\frac{9m-4}{2m(2m-1)}}}\nm{\ue^{Kt}g}_{L^2([0,\infty)\times\Gamma^-)}
+\im{\ue^{Kt}g}{[0,\infty)\times\Gamma^-}\bigg).\nonumber
\end{eqnarray}
\end{theorem}
\begin{proof}
Let $v=\ue^{K t}u$. Then $v$ satisfies the equation
\begin{eqnarray}
\left\{
\begin{array}{l}
\e^2\dt v+\e\vw\cdot\nabla_xv+v-\bar v=\ue^{Kt}\ss+K\e^2v\ \ \text{for}\ \
(t,\vx,\vw)\in[0,\infty)\times\Omega\times\s^1,\\\rule{0ex}{1.0em} v(0,\vx,\vw)=\h(\vx,\vw)\ \ \text{for}\ \
(\vx,\vw)\in\Omega\times\s^1,\\\rule{0ex}{1.0em} v(t,\vx_0,\vw)-\pp[v](t,\vx_0)=\ue^{Kt}\g(t,\vx_0,\vw)\
\ \text{for}\ \ t\in[0,\infty)\ \ \vx_0\in\p\Omega\ \ \text{and}\ \vw\cdot\vn<0.
\end{array}
\right.
\end{eqnarray}
Note that we have an extra term $K\e^2v$. However, $\e^2$ helps to recover all the estimates in previous theorems and we can obtain exactly the same results.
\end{proof}


\bibliographystyle{siam}
\bibliography{Reference}

\end{document}